\documentclass[a4paper, leqno, 10pt]{amsart}

\usepackage{amssymb, amsmath, amsthm, amsrefs}
\usepackage{hyperref}
\usepackage{mathrsfs}
\usepackage[matha, mathx]{mathabx}
\usepackage[shortlabels]{enumitem}
\usepackage{ esint }
\usepackage[colorinlistoftodos]{todonotes}
\usepackage{xcolor}
\numberwithin{equation}{section}
\newtheorem{thm}{Theorem}
\newtheorem{theo}[thm]{Theorem}
\newtheorem{corollary}[thm]{Corollary}
\newtheorem{lemma}[thm]{Lemma}
\newtheorem{proposition}[thm]{Proposition}

\newtheorem{rmk}[thm]{Remark}
\newtheorem{define}[thm]{Definition} 
\newcommand{\dd}{\mathrm{d}}
\newcommand{\T}{\mathbb{T}}
\newcommand{\R}{\mathbb{R}}
\newcommand{\norm}[1]{\left\Vert#1\right\Vert}
\newcommand{\brkt}[1]{\left(#1\right)}

\newcommand{\abs}[1]{\left|#1\right|}
\newcommand{\set}[1]{\left\{#1\right\}}
\newcommand{\jap}[1]{\left<#1\right>}

\newcommand{\D}{\mathbb{D}}
\newcommand{\C}{\mathbb{C}}
\newcommand{\Z}{\mathbb{Z}}
\newcommand{\N}{\mathbb{N}}
\newcommand{\I}{\mathrm{I}}
\newcommand{\II}{\mathrm{II}}
\newcommand{\dbtilde}[1]{\tilde{\tilde{#1}}}
\newcommand{\eps}{\varepsilon}

\author[Bakas]{Odysseas Bakas}
\address{BCAM - Basque Center for Applied Mathematics, 48009 Bilbao, Spain}
\email{obakas@bcamath.org}
\author[Pott]{Sandra Pott}
\address{Centre for Mathematical Sciences, Lund University, 221 00 Lund, Sweden}
\email{sandra.pott@math.lu.se}
\author[Rodr\'iguez-L\'opez]{Salvador Rodr\'iguez-L\'opez}
\address{Department of Mathematics, Stockholm University, 106 91 Stockholm, Sweden}
\email{s.rodriguez-lopez@math.su.se}
\author[Sola]{Alan Sola}
\address{Department of Mathematics, Stockholm University, 106 91 Stockholm, Sweden}
\email{sola@math.su.se}

\thanks{The first author was partially supported by the grant KAW 2017.0425, financed by the Knut and Alice Wallenberg Foundation and by the projects CEX2021-001142-S, RYC2018-025477-I, PID2021-122156NB-I00/AEI/10.13039/501100011033 funded by Agencia Estatal de Investigaci\'on and acronym ``HAMIP'', Juan de la Cierva Incorporaci\'on IJC2020-043082-I, and grant BERC 2022-2025 of the Basque Government. The third author was partially supported by the Spanish Government grant PID2020-113048GB-I00. The fourth author was partially supported by the National Science Foundation under DMS grant \# 1928930 while he participated in a programme hosted by MSRI (Berkeley, CA) during the spring 2022 semester.}


\let\iffalseHIDE\iffalse

\begin{document}

\title[Multipliers for Hardy--Orlicz spaces and applications]{Multipliers for Hardy--Orlicz spaces and applications}

\maketitle

\begin{abstract} 
Using real-variable methods, we characterise multipliers for general classes of Hardy--Orlicz spaces, %
unifying and extending several classical results due to Hardy and Littlewood; Duren and Shields; Paley; and others.
Applications of our results include inequalities involving Fourier coefficients and Fourier transforms of elements of Hardy--Orlicz spaces and their duals, as well as embeddings into spaces of generalised smoothness, Sobolev type-embeddings and Paley-Wiener type theorems. 
\end{abstract}

\section{Introduction}
Many, if not most, classical spaces of functions that arise in harmonic analysis and operator theory admit membership conditions that can be expressed in terms of growth conditions on %
Fourier coefficients or Fourier transforms of elements in the space, or, when the space consists of analytic functions, Taylor coefficients. Many operators acting on such spaces %
can either be defined in terms of their actions on the Fourier or Taylor side, or admit useful descriptions in these terms. This leads to the study of multiplier operators, or in brief, multipliers, and their mapping properties. The simplest case is that of $L^2(\T)$, the square integrable periodic functions: a function is in $L^2(\T)$ if, and only if, its Fourier coefficients are in the sequence space $\ell^2$, and so a multiplier is bounded if, and only if, it maps $\ell^2$ to $\ell^2$. Similarly, there are many classical results characterising when a multiplier is bounded from a given Banach space, or scale of Banach spaces, to $\ell^q$ for some suitable range of $q$. When dealing with less classical spaces, such as endpoint substitutes for $L^1$ spaces, one is naturally led to consider function spaces not of Banach type, and multiplier operators acting on such spaces. 

This paper is devoted to a study of multipliers in the setting of Hardy--Orlicz spaces, which generalise $H^p$ spaces on the circle or on Euclidean space, and their duals. Our goal is to both unify and generalise well-known results in the $H^p$ theory, and to exploit such results to give new descriptions of the elements of these spaces and how they embed into other naturally occurring spaces in complex and harmonic analysis.

Let us set the scene for this paper by letting $X(\D)$ be a space of holomorphic functions on the unit disc $\D := \{ z \in \C: |z| < 1 \}$, equipped with a norm, or quasi-norm, $\| \cdot \|_{X(\D)}$ so that $\norm{F}_{X(\D)} <\infty$ if, and only if, $F \in X(\D)$. For instance, $X(\D)$ could be a classical Hardy space $H^p (\D)$ with $p\geq 1$.

For $q \in (0, \infty]$, we say that a sequence $\lambda= \{ \lambda_n \}_{n=0}^{\infty}$ of complex numbers is a multiplier from $X(\D)$ to $\ell^q (\N_0)$ if there exists a constant $C_{\lambda, X} > 0$ such that
\begin{equation}\label{mult_def}
\norm{\{ \lambda_n f_n \}_{n=0}^{\infty}}_{\ell^q (\N_0)} \leq C_{\lambda, X} \norm{ F }_{X(\D)}  
\end{equation}
for all $F(z)= \sum_{n = 0}^{\infty} f_n z^n $ in $X(\D)$. The class of all multipliers from $X(\D)$ to $\ell^q (\N_0)$ is denoted by $\mathcal{M}_{X(\D) \rightarrow \ell^q (\N_0)}$. 
For $\lambda \in \mathcal{M}_{X(\D) \rightarrow \ell^q (\N_0)}$, we write
\[
\norm{\lambda}_{ \mathcal{M}_{X(\D) \rightarrow \ell^q (\N_0)}} := \inf \left\{ C_{\lambda, X} >0 : \eqref{mult_def} \text{ holds for all } F \in X(\D) \right\}.   
\]

In this paper we establish characterisations of multipliers from a broad class of Hardy--Orlicz spaces to $\ell^q$ for $q \in [1,2]$. Our characterisations include in a unified way known results for classical Hardy spaces as well as new multiplier results for Hardy--Orlicz spaces that arise in the study of products between Hardy spaces and their duals. %
Of particular interest to us is the space $H^{\log} (\D)$ that naturally arises in the study of products between functions in $H^1 (\D)$ and $BMO (\D)$ (see \cite{BIJZ}) and whose dual space $LMO (\D) \cong ( H^{\log} (\D) )^{\ast}$ appears in a number of problems in complex and harmonic analysis, including that of the characterisation of boundedness of Hankel operators on $H^1 (\D)$ (see \cite{JPS}). 

\subsection*{Multiplier theorems} Let us now recall some well-known multiplier theorems on classical Hardy spaces that our present work generalises.
If $p \in (0,1)$ then it was shown by P. L. Duren and A. L. Shields in \cites{DS_69,DS_70} that for $q \in [p, \infty)$ a sequence $\lambda = \{ \lambda_n \}_{n=0}^{\infty}$ is a multiplier from $H^p(\D)$ to $\ell^q (\N_0)$ if, and only if,  
\begin{equation}\label{Hp_lq}
\sup_{N \in \N} \frac{1}{N} \left( \sum_{n=1}^N n^{q/p} |\lambda_n|^q \right)^{1/q} < \infty.   
\end{equation}

In the case $p=1$ it was shown by G. H. Hardy and J. E. Littlewood \cites{H-L_37,H-L_41} and E. M. Stein and A. Zygmund \cite{SZ} that a sequence $\lambda = \{ \lambda_n \}_{n=0}^{\infty}$ is a multiplier from $H^1(\D)$ to $\ell^2 (\N_0)$ if, and only if,  
\begin{equation}\label{H1_l2}
\sup_{m \in \N} \sum_{n=2^m}^{2^{m+1}-1} |\lambda_n|^2 < \infty.
\end{equation}
For an introduction to classical Hardy spaces and for proofs of the aforementioned multiplier results, we refer the reader to Duren's book \cite{D}. 

In the $p=q=1$ case, a theorem attributed to C. Fefferman (see \cite{Dy} and \cite{Sl-St}) asserts that a sequence $\lambda = \{ \lambda_n \}_{n=0}^{\infty}$ is a multiplier from $H^1(\D)$ to $\ell^1 (\N_0)$ if, and only if,   
\begin{equation}\label{H1_l1}
\sup_{N \in \N} \sum_{k=1}^\infty \left(\sum_{n= k N}^{(k+1)N -1} |\lambda_n| \right)^2  < \infty.
\end{equation}

In this paper we extend the aforementioned results to a broad class of Hardy--Orlicz spaces. To state our results, recall that given a growth function $\Psi$ (see Definition \ref{main_def} below), the corresponding Hardy--Orlicz space $H^{\Psi} (\D)$ is defined as the class of functions $F$ that are holomorphic in $\D$ and satisfy
\[
\sup_{0 \leq r <1} \int_{[0,1)} \Psi \left( \abs{F \left(r e^{i 2\pi \theta}\right)} \right) \dd \theta < \infty.
\]
If $G \in H^{\Psi} (\D)$, we set
\[
\norm{G}_{H^{\Psi} (\D)} := \inf \left\{ \alpha >0 : \sup_{0 \leq r <1} \int_{[0,1)} \Psi \left(\alpha^{-1} \abs{G \left(r e^{i 2\pi \theta} \right)} \right) \dd \theta \leq 1 \right\}.
\]
Note that if $\Psi (t)=t^p$ then $( H^{\Psi} (\D), \| \cdot \|_{H^{\Psi} (\D)})$ coincides with the usual Hardy space $( H^p (\D), \| \cdot \|_{H^p (\D)})$. 

In what follows, if $\Psi$ is a growth function then $\Psi^{-1}$ denotes its inverse, i.e. $\Psi^{-1}: [0, \infty) \rightarrow [0, \infty)$ is such that  $\Psi^{-1} ( \Psi (t)) = t $ for all $t \geq 0$.

Our first result is a characterisation of the class $\mathcal{M}_{H^{\Psi} (\D) \rightarrow \ell^q (\N_0)}$ for $q \in [1, \infty)$ and for growth functions of order $p \in (1/2,1)$. 

\begin{thm}\label{gen_mult_p<1}
Let $\Psi$ be a growth function of order $p \in (1/2,1)$ and let $q \in [1, \infty)$. 

Let $\lambda = \{ \lambda_n \}_{n=0}^{\infty}$ be a sequence of complex numbers. The following are equivalent:
\begin{enumerate} 
\item  $\lambda$ is a multiplier from $H^{\Psi}(\D)$ to $\ell^q (\N_0)  $;
\item $\lambda$ satisfies the condition
\[
A_{\lambda, \Psi, p, q} := \sup_{N \in \N} \frac{1}{N} \left( \sum_{n=1}^N {\Psi^{-1}(n)}^q  |\lambda_n|^q \right)^{1/q} < \infty ;
\]
\item $\lambda$ satisfies the condition
\[
\widetilde{A}_{\lambda, \Psi, p,  q} := \sup_{m \in \N_0} \frac{\Psi^{-1}(2^m)} {2^m} \left( \sum_{n=2^m}^{2^{m+1}-1} |\lambda_n|^q \right)^{1/q} < \infty .
\]
\end{enumerate}
If $\lambda \in \mathcal{M}_{H^{\Psi}(\D) \rightarrow \ell^q (\N_0)}$ then
\[
\norm{ \lambda }_{\mathcal{M}_{H^{\Psi}(\D) \rightarrow \ell^q (\N_0)}} \approx A_{\lambda, \Psi, p, q} \approx \widetilde{A}_{\lambda, \Psi, p, q} ,
\]
where the implied constants depend only on $\Psi$, $p$, $q$ and not on $\lambda$. 
\end{thm}

Theorem \ref{gen_mult_p<1} includes Hardy--Orlicz spaces with growth functions of the form
\begin{equation}\label{Psi_r,p_def}
\Psi_{r,p} (t) := \frac{t^p}{\log^r (t+e)}, \quad t \geq 0, 
\end{equation}
for $p \in (1/2,1) $ and $r \in [0, \infty) $. In the Hardy space case, i.e. for $r=0$; $H^{\Psi_{0,p}} (\D) = H^p (\D)$, Theorem \ref{gen_mult_p<1} recovers the characterisation of Duren and Shields for the class $\mathcal{M}_{H^p(\D) \rightarrow \ell^q (\N_0)}$ when $p \in (1/2,1)$ and $q \in [1, \infty)$. 

Our next result extends Theorem \ref{gen_mult_p<1} to growth functions of order $p=1$ in the case where $q \in [1,2]$. More specifically, the following result holds true.

\begin{thm}\label{gen_mult_q_leq2}
Let $\Psi$ be a growth function of order $p \in (1/2,1]$.

Let $\lambda = \{ \lambda_n \}_{n=0}^{\infty}$ be a sequence of complex numbers and let $q \in [1,2]$.
\begin{enumerate}[(a)]
\item  If $q \in [1,2)$,  then the following are equivalent:
\begin{enumerate}[(i)]
\item $\lambda$  is a multiplier from $H^{\Psi}(\D)$ to $\ell^q (\N_0)$;
\item $\lambda$ satisfies the condition
\[
C_{\lambda, \Psi, p, q} := \sup_{N \in \N} \frac{ \Psi^{-1}(N)}{N} 
\left(\sum_{k=1}^\infty \left(\sum_{n= k N}^{(k+1)N -1} |\lambda_n|^q \right)^\frac{2}{(2-q)q} \right)^{(2-q)/2} < \infty.
\]
\end{enumerate}
If $\lambda \in \mathcal{M}_{H^{\Psi}(\D) \rightarrow \ell^q (\N_0)}$ then
\[
\norm{ \lambda }_{\mathcal{M}_{H^{\Psi}(\D) \rightarrow \ell^q (\N_0)}} \approx C_{\lambda, \Psi, p, q} ,
\]
where the implied constants depend only on $\Psi$, $p$, $q$ and not on $\lambda$. 
\item  If $q=2$, the following are equivalent:
\begin{enumerate}[(i)]
\item $\lambda$ is a multiplier from $H^{\Psi}(\D)$ to $\ell^2 (\N_0)$;
\item $\lambda$ satisfies the condition
\[
B_{\lambda, \Psi, p, 2}:= \sup_{N \in \N} \frac{ \Psi^{-1}(N)}{N} \sup_{k \in \N} \left(\sum_{n= k N}^{(k+1)N -1} |\lambda_n|^2 \right)^{1/2}  < \infty;
\]
\item $\lambda$ satisfies the condition
\[
\widetilde{B}_{\lambda, \Psi, p, 2} := \sup_{m \in \N} \frac{ \Psi^{-1}(2^m)  }{2^m } \left(\sum_{n=2^m}^{2^{m+1} -1} |\lambda_n|^2 \right)^{1/2} < \infty.
\]
\end{enumerate}
If $\lambda \in \mathcal{M}_{H^{\Psi}(\D) \rightarrow \ell^2 (\N_0)}$ then
\[
\norm{ \lambda }_{\mathcal{M}_{H^{\Psi}(\D) \rightarrow \ell^2 (\N_0)}} \approx B_{\lambda, \Psi, p, 2} \approx \widetilde{B}_{\lambda, \Psi, p, 2},
\]
where the implied constants depend only on $\Psi$ and $p$, but not on $\lambda$.
\end{enumerate}
\end{thm}

Theorem \ref{gen_mult_q_leq2} includes Hardy--Orlicz spaces with $\Psi_{r,p}$ as in \eqref{Psi_r,p_def}, with $r \in [0, \infty) $ and $p \in (1/2,1]$, endpoint included. Note that if $r=0$ and $p=1$, then $H^{\Psi_{0,1}} (\D)= H^1 (\D)$, and condition (3) in part  (b) of Theorem \ref{gen_mult_q_leq2} becomes \eqref{H1_l2} that is, Theorem \ref{gen_mult_q_leq2} recovers the aforementioned characterisation of Hardy--Littlewood and Stein--Zygmund for the class $\mathcal{M}_{H^1(\D) \rightarrow \ell^2 (\N_0)}$. 

Note that if $q \in [1,2)$ and $\Psi$ is a growth function of order $p \in (1/2,1)$, then conditions (2) and (3) in Theorem \ref{gen_mult_p<1} and condition (ii) in part (a) of Theorem \ref{gen_mult_q_leq2} are all equivalent. Similarly, for $q=2$, conditions (2) and (3) in Theorem \ref{gen_mult_p<1} and conditions (ii) and (iii) in part (b) of Theorem \ref{gen_mult_q_leq2} are all equivalent.

To prove Theorems \ref{gen_mult_p<1} and \ref{gen_mult_q_leq2} we use a real-variable approach that is based on duality. An important ingredient in our proofs is an unconditionality result; see Lemma \ref{lem:uncond1} below. 

\subsection*{Hardy--Littlewood and Paley-type inequalities}

If $\Psi$ is a growth function of order $p \in (1/2,1]$ then the sequence $\lambda= \{ \lambda_n \}_{n=0}^{\infty} $ given by
\[
\lambda_n := 
\begin{cases}
1/\Psi^{-1} (n), \quad & \text{if } n \in \N; \\
0, \quad & \text{if } n = 0
\end{cases}
\]
satisfies condition (2) in Theorem \ref{gen_mult_q_leq2}. We thus obtain the following corollary.

\begin{corollary}\label{HL_complex}
Let $\Psi$ be a growth function of order $p \in (1/2,1]$. 
Then there exists a  constant $M_{\Psi, p} >0 $ such that
\[
\sum_{n=1}^{\infty} \frac{ \abs{f_n}}{ \Psi^{-1} (n) } \leq M_{\Psi, p} \norm{ F }_{H^{\Psi} (\D)} 
\]
for all $F(z) = \sum_{n \geq 0} f_n z^n$ in $H^{\Psi} (\D)$. 
\end{corollary}

Corollary \ref{HL_complex} implies a classical inequality due to Hardy and Littlewood \cite{H-L}:
\begin{equation}\label{H-L_p=1}
\sum_{n=1}^{\infty} \frac{ \abs{f_n}}{ n} \lesssim \norm{ F }_{H^1 (\D)} 
\end{equation}
as well as its `$H^{\log}$'-variant
\begin{equation}\label{H-L_Hlog}
\sum_{n=1}^{\infty} \frac{ \abs{f_n}}{ n \log (n+1)} \lesssim \norm{ F }_{H^{\log} (\D)} 
\end{equation}
that the present authors had obtained in \cite{BPRS}.

In the $q=2$ case, Theorem \ref{gen_mult_q_leq2} has the following corollary which may be of independent interest.

\begin{corollary}\label{cor_Paley_ineq_Psi}
Let $\Psi$  be a growth function of order $p \in (1/2,1]$. 
Then there exists a constant $B_{\Psi} >0$ such that
\[
\left( \sum_{j = 0}^{\infty} \left( \frac {2^j} {\Psi^{-1}(2^j)} \right)^2 |f_{2^j}|^2 \right)^{1/2} \leq B_{\Psi, p} \norm{ F }_{H^{\Psi} (\D)}.
\]
for all functions $F(z)=\sum_{n=0}^{\infty} f_n z^n$ in $H^{\Psi} (\D)$.
\end{corollary}

Let us examine what Corollary \ref{cor_Paley_ineq_Psi} implies in some special cases. If  $\Psi_{r,1} $ is as in \eqref{Psi_r,p_def}, then its inverse function satisfies
\begin{equation}\label{inverse_r}
\Psi^{-1}_{r,1} (t) \approx_r  t \log^r (t+e), \quad t \geq 0,
\end{equation}
where the implied constants depend only on $r$. Hence, when $r=0$, then Corollary \ref{cor_Paley_ineq_Psi} recovers a classical inequality due to R. E. A. C. Paley \cite{P}:
\begin{equation}\label{Paley_ineq}
\left( \sum_{j = 0}^{\infty} |f_{2^j}|^2 \right)^{1/2} \lesssim \norm{ F }_{H^1 (\D)} 
\end{equation}
for all functions $F(z)=\sum_{n=0}^{\infty} f_n z^n$ in $H^1 (\D)$. The choice $r=1$ corresponds to the Hardy--Orlicz space
\[
H^{\log} (\D) := H^{\Psi_{1,1}} (\D) .\]
In view of \eqref{inverse_r}, Corollary \ref{cor_Paley_ineq_Psi} implies that
\begin{equation}\label{Paley_ineq_Hlog}
\left(\sum_{j = 0}^{\infty} \frac{|f_{2^j}|^2}{(j+1)^{2}} \right)^{1/2} \lesssim \norm{ F }_{H^{\log} (\D)}  
\end{equation}
for all functions $F(z)=\sum_{n=0}^{\infty} f_n z^n$ in $H^{\log}(\D)$, which can be regarded as a natural variant of Paley's inequality \eqref{Paley_ineq} for functions in $H^{\log} (\D)$.

\subsection*{Dual spaces}
It follows from the work of S. Janson \cite{J} that the spaces $H^{\Psi} (\D)$ and $BMO(\rho_{\Psi})(\D)$ can be put in duality. Recall that if \[\rho_{\Psi} (t):=t^{-1} /\Psi^{-1} (t^{-1}), \quad t > 0,\] then $BMO(\rho_{\Psi})(\D)$ is the space of holomorphic functions on $\D$ whose boundary values belong to 
\[ BMO(\rho_{\Psi})(\T):=  
\left\{ u \in L^2 (\T) : \sup_{\substack{I \subset \T:\\ I \text{ arc }}} \frac{1} { \left[ \rho_{\Psi} (|I|) \right]^2 |I|} \int_I \abs{ u (e^{2\pi i \theta}) - \langle u \rangle_I }^2 \dd \theta < \infty \right\}.
\]
Here, if $g \in L^1 (\T) $ and $J \subset \T$ is an arc, $\langle g \rangle_J$ denotes the average of $g$ over $J$. The class of functions on $\T$ that are boundary values of functions in $BMO(\rho_{\Psi})(\D)$ is denoted by $BMOA (\rho_{\Psi})(\T)$. %

If $\Psi_{0,1} (t)=t$, then one recovers the usual BMO-type spaces, whereas the case $\Psi_{1,1} (t)=t /\log(t+e)$ corresponds to LMO-type spaces. %

Note that by appealing to Corollary \ref{cor_Paley_ineq_Psi} and duality, it follows that if $\{ a_j \}_{j=0}^{\infty}$ is a sequence of complex numbers with
\begin{equation}\label{lac_BMO_Psi_cond1}
\sum_{j=0}^{\infty} \left( \frac{\Psi^{-1} (2^j) } {2^j} \right)^2  |a_j|^2  < \infty
\end{equation}
then the lacunary trigonometric series
\[
\sum_{j = 0}^{\infty} a_j e^{i 2\pi 2^j \theta}
\]
is the Fourier series of a function in $BMOA (\rho_{\Psi})(\T)$. In the case where $\Psi (t) = \Phi_{0,1} (t) = t$, condition \eqref{lac_BMO_Psi_cond1} characterises lacunary Fourier series in $BMOA (\T)$; see \S 6.3 (iii) in Chapter IV in \cite{Big_Stein}.  
 One might ask whether condition \eqref{lac_BMO_Psi_cond1} actually characterises lacunary Fourier series in $BMOA (\rho_{\Psi})(\T)$ in general. As we will see below, it {\it does not}.
 
\subsection*{Littlewood--Paley characterisations}
If the growth function is of the form $\Psi= \Psi_{r,1}$ (i.e.  as in \eqref{Psi_r,p_def} for $p=1$), implication $(3)\implies (1)$ in part (b) of Theorem \ref{gen_mult_q_leq2} can also be deduced from a Littlewood--Paley characterisation of $BMO (\rho_{\Psi_{r,1}}) (\T)$ that is of independent interest. 

To state this characterisation of $BMO(\rho_{\Psi_{r,1}})(\T)$, we need to introduce some notation first. For $n \in \N_0$, define
\[ 
\delta_n  :=
\begin{cases}
V_1  , \quad & \text{if } n=0;  \\
  V_{2^n} - V_{2^{n-1}} , \quad & \text{if } n \in \N,
\end{cases}
\]
where $V_m$ denotes the de la Vall\'ee Poussin kernel of order $m$; see \S \ref{LP_part} below. For $f \in L^1 (\T)$ and $n \in \N_0$, define the corresponding Littlewood--Paley projection by $ \Delta_n (f ) := \delta_n \ast f $.

\begin{theo}\label{LP_BMO_Psi} Let $r \geq 0$ be a given exponent and let $\Psi_{r,1}$ be as in \eqref{Psi_r,p_def} (for $p=1$).  

If $u $ is a function in $ L^2 (\T)$, then the following are equivalent:
\begin{enumerate}
\item $u$ belongs to $BMO (\rho_{\Psi_{r,1}}) (\T)$;
\item one has
\begin{equation}\label{L-P_BMO_Psi_cond}
\sup_{\substack{I \subseteq \T: \\ I \ {\rm arc}}} \frac{1} { \left[ \rho_{\Psi_{r,1}} (|I|) \right]^2 |I|} \int_I \sum_{\substack{n \in \N_0: \\ 2^n > |I|^{-1}}} \left| \Delta_n (u) (e^{i 2 \pi \theta}) \right|^2 \, \dd \theta < \infty.
\end{equation}
\end{enumerate}
Moreover, one has 
\begin{equation}\label{equiv_norm_L-P_BMO_Psi}
\norm{u}_{BMO(\rho_{\Psi_{r,1}})^+ (\T)} \approx_{r} \\ \abs{ \int_{[0,1)} u (e^{i 2 \pi \theta}) \dd \theta} + \norm{u}_{\Psi_{r,1}, \ast} ,
\end{equation}
where
\[ 
\norm{u}_{\Psi_{r,1}, \ast}:= \left( \sup_{\substack{I \subseteq \T: \\ I \ {\rm arc}}} \frac{1} { \left[ \rho_{\Psi_{r,1}} (|I|) \right]^2 |I|} \int_I \sum_{\substack{n \in \N_0: \\ 2^n > |I|^{-1}}}  \left| \Delta_n (u) (e^{i 2 \pi \theta}) \right|^2 \, \dd \theta \right)^{1/2}.
\]
\end{theo}

Theorem \ref{LP_BMO_Psi} generalises a well-known Littlewood--Paley characterisation of functions in $BMO (\T)$ due to M. Frazier and B. Jawerth \cite{FJ} and appears to be, to the best of our knowledge, new. As we will see below, if $u$ belongs to $BMO (\rho_{\Psi}) (\T)$ then \eqref{L-P_BMO_Psi_cond} holds %
for more general Littlewood--Paley-type partitions; see  Theorem \ref{reverse_arbitrary}  
for a precise statement of this fact. For the $BMO$-case this is implicit in the works of J. Bourgain \cite{B_SF} and J. L. Rubio de Francia \cite{RdF}, see also M. T. Lacey  \cite{L}.

Moreover, Theorem \ref{LP_BMO_Psi} can be to used to improve condition \eqref{lac_BMO_Psi_cond1}. More specifically, Theorem \ref{LP_BMO_Psi} implies the following characterisation of lacunary Fourier series in $BMOA (\rho_{\Psi}) (\T)$ in terms of the growth of their Fourier coefficients.

\begin{corollary}\label{BMOA_Psi_char_lac}
Let $r \geq 0$ be a given exponent.

Let $\{ a_j \}_{j=0}^{\infty} $ be a sequence of complex numbers. The following are equivalent:
\begin{enumerate}
\item The lacunary trigonometric series 
\[ 
\sum_{j=0}^{\infty} a_j e^{i 2\pi 2^j \theta}, \quad \theta \in [0,1),
\]
is the Fourier series of a function in $BMOA (\rho_{\Psi_{r,1}}) (\T)$;
\item one has
\begin{equation}\label{lac_BMO_Psi_cond_char}
\sup_{N \in \N} \left\{ \left( \frac{\Psi_{r,1}^{-1}(N)}{N}\right)^2 \sum_{j \geq \log N} |a_j|^2 \right\} < \infty.
\end{equation}
\end{enumerate}
\end{corollary}

Note that if $\Psi$ is not the identity map, then \eqref{lac_BMO_Psi_cond1} is  strictly stronger than  \eqref{lac_BMO_Psi_cond_char}. 

To illustrate Corollary \ref{BMOA_Psi_char_lac}, consider the case of $LMOA (\T)$ i.e. take $\Psi (t) = \Psi_{1,1} (t) = t/\log(t+e)$. Then the lacunary trigonometric series 
\[ 
\sum_{j=0}^{\infty} a_j e^{i 2\pi 2^j \theta}, \quad \theta \in [0,1),
\]
with $\{ a_j \}_{j=0}^{\infty} $ being a sequence of complex numbers, 
is the Fourier series of a function in $LMOA (\T)$ if, and only if,
\[ 
\sup_{N \in \N} \left\{ \log^2 N \sum_{j \geq \log N} |a_j|^2 \right\} < \infty.
\]
For instance, the function $u$ whose Fourier series is given by
 \begin{equation}\label{LMOA_lac_ex}
u (e^{i 2\pi \theta}) \sim \sum_{j=0}^{\infty} \frac{1}{j^{3/2}} e^{i 2\pi 2^j \theta}, \quad \theta \in [0,1), 
\end{equation}
belongs to ${\rm LMOA} (\T)$. Note that the lacunary sequence in \eqref{LMOA_lac_ex} does not satisfy \eqref{lac_BMO_Psi_cond1} (for $\Psi= \Psi_{1,1}$). 
 
\subsection*{Euclidean counterparts and applications}
Corollary \ref{HL_complex}, and our other results, have natural counterparts in the Euclidean setting. To formulate them we need to recall the definition of Hardy--Orlicz spaces on $\R^d$.  Following \cite{YLK}, if $\Psi$ is a growth function, let $ \mathcal{F}_{m_{\Psi}} $ denote  the family of Schwartz functions on $\R^d$ given by
\[
\mathcal{F}_{m_{\Psi}} : = \left\{ \phi \in \mathcal{S} (\R^d) :  \sup_{x \in \R^d} \sup_{\substack{ \beta \in \N_0^d : \\ |\beta| \leq m_{\Psi} +1 }} (1+|x|)^{(m_{\Psi}+2) (d+1)} \abs{\partial^{\beta} \phi (x)} \leq 1 \right\} ,
\]
where $m_{\Psi}$ is an appropriate non-negative integer depending on $\Psi$; see p. 16 in \cite{YLK}. 
For $t > 0$, we use the notation $\phi_{t} (x) : = t^{-d} \phi (t^{-1}x)$, $x \in \R^d$.  If $f$ is a tempered distribution on $\R^d$, define the corresponding non-tangential grand maximal function of $f$ by
\[
M_{ \mathcal{F}_{m_{\Psi}}}^{\ast} [f] (x) : = \sup_{\phi \in  \mathcal{F}_{m_{\Psi}}} \sup_{\substack{ (y,t ) \in \R^d  \times [0, \infty) : \\ |x-y| <t }  } \abs{ f \ast \phi_{t} (y)}, \quad x \in \R^d . 
\]

The Hardy--Orlicz space $H^{\Psi} (\R^d)$ is defined as the class of all tempered distributions on $\R^d$ satisfying
\[
\int_{\R^d} \Psi \brkt{M_{ \mathcal{F}_{m_{\Psi}}}^{\ast} [f ] (x)} \dd x < \infty. 
\]
If $f \in H^{\Psi} (\R^d)$, we set
\[ \| f \|_{H^{\Psi} (\R^d)} : = \inf \left\{ \lambda > 0 :   \int_{\R^d}    \Psi \brkt{\lambda^{-1} M_{ \mathcal{F}_{m_{\Psi}}}^{\ast} [f ] (x)}  \dd x   \leq 1 \right\} . 
\]

The local Hardy--Orlicz space $h^{\Psi} (\R^d)$ is defined as the class of all tempered distributions $f$ on $\R^d$ that are such that $ \Psi ( M_{ \mathcal{F}_{m_{\Psi}}, \mathrm{loc}}^{\ast} [f ]  ) \in L^1 (\R^d) $, where 
\[
M_{ \mathcal{F}_{m_{\Psi}}, \mathrm{loc}}^{\ast} [f ] (x) : = \sup_{\phi \in  \mathcal{F}_{m_{\Psi}}} \sup_{\substack{ (y,t ) \in \R^d  \times [0, 1) : \\ |x-y| <t }  } \abs{ f \ast \phi_{t} (y)}, \quad x \in \R^d .
\]
Note the restriction to $\R^d\times [0,1)$ in the first supremum on the right in the definition of $M^*_{\mathcal{F}_{m_{\Psi},\mathrm{loc}}}[f]$. For $f \in h^{\Psi} (\R^d)$, one sets
\[
\| f \|_{h^{\Psi} (\R^d)} : = \inf \left\{ \lambda > 0 :   \int_{\R^d}    \Psi \brkt{\lambda^{-1} M_{ \mathcal{F}_{m_{\Psi}}, \mathrm{loc}}^{\ast} [f ] (x)}  \dd x   \leq 1 \right\} . 
\]

Our next result implies a Euclidean analogue of Corollary \ref{HL_complex}.
 
\begin{theo}\label{eucl} Let $d \in \N$ be a given dimension and let $\Psi$ be a growth function of order $p$ with $p \in ( d/ (d+1), 1]$. 

Then there exists a constant $C_{d, \Psi, p} > 0$ such that 
\[
\int_{\R^d}   \frac{\Psi \brkt{ \abs{ \xi}^d \abs{ \widehat{f} (\xi) } } } { \abs{ \xi }^{2d} }\dd\xi \leq C_{d, \Psi,p}   \int_{\R^d} \Psi \left( M_{ \mathcal{F}_{m_{\Psi}}^*} [f] (x) \right)  \dd x
\] 
for all $f \in H^{\Psi} (\R^d)$.
\end{theo}

As a consequence of this result, we obtain in Corollary \ref{cor}, a Euclidean counterpart of Corollary \ref{HL_complex}.

Notice that if $f$ is a tempered distribution,  then its Fourier transform  does not necessarily coincide with a function. However (see e.g. \cite{Big_Stein}*{\S 5.4 (a) Chapter III}), if $f\in H^p(\R^d)$ with $p \in (0,1]$, its Fourier transform agrees with a continuous function and 
\[
	\abs{\widehat{f} (\xi)} \lesssim_{d,p} \abs{\xi}^{d (p^{-1}-1)} \norm{f}_{h^p (\R^d)}, \quad \xi \in \R^d. 
\]
In Proposition \ref{FT_HO} we give an extension of this result for distributions in $H^\Psi(\R^d)$, and an analogous result for $h^\Psi(\R^n)$ is obtained in Proposition \ref{FT_HO_local}.

The Littlewood--Paley description of the Hardy space $h^p(\R^d)$ can be stated in terms of the validity of the functional inequality (see \S \ref{s_log_smooth} for the notation) 
\begin{equation}\label{eq:embedding_HL}
\norm{\left(\sum_{j=0}^\infty\abs{\widetilde{\Delta}_jf}^2\right)^{1/2}}_{L^p(\R^d)}\lesssim \norm{f}_{H^p (\R^d)},
\end{equation}
and its reversed counterpart. The expression in the left-hand side can be identified as the norm of $f$ in the inhomogeneous Triebel--Lizorkin space $F^{0}_{p,2}(\R^d)$, allowing the inequality to be interpreted as saying that $h^p(\R^d)$ is embedded in $F^{0}_{p,2}(\R^d)$.

In Theorem \ref{LP_Psi_r_euclidean} we give an extension of \eqref{eq:embedding_HL} for $p=1$, obtaining that the space $h^{\Psi_{r,1}}(\R^d)$ can be embedded in the Triebel--Lizorkin space of generalised smoothness $F^{0,-r}_{1,2}(\R^d)$ (see \S \ref{s_log_smooth}).

The classical Sobolev embedding yields that, given $p \in (0,1]$, if a distribution $f$ satisfies $(1-\Delta)^{\frac{d}{2p}}f\in h^p(\R^d)$, then $f$ is a continuous and bounded function (see e.g. \cite{Franke}*{Theorem 2}). Informally, this can be rephrased as saying that if $f$ has its derivatives of order smaller or equal to $dp$ in $h^p(\R^d)$, then $f$ is a continuous and bounded function.  As a direct consequence of Corollary \ref{cor}, we obtain an extension of this  result for distributions in $h^{\Psi}(\R^d)$, and in particular in $h^{\Psi_{r,p}}(\R^d)$, (see Corollaries  \ref{Cor:18} and \ref{cor_sobolev_embedding}), involving derivatives of logarithmic order. 

In the direction of obtaining an extension of \eqref{eq:embedding_HL} for $\Psi_{r,p}$ for $p<1$, in Corollary \ref{cor:Sobolev-embedding_local} we establish a Sobolev-type embedding of $h^{\Psi_{r,p}}(\R^d)$ into spaces of generalised smoothness. We later apply this to establish an embedding of those distributions in $h^{\Psi_{r,p}}(\R)$, $p \in (1/2, 1]$, whose Fourier transforms are supported in $[0,\infty)$, into certain spaces of analytic functions in a halfplane (see Corollary \ref{cor:embedding_Analytic}). 
\subsection*{Organisation of the paper} 

In Section \ref{background} we set down basic notation and provide some background on growth functions, Hardy--Orlicz spaces and their atomic decomposition, as well as $H^{\Psi}$-$BMO(\rho_{\Psi})$ duality. 

Section \ref{firstproofsection} contains the proofs of Theorems \ref{gen_mult_p<1} and \ref{gen_mult_q_leq2}, spread over several subsections, as well as several auxiliary results.

In Section \ref{BMOsection}, we prove Theorem \ref{LP_BMO_Psi} and discuss some of its consequences.

Section \ref{eucl_proof} is devoted to the Euclidean setting and in particular to the proof of Theorem \ref{eucl} and a counterpart of Corollary \ref{HL_complex}. More specifically, in Subsection \ref{behaviour} we make some introductory remarks concerning the behaviour of the Fourier transform on Hardy--Orlicz spaces and then in Subsection \ref{eucl_det} we prove Theorem \ref{eucl}. In subsection \ref{applicationsection} we present some applications of the Euclidean results to Sobolev-type embeddings and spaces of analytic functions in a half-plane.

\section{Notation and Background}\label{background}

The set of natural numbers is denoted by $\N$, the set of non-negative integers is denoted by $\N_0$, and the set of integers is denoted by $\Z$.

In what follows, we identify functions on $\T$ with $1$-periodic functions in the usual way.  

We denote the class of Schwartz functions by $ \mathcal{S} (\R^d)$ and we denote by $ \mathcal{C}_0 (\R^d)$ the class of all continuous functions $f$ on $\R^d$ such that $f (x) \rightarrow 0$ as $|x| \rightarrow \infty$. The class of distributions on $\T$ is denoted by $\mathcal{D}'(\T)$.

\subsection{Growth functions}
We now give a formal definition of the growth functions used to define $H^{\Psi}$ spaces.
\begin{define}\label{main_def}
Let $0 < p \leq 1$. A function $\Psi : [0, \infty) \rightarrow [0, \infty)$ is called a {\it growth function of order} $p$ if it has the following properties: \begin{itemize}
\item  $\Psi$ is continuous and increasing, with $\Psi (0) = 0$ and 
$\lim_{t\to \infty}\Psi (t)=\infty$;
\item the function $t \mapsto t^{-1}\Psi(t)$ is non-increasing on $(0, \infty)$; 
\item  $\Psi$ is of lower type $p$, namely, there exists a constant $c_{\Psi, p} >0$ such that
\[
\Psi (s \cdot t) \leq c_{\Psi, p} \cdot s^p \cdot \Psi (t) 
\]
for all $t \geq 0$ and $s \in (0,1)$.
\end{itemize} 

If $\Psi$ is a growth function of order $p$ for some $p \in (0, 1 ]$, we say that $\Psi$ is a {\it   growth function}. 
\end{define}
 
As noted in \cite{BG}, see also \cite{V}, given a growth function $\Psi$ of order $p$, then $\widetilde{\Psi} (t) : = \int_0^t s^{-1} \Psi (s) \dd s$ is also a growth function of order $p$ that is concave and point-wise equivalent to $\Psi$. Hence, in what follows, we may assume that the growth functions considered are always concave and hence, sub-additive.

In several parts of this paper we shall use the following standard fact: if $\Psi$ is a growth function of order $p$ then $t \mapsto t^{-l} \Psi (t)$ is quasi-increasing with constant $c_{\Psi, p}$ for all $l \in (0, p]$. Namely,
\begin{equation}\label{eq:growth}
\frac{\Psi (s)} {s^l} \leq c_{\Psi, p } \frac{\Psi (t)} {t^l}
\end{equation}
for all $0<s \leq t$. See \cite{V}*{Proposition 3.1} for further details.

\subsection{Hardy--Orlicz spaces}

If $\Psi$ is a growth function, then the `real-variable' Hardy--Orlicz space $H^{\Psi} (\T)$ consists of all distributions $f$ on $\T$ satisfying
\[
\int_{[0,1)} \Psi \brkt{ f^{\ast} ( e^{ i 2 \pi \theta} ) } \dd \theta < \infty. 
\]
Here $f^{\ast}$ denotes the non-tangential maximal function of $f \in \mathcal{D}' (\T)$ given by
\[
f^{\ast} (  e^{ i 2 \pi \theta}  ) : = \sup_{r e^{i 2 \pi \phi} \in \Gamma (  \theta ) } \abs{ \sum_{n \in \Z} r^{|n|} \widehat{f}(n) e^{i 2 \pi n \phi} } , 
\]
where \[\Gamma ( \theta ) : = \{ z \in \D : |z - e^{i 2 \pi \theta }| < 2 (1-|z|) \} \] and $ \widehat{f} (n) : = \langle f, e_n \rangle$, $n \in \Z$, with $e_n (\theta) : = e^{i 2 \pi n \theta}$, $\theta \in [0,1)$. 

If $f \in H^{\Psi} (\T)$, we set
\[
\norm{f}_{H^{\Psi} (\T)} : = \inf \left\{ \lambda > 0 :   \int_{[0, 1)} \Psi \brkt{ \lambda^{-1}   f^{\ast} ( e^{ i 2 \pi \theta} ) }  \dd \theta   \leq 1 \right\}  . 
\]

The analytic Hardy--Orlicz space $H^{\Psi}_A (\T)$ is defined as
\[
H^{\Psi}_A (\T) : = \left\{ f \in H^{\Psi} (\T):  \mathrm{supp} \brkt{ \widehat{f} } \subset \N_0 \right\} . 
\]
Recall that the Riesz projection $P$ is the multiplier operator with symbol $\chi_{\N_0}$. By arguing as in \cite{YLK}*{Chapter 4}, one can show that $P$ is bounded on $H^{\Psi} (\T)$ and moreover, one can show that if $f \in H^{\Psi} (\T)$ then for $f_1 := P (f) $ and $f_2 := f - f_1$ one has that $f = f_1 + f_2$, $f_1, \overline{f_2} \in H^{\Psi}_A (\T)$, and
\begin{equation}\label{norm_equiv}
 \norm{f}_{H^{\Psi} (\T)} \approx \norm{f_1}_{H^{\Psi}_A (\T)} +  \norm{\overline{f_2}}_{H^{\Psi}_A (\T)}. 
\end{equation}

It is well-known that $H^{\Psi} (\D)$ can be identified with $H^{\Psi}_A (\T)$ in the following way: if $F \in H^{\Psi} (\D)$ then it converges in the sense of distributions to an $f  \in \mathcal{D}' (\T)
$ that belongs to $H^{\Psi}_A (\T)$ with $\norm{F}_{H^{\Psi} (\D)} \lesssim  \norm{f}_{H^{\Psi} (\T)}$ and conversely, if $f \in H^{\Psi}_A (\T)$ then  
\[
F(r e^{i2 \pi \theta}) := \sum_{n=0}^{\infty} r^n \widehat{f} (n) e^{i 2 \pi n \theta}, \quad z = r e^{i2 \pi \theta} \in \D, \]
is a well-defined function that belongs to $H^{\Psi} (\D)$ and satisfies $\norm{f}_{H^{\Psi} (\T)} \lesssim  \norm{F}_{H^{\Psi} (\D)}$; see, e.g., \cite{BG}.

\subsection{Duality and atomic decompositions} Let $\Psi$ be a growth function and let $\rho_{\Psi}$ be as above, i.e.,
\begin{equation}
\rho_{\Psi} (t):=t^{-1} /\Psi^{-1} (t^{-1}), \quad  t > 0. 
\label{rhodef}
\end{equation}
Let $q \in [1, \infty)$. It follows from the work of Viviani \cite{V} that $u \in \mathrm{BMO} (\rho_{\Psi}) (\T)$ if, and only if,
\[
 \sup_{\substack{ I \subseteq \T  : \\  \text{ arc } }} \left\{ \frac{1} { \left[ \rho_{\Psi} (|I|) \right]^q |I|} \int_{I} \abs{u (e^{ i 2 \pi \theta}) - \langle u \rangle_I}^q \dd \theta \right\} < \infty 
\]
and moreover, one has
\begin{multline}\label{BMO_Psi_norms}
\norm{u}_{\mathrm{BMO}^+ (\rho_{\Psi}) (\T)} \approx_q \\
\abs{ \int_{[0,1)} u ( e^{ i 2 \pi \theta} ) \dd \theta} + \sup_{\substack{ I \subseteq \T  : \\  \text{ arc } }}  \left( \frac{1} { \left[ \rho_{\Psi} (|I|) \right]^q |I|}  \int_{I} \abs{u (e^{ i 2 \pi \theta}) - \langle u \rangle_I}^q \dd \theta \right)^{1/q} .
\end{multline}
In particular, 
\begin{multline}\label{BMO_Psi_1}
\norm{u}_{\mathrm{BMO}^+ (\rho_{\Psi}) (\T)} \approx \\
\abs{ \int_{[0,1)} u ( e^{ i 2 \pi \theta} ) \dd \theta} + \sup_{\substack{ I \subseteq \T  : \\  \text{ arc } }} \left\{ \Psi^{-1} ( |I|^{-1} ) \int_{I} \abs{u (e^{ i 2 \pi \theta}) - \langle u \rangle_I} \dd \theta \right\} . 
\end{multline}

As was shown by Janson \cite{J}, for a given growth function $\Psi$, the dual of $H^{\Psi}$ can be identified with $BMO(\rho_{\Psi})(\T)$; the choice of $\rho_{\Psi}$ in \eqref{rhodef} is made precisely so as to make this duality hold. Note that $\rho_{\Psi}$ is a positive non-decreasing function \cite{V}*{Proposition 3.10}. In \cite{V}, B. Viviani proved the very useful result that the Hardy--Orlicz space $H^{\Psi} (\R^d)$ admits an atomic decomposition. As a consequence, she gave another proof of Janson's result on duality. 

Let us first recall the definition of atoms in the Euclidean setting. 
\begin{define}\label{atomdef}
Let $\Psi$ be a growth function. A locally integrable function $a_Q$ is said to be an {\it atom} in $H^{\Psi} (\R^d)$ associated to a cube $Q\subset \R^d$ if 
\begin{enumerate}[label=\rm{\roman*)}]
\item $\mathrm{supp} (a_Q) \subset Q$;
\item $\abs{a_Q(x)} \leq \Psi^{-1}(1/\abs{Q})$ for a.e. $x \in Q$;
\item $\int_Q a_Q (x)\dd x=0$.
\end{enumerate}
\end{define}

Note here that 
\[
\norm{\chi_Q}_{L^{\Psi}(\R^d)}=\inf\set{\lambda > 0 : \int_Q \Psi(1/\lambda)\leq 1}=1/\Psi^{-1}(1/\abs{Q}) 
\]
and so, the second property in the definition of atoms can be rewritten as
\[ \norm{a_Q}_{L^{\infty} (\R^d)} \leq \norm{\chi_Q}^{-1}_{L^{\Psi}(\R^d)} .\] 

A tempered distribution $f $ belongs to $ H^{\Psi} (\R^d)$ if, and only if, there exists a sequence $\{ b_{Q_k} \}_{k \in \N}$ of multiples of $H^{\Psi} (\R^d)$-atoms such that
\[
f = \sum_{k \in \N} b_{Q_k} 
\]
in the sense of distributions, and
\[
\sum_{k \in \N} \abs{Q_k} \Psi \brkt{ \norm{b_{Q_k}}_{L^{\infty} (\R^d)} } < \infty. 
\]
Moreover, we have
\[
\norm{f}_{H^{\Psi} (\R^d)} \approx \Lambda_{\infty} \brkt{\{ b_{Q_k} \}_{k \in \N}} ,
\]
where
\[ 
\Lambda_{\infty} \brkt{\{ b_{Q_k} \}_{k \in \N} } := 
\inf \left\{  \lambda > 0 : \sum_{k \in \N} \abs{Q_k} \Psi \brkt{ \lambda^{-1} \norm{b_{Q_k}}_{L^{\infty} (\R^d)} } \leq 1 \right\} . 
\]
An analogous characterisation holds in the periodic setting.

As for $H^{\Psi} (\R^d)$, the local Hardy--Orlicz space $h^{\Psi} (\R^d)$ admits an atomic decomposition (see for instance \cite{YLK}*{\S 8.3}). Let us recall that a locally integrable function $a_Q$ is an atom in $h^{\Psi} (\R^d)$ associated to a cube $Q$ if it satisfies i), ii) from Definition \ref{atomdef} above, and 
\begin{enumerate}[label=\roman*)']\setcounter{enumi}{2}
\item $\int_Q a_Q (x)\dd x=0$, when $\abs{Q}<1$.
\end{enumerate}

Similarly as for $H^\psi(\R^d)$, a tempered distribution $f $ belongs to $h^{\Psi} (\R^d)$ if, and only if, there exist a sequence $\{\beta_{Q_k} \}_{k\in \N}$ of multiples of atoms in $h^{\Psi} (\R^d)$, such that
\[
f = \sum_{k \in \N} \beta_{Q_k}
\]
in the sense of distributions and
\[
\sum_{k \in \N} \abs{Q_k} \Psi \brkt{ \norm{\beta_{Q_k}}_{L^{\infty} (\R^d)} } < \infty. 
\]
Moreover, 
\[
\norm{f}_{h^{\Psi} (\R^d)} \approx \Lambda_{\infty} \brkt{\{ \beta_{Q_k} \}_{k \in \N} }.
\]

\subsection{Well-distributed intervals and Littlewood--Paley-type projections}\label{LP_part}  

In this paper, we shall consider `frequency' intervals that arise from Whitney-type decompositions of arbitrary intervals. To be more specific, following \cite{RdF}, if $K = [ c_K - |K|/2, c_K + |K|/2]$ is a non-empty interval, then its Whitney decomposition $W(I)$ consists of the families of intervals
\[ 
\mathbf{K}^{\rm left} := 
\left\{ \left[ c_K -\frac{|K|}{2} + \frac{ 2^{-(l+1)}}{3} |K| , c_K -\frac{ |K|}{2} + \frac{ 2^{-l}}{3} |K| \right) \right\}_{l \in \Z},
\]
\[
\mathbf{K}^{\rm centre} := \left\{ \left[ c_K -\frac{ |K|}{6}, c_K + \frac{|K|}{6} \right] \right\} ,  
\]
and
\[
\mathbf{K}^{\rm right } := \left\{ \left( c_K + \frac{|K|}{2} - \frac{2^{-l}}{3} |K| , c_K + \frac{|K|}{2} - \frac{ 2^{-(l+1)}}{3} |K| \right] \right\}_{l \in \Z}. 
\]
Note that for any given non-empty interval $K$ interval, the intervals in $W(K)$ are mutually disjoint and satisfy the properties
\begin{equation*} 
 \norm{ \sum_{J \in W(K)} \chi_{2J} }_{L^{\infty} (\R)} \leq 5 \qquad \text{and} \qquad 2 J \subseteq K \quad \text{for all } J \in W (K).
\end{equation*}

For $n \in \N$, let $K_n$ denote the usual Fej\'er kernel of order $n$, namely,
\[ 
K_n (e^{i 2 \pi \theta}):= \sum_{j=-n}^n \left( 1- \frac{|j|}{n+1} \right)  e^{i 2 \pi j \theta}, \quad \theta \in [0,1). 
\]
Then the de la Vall\'ee Poussin kernel of order $m \in \N$ is given by
\[ 
V_m := 2 K_{2m+1} - K_m 
\]
and satisfies $\widehat{V_m} (j) = 1$ for all $j \in \Z$ with $|j| \leq m+1$, and $\widehat{V_m} (j) = 0$ for all $j \in \Z$ with $|j| \geq  2m + 1$. Moreover, $\widehat{V_m}$ is even and `affine' on $ [m+1, 2m+1] \cap \Z $. 

For $m \in \N$, we set
\[ 
\sigma_m := V_{2m} - V_m .
\]
Note that $\widehat{\sigma_m} (j)=0$ for all $j \in \Z$ with $|j| \leq m+1$ or $|j| \geq 4m+1$. Moreover, $\widehat{\sigma_m}$ is even and `affine' on $ [m+1, 2m+1] \cap \Z $ and on $ [2m+1, 4m+1] \cap \Z $.

Suppose now that $K$ is an interval with $|K| = 3 \cdot 2^{l+1}$ for some $l \in \N_0$ and moreover, its endpoints $l_K := c_K - |K|/2 $ and $r_K := c_K + |K|/2 $ are integers. Let $W(K)$ denote the Whitney decomposition of $K$. If $J \in \mathbf{K}^{\rm left} $ and $|J| \in \N$, define 
\[ 
\delta_J (e^{i 2 \pi \theta}) :=  e^{i 2 \pi r_J \theta} e^{-i 2 \pi (2m)  \theta} \sigma_{|J|} (e^{i 2 \pi \theta}) , \quad \theta \in [0,1),
\]
where $r_J$ denotes the right endpoint of $J$. Similarly, if $J \in \mathbf{K}^{\rm right} $ and $|J| \in \N$, define 
\[ 
\delta_J (e^{i 2 \pi \theta}) :=  e^{i 2 \pi l_J \theta} e^{i 2 \pi (2m)  \theta} \sigma_{|J|} (e^{i 2 \pi \theta}) , \quad \theta \in [0,1),
\]
where $l_J$ is the left endpoint of $J$. 

For $J \in \mathbf{K}^{\rm left} \cup \mathbf{K}^{\rm right}$ and $f \in L^1 (\T) $, we set $\Delta_J (f) := \delta_J \ast f$.

\section{Proofs of Theorems \ref{gen_mult_p<1} and \ref{gen_mult_q_leq2}}\label{firstproofsection}
In this section we prove Theorems \ref{gen_mult_p<1} and \ref{gen_mult_q_leq2}. 
To this end, we shall first establish the following characterisation of $\mathcal{M}_{H^{\Psi} (\D) \rightarrow \ell^1 (\N_0)}$.

\begin{proposition}\label{equiv_prop}
Let $\Psi$ be a growth function of order $p \in (1/2, 1]$.  
 
Let $\lambda = \{ \lambda_n \}_{n \in \N_0}$ be a sequence of complex numbers.
The following are equivalent:
\begin{enumerate}
\item $\lambda$ is a multiplier from $H^{\Psi} (\D)$ to $\ell^1 (\N_0)$;
\item There exists a constant $A_{\lambda, \Psi, p}>0$  such that
\[
\Gamma_{\epsilon, \lambda}^{(M)} (z) : = \sum_{n=0}^M \epsilon_n \lambda_n z^n, \quad z \in \D,
\]
belongs to $BMO (\rho_{\Psi}) (\D)$ with 
\[
\norm{\Gamma_{\epsilon, \lambda}^{(M)}}_{BMO (\rho_{\Psi}) (\D)} \leq A_{\lambda, \Psi,p} 
\]
for every choice of $\epsilon = \{ \epsilon_n \}_{n \in \N_0}$ with $| \epsilon_n | \leq 1$, $n \in \N_0$, and for all $M \in \N_0$.
\end{enumerate}
\end{proposition}

\begin{proof} To prove $(2) \implies (1)$, fix a function 
\[
F (z) = \sum_{n=0}^{\infty} f_n z^n, \quad z \in \D,
\]
in $H^{\Psi} (\D)$. We shall prove that there exists a constant $C_{\lambda, \Psi, p} > 0$, depending only on $\lambda$, $\Psi$, and $p$, such that
\begin{equation}\label{ineq_r}
\sum_{n=0}^M r^n | \lambda_n f_n| \leq C_{\lambda, \Psi, p}  \norm{ F }_{H^{\Psi} (\D)}
\end{equation}
for all $r \in (0,1)$ and for all $M \in \N_0$. The desired inequality then follows from \eqref{ineq_r} and a limiting argument.

In order to establish \eqref{ineq_r}, fix an $r \in (0,1)$ and define $F_r$ by
\[
F_r (z): = F(rz) = \sum_{n=0}^{\infty} r^n f_n z^n, \quad z \in \D . 
\]
Note that 
\begin{equation}\label{bound_r}
\norm{ F_r }_{H^{\Psi} (\D)} \leq \norm{ F }_{H^{\Psi} (\D)} . 
\end{equation}
Indeed, by the definition of the norm and a change of variables,
\begin{align*}
\norm{ F_r }_{H^{\Psi} (\D)} &  = \inf \left\{ \alpha >0 : \sup_{0 \leq \rho <1} \int_{[0,1)} \Psi \left(\alpha^{-1} \abs{F \left(r \rho e^{i 2\pi \theta} \right)} \right) \dd \theta \leq 1 \right\} \\
& \leq \inf \left\{ \alpha >0 : \sup_{0 \leq \sigma < 1} \int_{[0,1)} \Psi \left(\alpha^{-1} \abs{F \left( \sigma e^{i 2\pi \theta} \right)} \right) \dd \theta \leq 1 \right\} = \norm{ F }_{H^{\Psi} (\D)}.
\end{align*}

Fix an $\epsilon = \{ \epsilon_n \}_{n \in \N_0}$ with $|\epsilon_n| \leq 1$ for all $n \in \N_0$ and an $M \in \N_0$. We may write
\begin{equation}\label{pairing}
\abs{ \sum_{n=0}^M \epsilon_n r^n \lambda_n f_n   } = \abs{ \langle F_r ,  \Gamma_{\epsilon, \lambda, 0}^{(M)} \rangle }, 
\end{equation}
where $\Gamma_{\epsilon, \lambda, 0 }^{(M)} $ denotes the analytic trigonometric polynomial, which is the boundary value of $\Gamma_{\epsilon, \lambda}^{(M)} $.
By our assumption 
\[ 
\norm{ \Gamma_{\epsilon, \lambda}^{(M)} }_{BMO (\rho_{\Psi}) (\D)} \lesssim_{\lambda, \Psi, p} 1,
\]
where the implied constant is independent of the choice of $\epsilon = \{ \epsilon_n \}_{n \in \N_0}$ and $M \in \N_0$. Hence, it follows from \eqref{pairing} that
\[
\abs{ \sum_{n=0}^M \epsilon_n r^n  \lambda_n f_n }  \lesssim_{\lambda, \Psi, p} \norm{ F_r }_{H^{\Psi} (\D)}  ,
\]
which, combined with \eqref{bound_r}, yields
\begin{equation}\label{uni_eps}
\abs{ \sum_{n=0}^M \epsilon_n r^n  \lambda_n f_n   }  \leq C_{\lambda, \Psi, p} \norm{ F }_{H^{\Psi}   (\D)} ,
\end{equation}
where $C_{\lambda, \Psi, p} > 0$ is independent of $\epsilon = \{ \epsilon_n \}_{n \in \N_0}$ and $M \in \N_0$. Therefore, if we take
\[
\epsilon_n : =
\begin{cases}
 \frac{\overline{f_n}}{|f_n|} \cdot \frac{\overline{\lambda_n}}{|\lambda_n|} , \quad  & \text{if } f_n \neq 0; \\
0, \quad & \text{otherwise}  
\end{cases}
\]
in  \eqref{uni_eps}, we deduce that \eqref{ineq_r} holds.

To prove the reverse implication, suppose that $\lambda = \{ \lambda_n  \}_{n \in \N_0}$ satisfies $(1)$, i.e. $\lambda$ is a multiplier from $H^{\Psi} (\D)$ to $\ell^1 (\N_0)$.
 
For $\epsilon = \{ \epsilon_n \}_{n \in \N_0}$ with $| \epsilon_n | \leq 1$ for all $n \in \N_0$ and $M \in \N_0$, consider the trigonometric polynomial $b_{ \epsilon, \lambda }^{(M)}$ with Fourier coefficients given by 
\[
\widehat{b_{ \epsilon, \lambda }^{(M)}} (n): =
\begin{cases}
\epsilon_n \lambda_n , \quad &\text{if } n \in \N_0 \text{ with } n \leq M ; \\
0 \quad &\text{otherwise.}
\end{cases}
\]
To prove $(2)$, it suffices, in view of \eqref{BMO_Psi_1}, to show that 
\begin{equation}\label{BMO_1}
\sup_{\substack{ I \subseteq \T : \\ \text{arc}}} \left\{ \Psi^{-1} ( | I |^{-1} ) \int_I \abs{ b_{ \epsilon, \lambda }^{(M)} (e^{i 2 \pi \theta} ) - \langle b_{ \epsilon , \lambda }^{(M)} \rangle_I }\, \dd \theta \right\} \leq A_{\lambda, \Psi, p};   
\end{equation}
our arguments are in part inspired by a trick in \cite{JPS}*{p. 946}.
To establish \eqref{BMO_1}, fix an arc $I$ in $\T$ and define
\[ a_I  : = \frac{ \Psi^{-1} ( | I |^{-1} )}{2}   \brkt{ e^{ i \arg \left\{ b_{ \epsilon , \lambda }^{(M)} - \langle b_{  \epsilon , \lambda }^{(M)} \rangle_I \right\}} - \langle e^{ i \arg \left\{ b_{ \epsilon,  \lambda }^{(M)} - \langle b_{ \epsilon, \lambda}^{(M)} \rangle_I  \right\}} \rangle_I } \chi_I.
 \]
Then, $\mathrm{supp} (a_I) \subseteq I$, $ \int_I a_I = 0$, $\norm{ a_I }_{L^{\infty} (\T)} \leq \Psi^{-1} ( | I |^{-1} )$ and so, $a_I $ is an $H^{\Psi}$-atom. Moreover,
\begin{align*}
& \Psi^{-1} ( | I |^{-1} ) \int_I \abs{ b_{ \epsilon , \lambda}^{(M)} (e^{i 2 \pi \theta} ) - \langle b_{ \epsilon , \lambda}^{(M)} \rangle_I } \dd \theta  \\
& =2 \abs{ \int_I \brkt{ b_{ \epsilon , \lambda}^{(M)} (e^{i 2 \pi \theta} ) - \langle b_{ \epsilon, \lambda }^{(M)} \rangle_I } \overline{a_I (e^{i2 \pi \theta})} \, \dd \theta } . 
\end{align*}
Hence, Parseval's identity yields
\[ \Psi^{-1} ( | I |^{-1} ) \int_I \abs{ b_{ \epsilon , \lambda }^{(M)} (e^{i 2 \pi \theta} ) - \langle b_{ \epsilon, \lambda}^{(M)} \rangle_I } \dd \theta \leq 2 \sum_{n = 0 }^M | \lambda_n \widehat{a}_I (n) | 
\]
and so, our assumption on $\lambda = \{ \lambda_n \}_{n \in \N_0}$ implies that
\begin{equation}\label{Par_BMO1}
\Psi^{-1} ( | I |^{-1} ) \int_I \abs{ b_{ \epsilon , \lambda}^{(M)} (e^{i 2 \pi \theta} ) - \langle b_{ \epsilon, \lambda }^{(M)} \rangle_I } \dd \theta  \leq 2 C_{\lambda, \Psi, p} \norm{a_I}_{H^{\Psi}(\T)} . 
\end{equation}
Since $a_I$ is an $H^{\Psi}$-atom, \eqref{BMO_1} follows from \eqref{Par_BMO1}. \end{proof}

\subsection{Proof of Theorem \ref{gen_mult_p<1}}
The equivalence of (2) and (3) follows immediately from a geometric summation and the properties of $\Psi^{-1}$. Indeed, since $\Psi$ is a growth function of order $p$, it follows that $\Psi^{-1}$ grows at most polynomially and is thus quasi-constant on intervals of the form $[2^N, 2^{N+1}]$ that is, $\Psi^{-1} (t) \approx \Psi^{-1} (2^N) $ for all $t \in [2^N, 2^{N+1}]$ with the implied constants being independent of $N$. 

As to the equivalence of (1) and (2), we start by observing that it suffices to prove the case $q =1$. Indeed, to see this, notice that, for $q \in (1, \infty)$, a sequence $\lambda = \{ \lambda_n \}_{n \in \N_0}$ is a multiplier from $H^{\Psi} (\D)$ to $\ell^q (\N_0)$ if, and only if, for every $\alpha = \{ \alpha_n \}_{n \in \N_0}$ the sequence $\lambda_{\alpha} = \{ \alpha_n \lambda_n \}_{n \in \N_0}$ is a multiplier from $H^{\Psi} (\D)$ to $\ell^1 (\N_0)$ with
\[
\norm{ \lambda_{\alpha} }_{\mathcal{M}_{H^{\Psi} (\D) \rightarrow \ell^1 (\N_0)}} \leq \norm{\alpha}_{\ell^{q'} (\N_0)} \norm{ \lambda }_{\mathcal{M}_{H^{\Psi} (\D) \rightarrow \ell^q (\N_0)}}. 
\]
Hence, if $(1) \Leftrightarrow (2)$ holds for $q=1$ then $(1) \Leftrightarrow (2)$ holds for $q \in (1, \infty)$ as well.

We will now turn to the proof of the case $q=1$.
Suppose that $\{ \lambda_n \}_{n \in \N_0}$ is an $H^{\Psi}(\D)$-$\ell^1 (\N_0)$ multiplier of norm 1, hence
also $\{ |\lambda_n|\}_{n \in \N_0} \in \mathcal{M}_{H^{\Psi}(\D) \rightarrow \ell^1 (\N_0)}$ with the same norm. Let $N \in \N$. By Proposition 
\ref{equiv_prop}, the function
\[
b_N( e^{i 2\pi \theta} ) : = \sum_{n=1}^N |\lambda_n| e^{i 2\pi n \theta}, \quad \theta \in [0,1),
\]
is in $BMO (\rho_{\Psi}) (\T)$, with norm bound independent of $N$.

Letting $I := [0, (4N)^{-1}]$, note that
since 
\begin{equation}    \label{eq:sinestimate}
 \langle \sin(2\pi n \cdot) \rangle_I - \sin(2\pi n \theta) \approx \langle \sin(2\pi n \cdot) \rangle_I  \approx \frac{n}{N}
\end{equation}
for $\theta \in [0,(16N)^{-1}]$, we have
\begin{align*} %
\sum_{n=1}^N |\lambda_n| \frac{n}{N^{2}} & \lesssim \int_I \left|\sum_{n=1}^N |\lambda_n| \left(\sin(2 \pi n \theta ) - \langle \sin(2\pi n \cdot) \rangle_I  \right)   \right| \, \dd \theta  \\
& \leq \int_I | b_N( e^{i 2\pi \theta} ) - \langle b_N \rangle_I | \, \dd \theta 
\end{align*}
and hence,
\begin{equation}\label{eq:osclow}
\sum_{n=1}^N |\lambda_n| \frac{n}{N^{2}}  \lesssim \frac{1}{\Psi^{-1}(|I|^{-1})}. 
\end{equation}

Now, since $t \mapsto t^{-1} \Psi(t)$ is non-increasing on $(0, \infty)$, the map $t \mapsto t^{-1}/ \Psi^{-1} (t)$ is non-decreasing on $(0, \infty)$ and hence,
\[
\sum_{n=1}^N {\Psi^{-1}(n)} |\lambda_n|
\leq \sum_{n=1}^N \Psi^{-1}(N) \frac{n}{N} |\lambda_n| \lesssim N,
\]
where the implied constant in the second inequality does not depend on $N$.

Conversely, suppose that (2) holds for $q=1$. By
Proposition \ref{equiv_prop}, we have to check that the function
\[
\Gamma_{\epsilon, \lambda}^{(M)} (z) : = \sum_{n=0}^M \epsilon_n \lambda_n z^n, \quad z \in \D,
\]
belongs to $BMO (\rho_{\Psi}) (\D)$, uniformly 
for every choice of $\epsilon = \{ \epsilon_n \}_{n \in \N_0}$ with $| \epsilon_n | \leq 1$, $n \in \N_0$, and for all $M \in \N_0$.
To this end, fix an interval $I$ in $\T$ and let $N \in \N$ be such that
\[
\frac{1}{N} \leq |I| \leq \frac{2}{N}.
\]
We write
\begin{equation}\label{eq:highlow}
\Gamma_{\epsilon, \lambda}^{(M)}     =
K_{\eps, \lambda} + N_{\eps, \lambda} ,
\end{equation}
where
\[ K_{\eps, \lambda} (z) :=  \sum_{n\leq |I|^{-1}} \epsilon_n \lambda_n z^n
\quad \text{and} \quad N_{\eps, \lambda} (z) := \sum_{M \geq n > |I|^{-1}} \epsilon_n \lambda_n z^n , \quad z \in \D.
\]
We consider first
\begin{align*}      \label{eq:lowfreq}
\int_I \left| K_{\eps, \lambda}(e^{i 2\pi \theta}) - \langle K_{\eps, \lambda} \rangle_I \right| \dd \theta 
&\leq \frac{1}{|I|}
\int_I \int_I \left| K_{\eps, \lambda}(e^{i 2 \pi \theta}) - K_{\eps, \lambda}(e^{i 2 \pi \phi}) \right| \, \dd \theta \, \dd \phi \\
&\lesssim  \sum_{n\leq |I|^{-1}} n |\lambda_n| \frac{1}{|I|}
\int_I \int_I \left|\theta - \phi \right| \, \dd \theta \, \dd \phi \\
&\lesssim  \sum_{n\leq |I|^{-1}} n |\lambda_n| |I|^2 \\
&\lesssim |I|  \sum_{k=1}^\infty 2^{-k} \sum_{ 2^{-k+1} |I|^{-1} \geq n> 2^{-k} |I|^{-1} } |\lambda_n|  \\
&\lesssim  |I| \sum_{k=1}^\infty 2^{-k} \frac{2^{-k} |I|^{-1}} {\Psi^{-1}(2^{-k} |I|^{-1})}  \\
&= |I|^{1+1/p} \sum_{k=1}^\infty 2^{-k (2- 1/p)} \frac{2^{-k/p} |I|^{-1/p}} {\Psi^{-1}(2^{-k} |I|^{-1})}  \\
&\lesssim_p  |I|^{1+1/p} \frac{ |I|^{-1/p}} {\Psi^{-1}( |I|^{-1})} 
=  \frac{1}{\Psi^{-1}( |I|^{-1})},
\end{align*}
where we used that, since $\Psi$ is a growth function of order $p > 1/2 $, $t \mapsto  t^{1/p} / \Psi^{-1}(t)$ is quasi-increasing on $(0, \infty)$. 
   
To show that 
\begin{equation}\label{N_eps}
\int_I \left| N_{\eps, \lambda}(e^{i 2 \pi \theta}) - \langle N_{\eps, \lambda} \rangle_I \right| \, \dd  \theta 
\lesssim \frac{1}{\Psi^{-1}( |I|^{-1})},
\end{equation}
we shall estimate
\begin{align*} 
\int_I \left| N_{\eps, \lambda}(e^{i 2 \pi \theta}) - \langle N_{\eps, \lambda} \rangle_I \right|^2 \, \dd  \theta
&\leq  \int_I \left| N_{\eps, \lambda}(e^{i 2 \pi \theta}) \right|^2 \, \dd  \theta \\
& \lesssim |I| \sum_{k=0}^{\infty} \left( \sum_{2^{k+1}|I|^{-1} \geq n > 2^k|I|^{-1}} |\lambda_n| \right)^2 \\
& \lesssim |I|\sum_{k=0}^{\infty} \frac{2^{2k} |I|^{-2}}{[\Psi^{-1}(|I|^{-1} 2^k) ]^2  } \\
& \lesssim |I|\int_{|I|^{-1}}^\infty  \frac{t}{[\Psi^{-1}(t)]^2  } \, \dd t \\
& \lesssim |I|\frac{|I|^{-2/p}}{[\Psi^{-1}(|I|^{-1})]^2}  \int_{|I|^{-1}}^\infty  \frac{1}{ t^{2/p -1} } \, \dd t  \lesssim \frac{|I|^{-1}}{[\Psi^{-1}(|I|^{-1})]^2},
\end{align*}
 where we have used orthogonality in the first inequality and Lemma \ref{lem:uncond1}, Part (\ref{geometric}) below in the second inequality, and finally the fact that $\Psi$ is of upper type $p \in (1/2, 1) $.
 Hence, \eqref{N_eps} follows from the Cauchy--Schwarz inequality and the last estimate. \qed

In  this last part of the proof of Theorem \ref{gen_mult_p<1} we used Lemma \ref{lem:uncond1} below, which can be understood as an unconditionality estimate.

At this point, we would like to mention that, even though the statement of Lemma \ref{lem:uncond1} might seem classical, we have not found Part (\ref{arithmetic}), which will be required for the proof of Theorem \ref{gen_mult_q_leq2} below, in the existing literature or been able to directly extract it from known results. 
The proof of Lemma \ref{lem:uncond1}, Part (\ref{arithmetic})  that we present here is elementary, if somewhat technical, and uses the $\ell^2$-boundedness of the discrete Hilbert transform. We deduce Part (\ref{geometric}) from Part (\ref{arithmetic}) here, but it can also be proved with a simpler direct estimate.

\begin{lemma}\label{lem:uncond1}
If $n \in \N_0$, %
then for every interval $I$ with $\abs{I}= 2^{-n}$ and for every analytic trigonometric polynomial $f$, 
\begin{enumerate}
\item \label{geometric}
\[    \int_I \left| f( e^{i 2 \pi \theta} ) \right|^2 \, \dd \theta \lesssim  \abs{I}  \brkt{\sum_{j=0}^{2^{n}-1} \abs{f_j}}^2 +
    \abs{I}  \sum_{k=0}^\infty \brkt{\sum_{j=2^k 2^n}^{2^{k+1} 2^{n}-1} \abs{f_j}}^2,
\]
\item \label{arithmetic}
\[    \int_I \left| f( e^{i 2 \pi \theta} ) \right|^2 \, \dd \theta \lesssim 
    \abs{I}  \sum_{k=0}^\infty \brkt{\sum_{j=k2^n}^{(k+1)2^{n}-1} \abs{f_j}}^2,
\]
\end{enumerate}
where $f_j := \widehat{f} (j)$, $j \in \N_0$. 
\end{lemma}
\begin{proof} Let $|I| = 2^{-n}$.
We remark that Part (\ref{geometric}) of the lemma follows from Part (\ref{arithmetic}) by a simple $\ell^1 - \ell^2$ estimate, so it suffices to prove Part (\ref{arithmetic}).

Note that
\begin{align*}
\int_I f(e^{i 2 \pi \theta})\overline{f(e^{i 2 \pi \theta})} \, \dd \theta & = \sum_{k_1 = 0}^\infty\sum_{k_2=0}^\infty \sum_{j=k_1 2^{n}}^{(k_1+1)2^{n}-1}\sum_{l=k_2 2^n}^{(k_2 +1) 2^{n}-1} f_j \overline{f_l} \int_I e ^{i 2\pi (j-l) \theta} \, \dd \theta \\
&=  \I + \II,
\end{align*}
where
\[
\I := \sum_{\substack{ k_1, k_2 \geq 0:\\ |k_1-k_2| \leq 1}} \sum_{j=k_1 2^{n}}^{(k_1+1)2^{n}-1}\sum_{l=k_2 2^n}^{(k_2 +1) 2^{n}-1} f_j \overline{f_l} \int_I e^{i 2\pi (j-l) \theta} \, \dd \theta
\]
and
\[
\II := \sum_{\substack{ k_1, k_2 \geq 0:\\ |k_1-k_2| \geq 2}} \sum_{j=k_1 2^{n}}^{(k_1+1)2^{n}-1}\sum_{l=k_2 2^n}^{(k_2 +1) 2^{n}-1} f_j \overline{f_l} \int_I e^{i 2\pi (j-l) \theta} \, \dd \theta .
\]

Term $\I$ may easily be estimated by
\[
| \I | \leq
\sum_{\substack{k_1, k_2 \geq 0:\\ |k_1-k_2| \leq 1}} \sum_{j=k_1 2^{n}}^{(k_1+1)2^{n}-1}\sum_{l=k_2 2^n}^{(k_2 +1) 2^{n}-1} | f_j | \cdot |  f_l | \cdot 2^{-n}
\leq 3 \cdot 2^{-n} \sum_{k \geq 0} \left(\sum_{j=k 2^{n}}^{(k+1)2^{n}-1} |f_j| \right)^2.
\]
For term $\II$, consider first the case $I = [0,2^{-n}]$. Then
\begin{align*}
&\sum_{\substack{k_1, k_2 \geq 0:\\ |k_1-k_2| \geq 2}} \sum_{j=k_1 2^{n}}^{(k_1+1)2^{n}-1}\sum_{l=k_2 2^n}^{(k_2 +1) 2^{n}-1} f_j \overline{f_l} \int_I e^{i 2\pi (j-l) \theta} \, \dd \theta = \\
& \sum_{\substack{k_1, k_2 \geq 0:\\ |k_1-k_2| \geq 2}}
\sum_{j=k_1 2^{n}}^{(k_1+1)2^{n}-1} \sum_{l=k_2 2^n}^{(k_2 +1) 2^{n}-1} f_j \overline{f_l} \, \frac{1 - e^{i 2 \pi (j-l) 2^{-n}}}{i 2\pi (j-l)} =  \\
&  \text{III} + \text{IV},
\end{align*}
where
\[
\text{III} :=  \sum_{\substack{ k_1, k_2 \geq  0:\\ k_1 \neq k_2} }^\infty 
\sum_{j = k_1 2^{n}}^{(k_1+1) 2^{n}-1} 
\sum_{l = k_2 2^{n}}^{(k_2+1) 2^{n}-1} f_{j} \overline{f_{l}} \frac{1}{i 2\pi ( j-l)}
\]
and
\[
\text{IV} := - \sum_{ \substack{k_1, k_2 \geq  0: \\ k_1 \neq k_2 }}^\infty 
\sum_{j = k_1 2^{n}}^{(k_1+1) 2^{n}-1} 
\sum_{l = k_2 2^{n}}^{(k_2+1) 2^{n}-1}  f_{j} \overline{f_{l}}  \frac{  e^{i 2 \pi (j-l) 2^{-n}}}{i 2\pi  ( j-l)} . 
\]

Moreover,
\begin{align*}
\text{III} &= \sum_{\substack{k_1, k_2\geq 0:\\ |k_1-k_2| \geq 2}}^\infty
\sum_{j_1=0}^{2^{n}-1} \sum_{j_2=0}^{2^{n}-1} f_{k_1 2^n +j_1} \overline{f_{k_2 2^n + j_2}}  \frac{1}{i 2\pi ( 2^n (k_1-k_2) + (j_1-j_2))} \\
& = \text{V} + \text{VI},
\end{align*}
where
\[
\text{V}:= \sum_{\substack{k_1, k_2\geq 0:\\ |k_1 -k_2|\geq 2}}
\sum_{j_1=0}^{2^{n}-1} \sum_{j_2=0}^{2^{n}-1} f_{k_1 2^n +j_1} \overline{f_{k_2 2^n + j_2}}  \frac{1}{i 2\pi 2^n (k_1-k_2)} 
\]
and
\begin{multline*}
\text{VI} := \\
\sum_{\substack{ k_1, k_2\geq 0:\\ |k_1-k_2| \geq 2}}^\infty
\sum_{j_1=0}^{2^{n}-1} \sum_{j_2=0}^{2^{n}-1} f_{k_1 2^n +j_1} \overline{f_{k_2 2^n + j_2}}  \frac{j_2 -j_1}{i 2\pi ( 2^n (k_1-k_2) + (j_1-j_2)) 2^n (k_1-k_2) } .
\end{multline*}

Term V can be estimated by
\begin{align*}
& \left|\sum_{|k_1 -k_2|\geq 2}
\sum_{j_1=0}^{2^{n}-1} \sum_{j_2=0}^{2^{n}-1} f_{k_1 2^n +j_1} \overline{f_{k_2 2^n + j_2}}  \frac{1}{i 2\pi 2^n (k_1-k_2)}
\right| = \\
& \frac{2^{-n}}{2 \pi} \left|\sum_{|k_1 -k_2|\geq 2}
\left(\sum_{j_1=0}^{2^{n}-1}  f_{k_1 2^n +j_1} \right)
\left( \sum_{j_2=0}^{2^{n}-1} \overline{f_{k_2 2^n + j_2}} \right) \frac{1}{k_1-k_2}
\right|
\lesssim \\
&  2^{-n} \sum_{k \geq 1} \left(\sum_{j=k 2^{n}}^{(k+1)2^{n}-1} |f_j| \right)^2,
\end{align*}
where we use the boundedness of the discrete Hilbert transform on $l^2$.

For term $\text{VI}$, we observe that
\[
\left|\frac{j_2 -j_1}{i 2\pi ( 2^n (k_1-k_2) + (j_1-j_2))  2^n (k_1-k_2)}
 \right| \leq 2 \cdot 2^{-n} \frac{1}{2\pi  (k_1-k_2)^2}
\]
for $|k_1-k_2| \geq 2$ and $j_1, j_2 \in \{ 0, 1, \cdots, 2^n-1\}$, which yields
\begin{align*}
 & \left|\sum_{\substack{ k_1, k_2\geq 0:\\ |k_1-k_2| \geq 2}}^\infty
\sum_{j_1=0}^{2^{n}-1} \sum_{j_2=0}^{2^{n}-1} f_{k_1 2^n +j_1} \overline{f_{k_2 2^n + j_2}}  \frac{j_2 -j_1}{i 2\pi ( 2^n (k_1-k_2) + (j_1-j_2)) 2^n (k_1-k_2)   } \right| \leq \\
&   \frac{2^{-n}}{ \pi} \sum_{|k_1 -k_2|\geq 2}
\left(\sum_{j_1=0}^{2^{n}-1}  |f_{k_1 2^n +j_1} |\right)
\left( \sum_{j_2=0}^{2^{n}-1} |f_{k_2 2^n + j_2}| \right) \frac{1}{(k_1-k_2)^2}
\lesssim \\
& 2^{-n} \sum_{k \geq 0} \left(\sum_{j=k 2^{n}}^{(k+1)2^{n}-1} |f_j| \right)^2,
\end{align*}
where we use the the boundedness of the linear operator on $\ell^2$ induced by the matrix
\[
\left(\frac{1}{(k_1-k_2)^2}\right)_{k_1, k_2 \geq 0, |k_1 -k_2|\geq 2},
\]
as can be verified e.g. by appealing to the Schur test.

Note that we have yet to estimate term IV. Instead of doing this directly, we observe that for $I = [a, a+ 2^{-n}]$, term II becomes
\[ 
\sum_{\substack{k_1, k_2 \geq 0:\\ |k_1-k_2| \geq 2}}
\sum_{j=k_1 2^{n}}^{(k_1+1)2^{n}-1} \sum_{l=k_2 2^n}^{(k_2 +1) 2^{n}-1} f_j \overline{f_l} \, 
\frac{e^{i 2 \pi (j-l)a} - e^{i 2 \pi (j-l) (a+2^{-n})}}{i 2\pi (j-l)} = \text{VII} + \text{VIII},
\]
where
\[
\text{VII} := \sum_{\substack{k_1, k_2 \geq 0:\\ |k_1-k_2| \geq 2}}
\sum_{j=k_1 2^{n}}^{(k_1+1)2^{n}-1} \sum_{l=k_2 2^n}^{(k_2 +1) 2^{n}-1} e^{i 2 \pi j a} f_j \overline{e^{i 2 \pi l a} f_l} \, 
    \frac{1}{i 2\pi (j-l)}
\]
and
\[
\text{VIII} := - \sum_{\substack{k_1, k_2 \geq 0:\\ |k_1-k_2| \geq 2}}
\sum_{j=k_1 2^{n}}^{(k_1+1)2^{n}-1} \sum_{l=k_2 2^n}^{(k_2 +1) 2^{n}-1} e^{i 2 \pi j (a + 2^{-n})} f_j \overline{e^{i 2 \pi l (a+ 2^{-n})} f_l} \, 
\frac{1}{i 2\pi (j-l)} .
\]
To handle these last expressions, note that $\text{VII}$ and $\text{VIII}$ have exactly the same form as term $\text{III}$, just with unimodular factors on the Fourier coefficients $f_j$. This observation combined with our previous estimates finishes the proof of the lemma. \end{proof}

\subsection{Proof of Theorem \ref{gen_mult_q_leq2}}
We begin with the implication $(ii) \Rightarrow (i)$ in Part (a). Proceeding along the same lines as in the proof of Theorem \ref{gen_mult_p<1}, we perform the splitting (\ref{eq:highlow}) into high and low frequencies. A simple duality argument again shows that it suffices to prove the case $q=1$. Note that in this case condition $(ii)$ in Theorem \ref{gen_mult_q_leq2}, Part (a), implies in particular condition $(ii)$ in Theorem \ref{gen_mult_p<1}, so the low frequency estimate for $K_{\eps, \lambda}$ goes through without change (the upper type of $\Psi$
is not required here). As to the high frequency estimate for 
$N_{\eps, \lambda}$, we write
\begin{align*}  
\int_I \left| N_{\eps, \lambda}(e^{i 2 \pi \theta}) - \langle N_{\eps, \lambda} \rangle_I \right|^2 \, \dd \theta
&\leq  \int_I \left| N_{\eps, \lambda}(e^{i 2 \pi \theta})  \right|^2 \, \dd \theta\\
& \lesssim |I| \sum_{k=1}^{\infty} \left( \sum_{(k+1)|I|^{-1} \geq n > k|I|^{-1}} |\lambda_n| \right)^2 \\
& \lesssim \frac{|I|^{-1}}{( \Psi^{-1}(|I|^{-1}))^2} 
\end{align*}
by condition $(ii)$, where we have now used the full power of Lemma \ref{lem:uncond1}, namely Part (\ref{arithmetic}).
  
For the reverse estimate, note that by Proposition \ref{equiv_prop}
the function
\[
b_\eps( e^{i 2 \pi \theta} ) = \sum_{k=1}^M \eps_k \sum_{n= k N}^{(k+1)N -1}  e^{i 2 \pi n \theta} |\lambda_n|
= \sum_{k=1}^M \eps_k b_k( e^{i 2 \pi  \theta} ) 
\]
is in $BMO (\rho_{\Psi}) (\D)$, uniformly 
for every choice of $\epsilon = \{ \epsilon_k \}_{k \in \N}$ with $| \epsilon_k | \leq 1$ and for all $M \in \N$, so in particular for $\eps_k \in \{-1,1\}$.
As above, we choose $I= [0, (4N)^{-1}]$. Here, 
\[
b_k (e^{i 2 \pi \theta}) := \sum_{n= k N}^{(k+1)N -1}  e^{i 2 \pi n \theta} |\lambda_n|, \quad \theta \in [0,1).
\]
An integration over the standard product probability space 
$\Omega = \{-1,1\}^\N$ and an application of Khintchine's Theorem (see, e.g., \cite{Grafakos_modern}*{Appendix C}) thus yields
\begin{align*}
\frac{1}{ \Psi^{-1}(|I|^{-1})} & \gtrsim \int_\Omega \int_I \left|  b_\eps(e^{i 2 \pi \theta}) - \langle b_\eps \rangle_I \right| \, \dd \theta \, \dd \mathbb{P} (\eps) \\
& \approx \left( \sum_{k=1}^M \left( \int_I \left|\sum_{n= k N}^{(k+1)N -1}  e^{i 2 \pi n \theta} |\lambda_n| - \langle b_k \rangle_I \right| \, \dd \theta \right)^2\right)^{1/2} \\
& = \left( \sum_{k=1}^M \left( \int_I \left|\sum_{n= k N}^{(k+1)N -1}  e^{i 2 \pi (n- kN)\theta} |\lambda_n|  -  e^{- i 2 \pi k N \theta}\langle b_k \rangle_I \right| \, \dd \theta \right)^2\right)^{1/2} \\
& \gtrsim  \left( \sum_{k=1}^M \left( \sum_{n= k N}^{(k+1)N -1}   |\lambda_n|   \right)^2\right)^{1/2} \frac{1}{N},
\end{align*}
where we used a similar argument as in the proof of (\ref{eq:osclow}) for the last 
inequality. To make this precise, we note that for any $k$,
\[
\left|\sum_{n= k N}^{(k+1)N -1}  e^{i 2 \pi n \theta} |\lambda_n| \right| =  \left|\sum_{n= k N}^{(k+1)N -1}  e^{i 2 \pi (n- kN)\theta} |\lambda_n| \right| \geq \frac{1}{\sqrt{2}} \sum_{n= k N}^{(k+1)N -1} |\lambda_n|
\]
for all $\theta \in I$. If 

\begin{equation}\label{eq:smaverage}
|\langle b_k \rangle_I| \leq \frac{1}{4} \sum_{n= k N}^{(k+1)N -1}   |\lambda_n|,
\end{equation}
this implies
\begin{equation}  \label{eq:lower1}
\int_I \left|\sum_{n= k N}^{(k+1)N -1}  e^{i 2 \pi n \theta} |\lambda_n|  - \langle b_k \rangle_I \right| \, \dd \theta \geq |I| \frac{1}{4} \sum_{n= k N}^{(k+1)N -1} |\lambda_n|.
\end{equation}
Inequality (\ref{eq:smaverage}), and therefore (\ref{eq:lower1}), hold in particular for $k \geq 4$.

If \[
|\langle b_k \rangle_I| > \frac{1}{4} \sum_{n= k N}^{(k+1)N -1}   |\lambda_n|,\]
then $k \leq 3$. Choosing 
\[
\tilde I = 
\left[0, \frac{1}{16 N} \right] \quad \text{and} \quad \dbtilde{I} = \left[0, \frac{1}{64 N} \right],
\]
we note that on $\tilde I$, all terms in $b_k$ 
only take values in the first quadrant of $\C$ and hence
\begin{align*}  %
& \int_{\tilde I} \left|\sum_{n= k N}^{(k+1)N -1}  e^{i 2 \pi n \theta} |\lambda_n|  -  \langle b_k \rangle_{\tilde I} \right| \,  \dd \theta  \\ \geq
&  \int_{\tilde I} \left|\sum_{n= k N}^{(k+1)N -1}  \sin(2 \pi n \theta) |\lambda_n|  -  \langle \sin(2 \pi n \theta) \rangle_{\tilde I} |\lambda_n| \right| \, \dd \theta  \\ \geq
&  \int_{\dbtilde I} \left|\sum_{n= k N}^{(k+1)N -1} \left( \langle   \sin(2 \pi n \theta) \rangle_{\tilde I} - \sin(2 \pi n \theta)\right) |\lambda_n| \right| \, \dd \theta  \\ \gtrsim
&  \frac{1}{N}   \sum_{n= k N}^{(k+1)N -1} \frac{n}{N}  |\lambda_n|  \\ \gtrsim
&  |I|  \sum_{n= k N}^{(k+1)N -1}  |\lambda_n|,
\end{align*}
where we apply the estimate (\ref{eq:sinestimate}) in the penultimate inequality.
\qed

\section{Theorem \ref{LP_BMO_Psi} and its consequences}\label{BMOsection}  
The proof of Theorem \ref{LP_BMO_Psi} can be obtained by suitably adapting known proofs in the literature for the BMO-case, i.e. for the case $r=0$;  for instance,  one can adapt the corresponding arguments in the papers of S. V. Bochkarev \cite{B} and I. Vasilyev and A. Tselishchev \cite{V-T}. Regarding the implication $(1) \implies (2)$, by adapting an argument due to Bourgain \cite{B_SF} one can actually show that a stronger version of $(1) \implies (2)$ holds true; see Theorem \ref{reverse_arbitrary} below). Note that in the statement of Theorem \ref{reverse_arbitrary}, if $K$ is a non-empty interval then $W(K)$ and $\mathbf{K}^{\rm centre}$ are as in \S \ref{LP_part}

\begin{theo}\label{reverse_arbitrary} Let $r \geq 0$ be a given exponent.

Let $\mathcal{K}$ be an arbitrary collection of mutually disjoint intervals on the real line such that $|K| = 3 \cdot 2^{l_K+1}$ for some $l_K \in \N_0$ and $0 \notin K$ for all $K \in \mathcal{K}$. Then 
\begin{multline*}
\left( \sup_{\substack{I \subseteq \T: \\ I \ {\rm arc}}} \frac{1} { \left[ \rho_{\Psi_{r,1}} (|I|) \right]^2 |I|} \int_I \sum_{K \in \mathcal{K}}  \sum_{\substack{J \in W (K) \setminus \mathbf{K}^{\rm centre} : \\ |J| > |I|^{-1}}} \left| \Delta_J (u) (e^{i 2 \pi \theta}) \right|^2 \, \dd \theta \right)^{1/2} \lesssim_r  \\
\norm{ u }_{BMO (\rho_{\Psi_{r,1}}) (\T)}
\end{multline*}
for all $u \in BMO (\rho_{\Psi_{r,1}}) (\T)$, where the implied constant depends only on $r$ and not on $\mathcal{K}$ and $u$. 
\end{theo}

As mentioned above, Theorems \ref{LP_BMO_Psi} and \ref{reverse_arbitrary} can be obtained by modifying known arguments in the literature. We shall briefly present how this can be done for the convenience of the reader. %
More specifically, in \S \ref{rev_arb} we outline how one can adapt an argument in \cite{B_SF}  to establish Theorem \ref{reverse_arbitrary} and in \S \ref{2_implies_1} we briefly present how a modification of the argument in \cite{V-T} for the corresponding BMO-case can be used to establish the implication $(2) \implies (1)$ in Theorem \ref{LP_BMO_Psi}. Apart from the reason of making the presentation to a certain extent self-contained, we also include \S\ref{rev_arb} and \S \ref{2_implies_1} in order to make transparent in which parts of the proofs one uses the assumption that the growth function is of the form $\Psi = \Psi_{r,1}$.

In \S \ref{L-P_ineq_duality}, we present some consequences of Theorem \ref{LP_BMO_Psi}.

\subsection{Proof of Theorem \ref{reverse_arbitrary}}\label{rev_arb}

By translation invariance, to prove Theorem \ref{reverse_arbitrary} it suffices to show that
\begin{multline}\label{main_rev_ineq}
\frac{1} { \left[ \rho_{\Psi_{r,1}} (|I|) \right]^2 |I|} \int_I \sum_{K \in \mathcal{K}}   \sum_{\substack{J \in W (K) \setminus \mathbf{K}^{\rm centre} : \\ |J| > |I|^{-1}}} \left| \Delta_J (u) (e^{i 2 \pi \theta}) \right|^2 \, \dd \theta \lesssim_r \\
\norm{ u }^2_{BMO (\rho_{\Psi_{r,1}}) (\T)},
\end{multline}
for any fixed interval of the form $I = [-\delta/2, \delta/2) $, $\delta \in (0,1] $. Since any de la Vall'ee Poussin-type kernel can be written as a linear combination of two modulated Fej\'er-type kernels (see \S \ref{LP_part}), \eqref{main_rev_ineq} is a direct consequence of the following lemma, which is a variant of \cite{B_SF}*{lemma 2}.

\begin{lemma}\label{reverse_arbitrary_Fejer}
Let $r \geq 0$ be a given exponent. For $\delta \in (0,1] $, consider the interval $I_{\delta} : = [-\delta/2, \delta/2) $. 

Let $\{ F_{\alpha} \}_{\in \mathcal{A}}$ be a finite collection of trigonometric polynomials of the form
\[
F_{\alpha} (e^{i 2 \pi \theta} ) = e^{i 2 \pi c_{\alpha} \theta} K_{N_{\alpha}} (e^{i 2 \pi \theta}), \quad \theta \in [-1/2, 1/2),
\]
where $N_{\alpha} > \delta^{-1}$ for all $\alpha \in \mathcal{A}$ and moreover, the frequency supports of $F_{\alpha}$ and $F_{\alpha'}$ are disjoint for $\alpha \neq \alpha'$ and $0 \notin \mathrm{supp} (\widehat{F_{\alpha}})$ for all $\alpha \in \mathcal{A}$.   

Then 
\[
\frac{1} { \left[ \rho_{\Psi_{r,1}} (|I_{\delta}|) \right]^2 |I_{\delta}|} \int_{I_{\delta}} \sum_{\alpha \in \mathcal{A}} \left| ( F_{\alpha} \ast u) (e^{i 2 \pi \theta}) \right|^2 \, \dd \theta \lesssim_r \norm{ u }^2_{BMO (\rho_{\Psi_{r,1}}) (\T)}
\]
for all $u \in BMO (\rho_{\Psi_{r,1}}) (\T)$, where the implied constant depends only on $r$ and not on $\delta$, $\mathcal{A}$, and $u$. 
\end{lemma}

\begin{proof} Let $u$ be a given function in $ BMO (\rho_{\Psi_{r,1}}) (\T)$. If
\[
\omega_{\delta} ( \theta) : = \frac{\delta^2}{\delta^2 + \sin^2 ( \pi \theta )}, \quad \theta \in [-1/2,1/2),
\]
then \cite{B_SF}*{(6)} (in the proof of \cite{B_SF}*{lemma 2}) applied to $g= u - \langle u \rangle_{I_{\delta}}$ asserts that there exists an absolute constant $C>0$ such that
\begin{multline}\label{est_Bourgain}
\int_{[-1/2,1/2)} \left( \sum_{\alpha \in \mathcal{A}} \left| ( F_{\alpha} \ast u) (e^{i 2 \pi \theta}) \right|^2 \right) \omega_{\delta} ( \theta )\, \dd \theta \\ \leq
C \int_{[-1/2,1/2)} | u ( e^{i 2 \pi \theta} ) - \langle u \rangle_{I_{\delta}} |^2 \omega_{\delta} ( \theta) \, \dd \theta ,
\end{multline}
where we also used the assumption that $0 \notin \mathrm{supp} (\widehat{F_{\alpha}})$ for all $\alpha \in \mathcal{A}$. Notice that \cite{B_SF}*{(6)} is a consequence of \cite{B_SF}*{Lemma 1}, which is a Cotlar-type almost orthogonality result. 

To prove the desired estimate, note that since $\omega_{\delta} ( \theta) \geq 1/( 2 \pi^2 ) $ for all $\theta \in I_{\delta}$, \eqref{est_Bourgain} yields
\begin{multline}\label{est_omega}
\frac{1} { \left[ \rho_{\Psi_{r,1}} (|I_{\delta}|) \right]^2 |I_{\delta}|} \int_{I_{\delta}} \sum_{\alpha \in \mathcal{A}} \left| ( F_{\alpha} \ast u) (e^{i 2 \pi \theta}) \right|^2 \dd \theta \lesssim \\ 
\frac{1} { \left[ \rho_{\Psi_{r,1}} (|I_{\delta}|) \right]^2 |I_{\delta}|}  \int_{[-1/2,1/2)} | u ( e^{i 2 \pi \theta} ) - \langle u \rangle_{I_{\delta}} |^2 \omega_{\delta} ( \theta ) \, \dd \theta .
\end{multline}
Hence, in view of \eqref{est_omega}, it suffices to show that
\begin{equation}
\label{u_omega_est}
\frac{1} { \left[ \rho_{\Psi_{r,1}} (|I_{\delta}|) \right]^2 |I_{\delta}|} \int_{[-1/2,1/2)} | u ( e^{i 2 \pi \theta} ) - \langle u \rangle_{I_{\delta}} |^2 \omega_{\delta} ( \theta ) \, \dd \theta  \lesssim %
\norm{ u }^2_{BMO (\rho_{\Psi_{r,1}}) (\T)}. 
\end{equation}
As $\omega_{\delta} (\theta) \leq 1 $ for all $\theta \in [-1/2,1/2)$, we have
\begin{equation}\label{I_u_omega_est}
\frac{1} { \left[ \rho_{\Psi_{r,1}} (|I_{\delta}|) \right]^2 |I_{\delta}|} \int_{I_{\delta}} | u (e^{i 2 \pi \theta}) - \langle u \rangle_{I_{\delta}} |^2 \omega_{\delta} ( \theta ) \, \dd \theta \lesssim \norm{ u }^2_{BMO (\rho_{\Psi_{r,1}}) (\T)} .
\end{equation}

To handle the contribution of the term involving the part of the integral on the left-hand side of \eqref{u_omega_est} over $[-1/2, 1/2) \setminus I_{\delta}$, note that as $\sin (t)/t \geq 2/\pi $ for $t \in [-\pi/2, \pi/2] \setminus \{0 \}$ we have
\begin{align*}
& \frac{1} { \left[ \rho_{\Psi_{r,1}} (|I_{\delta}|) \right]^2 |I_{\delta}|} \int_{[- 1/2, - \delta/2) \cup [\delta/2, 1/2)} | u ( e^{i 2 \pi \theta} ) - \langle u \rangle_{I_{\delta}} |^2 \omega_{\delta} ( \theta ) \, \dd \theta \lesssim \\ 
& \frac{\log^{2r} ( \delta^{-1} )} { \delta } \sum_{\substack{k \in \N_0 :\\ 2^k \leq \delta^{-1} }} \frac{1} { 2^{2k} } \int_{ 2^{k-1} \delta \leq |\theta| \leq 2^k \delta } | u ( e^{i 2 \pi \theta} ) - \langle u \rangle_{I_{\delta}} |^2 \, \dd \theta .
\end{align*}
Since $2^k \leq \delta^{-1}$, we have
\[
\log \left( \delta^{-1} \right) \leq k
\log \left( ( 2^{k-1} \delta)^{-1} \right) 
\]
and so,
\begin{align*}
& \frac{\log^{2r} ( \delta^{-1} )} {  \delta } \sum_{\substack{k \in \N_0 :\\ 2^k \leq \delta^{-1} }}  \frac{1} { 2^{2k} } \int_{ 2^{k-1} \delta \leq |\theta|  \leq 2^k \delta } | u ( e^{i 2 \pi \theta} ) - \langle u \rangle_{I_{\delta}} |^2 \, \dd \theta \lesssim_r \\ 
& \sum_{\substack{k \in \N_0 :\\ 2^k \leq \delta^{-1} }} \frac{k^{2r}} { 2^{k} } \left( \frac{\log^{2r} \left( (2^k \delta)^{-1} \right)} {  2^k \delta } \int_{ 2^{k-1} \delta \leq |\theta| \leq 2^k \delta } | u ( e^{i 2 \pi \theta} ) - \langle u \rangle_{I_{\delta}} |^2  \, \dd \theta \right) .
\end{align*}
A standard argument now yields
\begin{multline}\label{tr_ineq_BMO}
\left( \frac{\log^{2r} \left( (2^k \delta)^{-1} \right)} { 2^k \delta } \int_{ 2^{k-1} \delta \leq |\theta| \leq 2^k \delta } | u ( e^{i 2 \pi \theta} ) - \langle u \rangle_{I_{\delta}} |^2 \, \dd \theta \right)^{1/2} \lesssim  \\
k \norm{ u }_{BMO (\rho_{\Psi_{r,1}}) (\T)}.
\end{multline}
Indeed, it follows from the triangle inequality that
\begin{align*} 
\left( \frac{\log^{2r} \left( (2^k \delta)^{-1} \right)} { 2^k \delta } \int_{ 2^{k-1} \delta \leq |\theta| \leq 2^k \delta } | u ( e^{i 2 \pi \theta} ) - \langle u \rangle_{I_{\delta}} |^2 \, \dd \theta \right)^{1/2}  \lesssim_r \\ 
\sum_{l=0}^k \left( \frac{\log^{2r} \left( (2^k \delta)^{-1} \right)} { 2^l \delta } \int_{ 2^l I_{\delta} } | u ( e^{i 2 \pi \theta} ) - \langle u \rangle_{2^{l+1} I_{\delta}} |^2 \, \dd \theta \right)^{1/2}
\end{align*}
where $2^l I_{\delta} = [- 2^{l-1} \delta, 2^{l-1} \delta)$. Using \eqref{tr_ineq_BMO} we get
\begin{align*}
&\frac{1} { \left[ \rho_{\Psi_{r,1}} (|I_{\delta}|) \right]^2 |I_{\delta}|} \int_{[- 1/2, - \delta/2) \cup [\delta/2, 1/2)} | u ( e^{i 2 \pi \theta} ) - \langle u \rangle_{I_{\delta}} |^2 \omega_{\delta} ( \theta ) \, \dd \theta  \lesssim_r  \\ 
& \norm{ u }^2_{BMO (\rho_{\Psi_{r,1}}) (\T)} \sum_{k \in \N_0 }  \frac{k^{2r+2}} { 2^{2k} }
\end{align*}
and so,
\begin{multline}\label{II_u_omega_est}
\frac{1} { \left[ \rho_{\Psi_{r,1}} (|I_{\delta}|) \right]^2 |I_{\delta}|} \int_{ [- 1/2, - \delta/2) \cup [\delta/2, 1/2) } | u ( e^{i 2 \pi \theta} ) - \langle u \rangle_{I_{\delta}} |^2 \omega_{\delta} ( \theta ) \, \dd \theta  \lesssim_r \\ \norm{ u }^2_{BMO (\rho_{\Psi_{r,1}}) (\T)} ,
\end{multline}
as desired.

It thus follows from \eqref{I_u_omega_est} and \eqref{II_u_omega_est} that \eqref{u_omega_est} holds true and hence, the proof of the lemma is complete.
\end{proof}

\subsection{Proof of that (2) implies (1) in Theorem \ref{LP_BMO_Psi}}\label{2_implies_1}
The proof that we present below is an adaptation of the corresponding argument in \cite{V-T} for the $BMO$-case.

Let $u \in L^2 (\T)$ be such that (2) in Theorem \ref{LP_BMO_Psi} holds. It suffices to prove that for any fixed arc $I$ in $\T$ one has
\begin{equation}\label{main_est_1_implies_2_I}
M_I \left( \int_I \abs{ u (e^{i 2 \pi \theta}) - \langle u \rangle_I }^2 \, \dd \theta \right)^{1/2} \lesssim_r \norm{u}_{\Psi_{r,1}, \ast} ,
\end{equation}
where  $M_I := |I|^{-1/2} \log^r ( |I|^{-1} )$ and $\norm{u}_{\Psi_{r,1}, \ast}$ is as in the statement of Theorem \ref{LP_BMO_Psi}. Note that for any $c \in \C$ one has
\[ 
\left( \int_I \abs{ u (e^{i 2 \pi \theta}) - \langle u \rangle_I }^2 \, \dd \theta \right)^{1/2} \leq  2 \left( \int_I \abs{ u (e^{i 2 \pi \theta}) - c }^2 \, \dd \theta \right)^{1/2} .
\]
Hence, if we choose
\[
c:= \sum_{ 2^n \leq | I |^{-1} }  \Delta_n (u) (e^{i 2 \pi c_I}),
\]
with $c_I$ being the centre of $I$, then it suffices to prove that
\begin{equation}\label{U_final}
U  \lesssim \norm{u}_{\Psi_{r,1}, \ast} \quad \text{for } U = A \text{ and } U=B,
\end{equation}
where
\[
A:= M_I \left( \int_I \abs{ \sum_{ 2^n > | I |^{-1} } \Delta_n (u) (e^{i 2 \pi \theta}) }^2 \, \dd \theta \right)^{1/2} 
\]
and
\[
B:= M_I \left( \int_I \abs{ \sum_{ 2^n \leq | I |^{-1} }  \int_{\T} u (e^{i 2 \pi \sigma}) \left[  \delta_n \left( e^{i 2 \pi (\theta - \sigma ) } \right) - \delta_n \left(  e^{i 2 \pi (c_I - \sigma) } \right)  \, \dd \sigma \right] }^2 \, \dd \theta \right)^{1/2} .
\]

Note that if, for $n \in \N$, we define $ \widetilde{\delta}_n := \delta_{n-1}+\delta_n + \delta_{n+1} $, then  $\Delta_n = \widetilde{\Delta}_n \Delta_n $. Here, for $f \in L^1 (\T)$, $ \widetilde{\Delta}_n  (f) := \widetilde{\delta}_n \ast f $. In what follows, we shall use several times that  $\int_K |\widetilde{\Delta}_n  (u) |^2 \lesssim 2^{-n} n^{-2r} \norm{u}_{\Psi_{r,1}, \ast}$ for any arc $K$ with $|K| \approx2^{-n}$  (cf. \cite{V-T}*{Lemma 3.1}). 

We may assume, without loss of generality, that $I$ is of the form $[-|I|/2,|I|/2)$.

\subsubsection{Proof of \eqref{U_final} for $U=A$}
By the triangle inequality, we have $A \leq A_1 + A_2$,
where
\[
A_1 := M_I \left( \int_I \abs{ \sum_{ 2^n > | I |^{-1} }  \int_{\T \setminus I} \widetilde{\Delta}_n (u) (e^{i 2 \pi \sigma}) \delta_n \left( e^{i 2 \pi (\theta - \sigma )  }   \right) \, \dd \sigma  }^2  \, \dd \theta \right)^{1/2} 
\]
and
\[
A_2 := M_I \left( \int_I \abs{  \sum_{ 2^n > | I |^{-1} } \int_I \widetilde{\Delta}_n (u) (e^{i 2 \pi \sigma}) \delta_n \left( e^{i 2 \pi (\theta - \sigma )  } \right)   \, \dd \sigma  }^2 \, \dd \theta \right)^{1/2} .
\]
\subsubsection*{Estimate for $A_1$}
For $n \in \N$, let $I_n^-$ denote the interval that is concentric with $I$ and has length $| I_n^- | = | I | - 2^{-n}$. We then have $ A_1 \leq A_{1,1} + A_{1,2}$, 
where
\[
A_{1,1} := M_I \sum_{ 2^n > | I |^{-1} } \left( \int_{I_n^-} \left(   \int_{\T \setminus I} \abs{ \widetilde{\Delta}_n (u) (e^{i 2 \pi \sigma})  } \cdot \abs{ \delta_n \left( e^{i 2 \pi (\theta - \sigma )  } \right) }  \, \dd \sigma  \right)^2 \, \dd \theta \right)^{1/2} 
\]
and
\[
A_{1,2} := M_I \sum_{ 2^n > | I |^{-1} } \left( \int_{ I \setminus I_n^-} \left( \int_{\T \setminus I} \abs{ \widetilde{\Delta}_n (u) (e^{i 2 \pi \sigma}) } \cdot \abs{ \delta_n \left( e^{i 2 \pi (\theta - \sigma )  } \right) }  \, \dd \sigma  \right)^2 \, \dd \theta \right)^{1/2} .
\]
The terms $A_{1,1}$ and $A_{1,2}$ will be estimated separately.

\subsubsection*{Estimate for $A_{1,1}$}
Without loss of generality we may assume that $| I | = 2^{-\zeta_0}$ for some $\zeta_0 \in \N_0$. Hence, we may write $ I_n^- = \bigcup_{Q \in \mathcal{Q}_n} Q$ and $ \T \setminus I = \bigcup_{P \in \mathcal{P}_n} P$, where $\mathcal{Q}_n$ and $\mathcal{P}_n$ are families of mutually disjoint arcs of length $2^{-n-2}$. Using the Cauchy--Schwarz inequality, 
we have 
\[
A_{1,1}  \lesssim  
M_I  \sum_{ 2^n > | I |^{-1} }  \left( K_n  \sum_{Q \in \mathcal{Q}_n } \int_Q  \left( \sum_{P \in \mathcal{P}_n } \int_P \frac{ \abs{ \widetilde{\Delta}_n (u) (e^{i 2 \pi \sigma}) }^2 }{ \sin^2 (   \pi (\theta-\sigma) ) } \, \dd \sigma \right)   \, \dd \theta \right)^{1/2} ,
\] 
where
\[
K_n:=  \sum_{P \in \mathcal{P}_n }  \int_P \sin^2 (   \pi (\theta-\sigma) ) \abs{ \delta_n \left( e^{i 2 \pi (\theta - \sigma )} \right) }^2 \, \dd \sigma .
\]
As $\phi \mapsto \sin^2 (   \pi \phi )$ is $1$-periodic and $|\sin (t)| \leq t$ for all $t \in \R$, we have $
K_n \lesssim a_n + b_n $, 
where
\[
a_n := \int_{ | \phi | \leq 2^{-n} }  \phi^2  \abs{ \delta_n \left( e^{i 2 \pi \phi } \right) }^2 \, \dd \phi   \quad \text{and} \quad
b_n :=  \int_{ 2^{-n} <  | \phi | \leq  1/2 }  \phi^2  \abs{ \delta_n \left( e^{i 2 \pi \phi } \right) }^2 \, \dd \phi  .
\]
Since $\norm{\delta_n}_{L^{\infty}(\T)} \lesssim 2^n $, we have
$ 
a_n \lesssim  2^{-n}$.
To estimate $b_n$, note that by the properties of the Fej\'er kernel we have $ \abs{ \delta_n \left( e^{i 2 \pi \phi } \right) } \lesssim 2^{-n}  \phi^{-2} $ for all $2^{-n} <  | \phi | \leq  1/2$, 
which readily implies that $b_n \lesssim  2^{-n}$. We thus have
\[
A_{1,1} \lesssim
M_I \sum_{ 2^n > | I |^{-1} }  \left( \sum_{Q \in \mathcal{Q}_n }   \int_Q  \left( \sum_{P \in \mathcal{P}_n } 2^{-n} \int_P \frac{ \abs{ \widetilde{\Delta}_n (u) (e^{i 2 \pi \sigma}) }^2 }{ \sin^2 (   \pi (\theta-\sigma) ) } \, \dd \sigma \right) \, \dd \theta \right)^{1/2} 
\]
and so,
\begin{align*}
A_{1,1} & \lesssim M_I \sum_{ 2^n > | I |^{-1} } \left( \sum_{Q \in \mathcal{Q}_n } \int_Q \left( \sum_{P \in \mathcal{P}_n } \frac{2^{-n}}{ \mathrm{dist}(p,q)^2}  \int_P \abs{ \widetilde{\Delta}_n (u) (e^{i 2 \pi \sigma}) }^2  \, \dd \sigma \right) \, \dd \theta \right)^{1/2} \\
& \lesssim M_I \norm{u}_{\Psi_{r,1}, \ast}  \sum_{ 2^n > | I |^{-1} } \frac{2^{-n/2}}{n^r}  \left( \sum_{Q \in \mathcal{Q}_n } \int_Q  \sum_{P \in \mathcal{P}_n} \frac{2^{-n}}{\mathrm{dist}(p,q)^2} \, \dd \theta \right)^{1/2} \\
& \lesssim M_I \norm{u}_{\Psi_{r,1}, \ast}  \sum_{ 2^n > | I |^{-1} }  \frac{2^{-n/2}}{n^r}  \left(  \int_{ [0, | I |/2 - 2^{-n-1} ) } \left( \int_{ [ | I |/2, 1) }  \frac{ 1 }{  | \theta-\sigma |^2 } \, \dd \sigma \right)  \, \dd \theta \right)^{1/2} .
\end{align*}
It can easily be seen that the sum in the last expression is bounded by $C M_I^{-1}$, with $C$ being a constant independent of $I$. Hence, $A_{1,1} \lesssim   \norm{u}_{\Psi_{r,1}, \ast}$.

\subsubsection*{Estimate for $A_{1,2}$}
To estimate $A_{1,2}$, we have $A_{1,2} \leq A_{1,2,a} + A_{1,2,b}$, where
\[ A_{1,2,a}  :=  
M_I \sum_{ 2^n > | I |^{-1} } \left( \int_{ I \setminus I_n^-} \left( \int_{\T \setminus I_n^+} \abs{ \widetilde{\Delta}_n (u) (e^{i 2 \pi \sigma}) } \cdot \abs{ \delta_n \left( e^{i 2 \pi (\theta - \sigma )  } \right) }  \, \dd \sigma  \right)^2 \, \dd \theta \right)^{1/2}
\]
and
\[
A_{1,2,b}  := 
M_I \sum_{ 2^n > | I |^{-1} } \left( \int_{ I \setminus I_n^-} \left( \int_{I_n^+ \setminus I} \abs{ \widetilde{\Delta}_n (u) (e^{i 2 \pi \sigma}) } \cdot \abs{ \delta_n \left( e^{i 2 \pi (\theta - \sigma )  } \right) }  \, \dd \sigma  \right)^2 \, \dd \theta \right)^{1/2} .
\]
Here, for $n \in \N$, $I_n^+$ denotes the interval that is concentric with $I$ and has length $| I_n^+ | =  | I | + 2^{-n}$. Arguing as in the case of $A_{1,1}$, one has
$A_{1,2,a} \lesssim \norm{u}_{\Psi_{r,1}, \ast} $. To estimate $A_{1,2,b} $, we use the Cauchy--Schwarz inequality, as well as that $\norm{\delta_n}_{L^{\infty}(\T)} \lesssim  2^n$ and $| I \setminus I_n^-| = | I_n^+ \setminus I  | = 2^{-n-1} $ to deduce that
\[
A_{1,2,b}  \lesssim  M_I \sum_{ 2^n > | I |^{-1} }   \left( \int_{ I_n^+ \setminus I } \abs{ \widetilde{\Delta}_n (u) (e^{i 2 \pi \sigma}) }^2 \, \dd \sigma  \right)^{1/2} \lesssim  M_I  \norm{u}_{\Psi_{r,1}, \ast}  \sum_{ 2^n > | I |^{-1} }  \frac{ 2^{-n/2} } {n^r}.
\]
So, $
A_{1,2,b} \lesssim   \norm{u}_{\Psi_{r,1}, \ast} $. 
Hence, combining the estimates for $A_{1,2,a}$ and $A_{1,2,b}$, together with the estimate for $A_{1,1}$, we conclude that $A_1 \lesssim   \norm{u}_{\Psi_{r,1}, \ast}$.

\subsubsection*{Estimate for $A_2$}
Since the maps
\[ \theta \mapsto \int_I \widetilde{\Delta}_{n_1} (u) (e^{i 2 \pi \sigma}) \delta_{n_1} (e^{i 2 \pi ( \theta - \sigma ) }) \, \dd \sigma \  \text{ and } \ \theta \mapsto \int_I \widetilde{\Delta}_{n_2} (u) (e^{i 2 \pi \sigma}) \delta_{n_2} (e^{i 2 \pi ( \theta - \sigma ) }) \, \dd \sigma \]
are orthogonal in $L^2 (\T)$ for $n_1, n_2 \in \N$ with $| n_1 - n_2  | \geq 2$, we have $ A_2 \lesssim A_{2,1} + A_{2,2} $, where
\[
A_{2,1} := M_I \left( \sum_{ 2^n > | I |^{-1} }  \int_{\T \setminus I} \abs{ \int_I  \widetilde{\Delta}_n (u) (e^{i 2 \pi \sigma}) \delta_n \left( e^{i 2 \pi (\theta - \sigma )  } \right)   \, \dd \sigma  }^2 \, \dd \theta \right)^{1/2} 
\]
and
\[
A_{2,2} := M_I \left( \sum_{ 2^n > | I |^{-1} }  \int_I \abs{ \int_I  \widetilde{\Delta}_n (u) (e^{i 2 \pi \sigma}) \delta_n \left( e^{i 2 \pi (\theta - \sigma )  } \right)   \, \dd \sigma  }^2 \, \dd \theta \right)^{1/2} .
\]
Arguing as in the case of  $A_1$, one shows that $A_{2,1} \lesssim \norm{u}_{\Psi_{r,1}, \ast}$.  

We shall now focus on $A_{2,2}$; by the triangle inequality, we have $
A_{2,2} \leq A_{2,2,a} + A_{2,2,b} $,
where
\[
 A_{2,2,a} := 
M_I \left( \sum_{ 2^n > | I |^{-1} }  \int_I \abs{ \int_{\T}  \widetilde{\Delta}_n (u) (e^{i 2 \pi \sigma}) \delta_n \left( e^{i 2 \pi (\theta - \sigma ) } \right) \, \dd \sigma  }^2 \, \dd \theta \right)^{1/2}
\]
and
\[
A_{2,2,b} := 
M_I \left( \sum_{ 2^n > | I |^{-1} }  \int_I \abs{ \int_{\T \setminus I}  \widetilde{\Delta}_n (u) (e^{i 2 \pi \sigma}) \delta_n \left( e^{i 2 \pi (\theta - \sigma )  } \right)   \, \dd \sigma  }^2 \, \dd \theta \right)^{1/2}. 
\]
Arguing as in the case of  $A_1$, one proves that $A_{2,2,b} \lesssim \norm{u}_{\Psi_{r,1}, \ast} $. 
The desired estimate for $A_{2,2,a}$, i.e. $A_{2,2,a} \lesssim \norm{u}_{\Psi_{r,1}, \ast}$, follows directly from the definition of $\norm{u}_{\Psi_{r,1}, \ast}$.  
It thus follows that $A_2 \lesssim \norm{u}_{\Psi_{r,1}, \ast}$ which, together with the corresponding estimate for $A_1$, obtained above, completes the proof of \eqref{U_final} for $U=A$.  

\subsubsection{Proof of \eqref{U_final} for $U=B$} 

If $ \mathcal{G}_n $ denotes a decomposition of $\T$ into mutually disjoint arcs of length $2^{-n-1}$ then, by using the Cauchy--Schwarz inequality and the definition of $\norm{u}_{\Psi_{r,1}, \ast}$, we have
\[
B \lesssim  M_I | I |^{1/2} \norm{u}_{\Psi_{r,1}, \ast} \sum_{ 2^n \leq | I |^{-1} }  \frac{2^{-n}}{n^r} \sum_{G \in \mathcal{G}_n} \max_{\substack{ \theta \in I \\ \sigma \in G}} \abs{ \delta_n \left( e^{i 2 \pi (\theta - \sigma ) } \right) - \delta_n \left( e^{i 2 \pi (c_I - \sigma) } \right) } .
\] 
Using the mean value theorem and the definition of $\delta_n$, we have
\[
\sum_{G \in \mathcal{G}_n} \max_{\substack{ \theta \in I \\ \sigma \in G}} \abs{ \delta_n \left( e^{i 2 \pi (\theta - \sigma ) } \right) - \delta_n \left( e^{i 2 \pi (c_I - \sigma) } \right) } \lesssim \sum_{G \in \mathcal{G}_n}    \frac{2^{2n}  | I |}{ \left[ 1 + 2^n \mathrm{dist} (I , G) \right]^2 } = S_n + T_n,
\] 
where
\[
S_n := \sum_{\substack{ G \in \mathcal{G}_n: \\ \mathrm{dist} (G, I) \leq 2^{-n}}}\frac{2^{2n}  | I |}{ \left[ 1 + 2^n \mathrm{dist} (I , G) \right]^2 }      \ \text{ and } \  T_n := \sum_{\substack{G \in \mathcal{G}_n: \\ \mathrm{dist} (G, I) \geq 2^{-n} } }   \frac{2^{2n}  | I |}{ \left[ 1 + 2^n \mathrm{dist} (I , G) \right]^2 }  .
\]
Since $2^n \leq | I |^{-1}$, there are $O(1)$ intervals $G \in \mathcal{G}_n$ with $ \mathrm{dist} (G, I) \leq 2^{-n}$ and so, $ S_n \approx 2^{2n} |I| $. To handle $T_n$, we have 
\[
T_n \approx 2^{2n} |I| \sum_{\substack{G \in \mathcal{G}_n: \\ \mathrm{dist} (G, I) \geq 2^{-n} } }  \frac{1}{  \left[ 2^n \mathrm{dist} (I , G) \right]^2 }  \lesssim 2^{2n} |I| \sum_{k=1}^{2^n} \frac{1}{k^2} \lesssim 2^{2n} |I| .
\]
Hence, \eqref{U_final} for $U=B$ holds, as
\[
B \lesssim M_I | I |^{3/2} \norm{u}_{\Psi_{r,1}, \ast} \sum_{ 2^n \leq | I |^{-1} }  \frac{2^n}{n^r}   \lesssim M_I | I |^{3/2} \norm{u}_{\Psi_{r,1}, \ast}  \frac{ | I |^{-1}}{\log^r ( | I |^{-1})} \approx \norm{u}_{\Psi_{r,1}, \ast}  .
\] 

\subsection{Consequences of Theorem \ref{LP_BMO_Psi}}\label{L-P_ineq_duality}

For $r \geq 0$, consider the Triebel--Lizorkin-type space 
\[
S^{0,r}_{\infty,2} (\T) := %
\left\{ f \in \mathcal{D}' (\T): \sup_{\substack{I \subseteq \T:\\ I \text{ arc }}} \frac{1}{|I|} \int_I \sum_{\substack{ n \in \N_0 : \\ 2^n \geq |I|^{-1} } } \left( n +1 \right)^{2r} |  \Delta_n (f) (e^{i 2 \pi \theta}) |^2 \, \dd \theta  < \infty \right\}. 
\]

Note that if $f \in S^{0,r}_{\infty,2} (\T)$ then it follows from the definition of $S^{0,r}_{\infty,2} (\T)$ that its classical Littlewood--Paley square function is in $L^2 (\T)$ and hence, $f$ can be identified with an $L^2$-function. 

Theorem \ref{LP_BMO_Psi} immediately implies the following inclusion.

\begin{theo}\label{inclusion_infty}
Let $r \geq 0$ be a given exponent.

One has the inclusion
\begin{equation}\label{inclusion_T-L_infty}
S^{0,r}_{\infty,2} (\T) \subseteq BMO (\rho_{\Psi_{r,1}}) (\T). 
\end{equation}
\end{theo}

\begin{rmk}\label{proper_inclusion_infty}
If $r=0$, then \eqref{inclusion_T-L_infty} holds as an equality; see \cite{B}, \cite{FJ}, \cite{V-T}.
If $r >0$, the inclusion \eqref{inclusion_T-L_infty} is proper. To see this, consider the function $u$ whose Fourier series is given by
 \begin{equation}\label{BMO_psi_lac_ex}
u (e^{i 2\pi \theta}) \sim \sum_{j=0}^{\infty} \frac{1}{j^{2r +1/2}} e^{i 2\pi 2^j \theta}, \quad \theta \in [0,1).
\end{equation}
Then 
\[
u \in 
BMO (\rho_{\Psi_{r,1}}) (\T) \setminus S^{0,r}_{\infty,2} (\T) .
\]
\end{rmk}

Note that 
\begin{equation}\label{duality_T-L}
S^{0,r}_{\infty,2} (\T) \cong ( S^{0,-r}_{1,2} (\T) )^{\ast}.
\end{equation}
For a proof of a Euclidean analogue of \eqref{duality_T-L} see \S \ref{s_log_smooth} below. The Triebel--Lizorkin-type space $S^{0,-r}_{1,2} (\T)$ appearing in \eqref{duality_T-L}  can be defined as the class of all $f \in \mathcal{D}' (\T)$ such that 
\[
\left( \sum_{n \in \N_0} \frac{ \left| \Delta_n (f) \right|^2}{\left( n +1 \right)^{2r}} \right)^{1/2} \in L^1 (\T).
\]
By combining Theorem \ref{inclusion_infty} and \eqref{duality_T-L} we obtain the following inclusion.

\begin{theo}\label{LP_Psi_r_predual_Thm} Let $r \geq 0$ be a given exponent.

There exists a constant $C_r>0$ such that
\begin{equation}\label{LP_H_ineq_torus}
\norm{ \left( \sum_{n \in \N_0} \frac{ \left| \Delta_n (f) \right|^2}{\left( n +1 \right)^{2r}} \right)^{1/2}  }_{L^1 (\T)} \leq C_r \norm{ f }_{H^{\Psi_{r,1}} (\T)}. 
\end{equation}
\end{theo}

Theorem \ref{LP_Psi_r_predual_Thm} can also be proved directly by using the atomic decomposition of  $H^{\Psi_{r,1}} (\T)$. We present such an approach in the corresponding Euclidean case; see \S \ref{LP_Psi_r_predual} below.

For $r >0$, in view of Remark \ref{proper_inclusion_infty}, the inclusion
\begin{equation}\label{inclusion_T-L_1}
H^{\Psi_{r,1}} (\T) \subseteq S^{0,-r}_{1,2} (\T)
\end{equation}
is proper. However, for $r=0$, we have $H^{\Psi_{0,1}} (\T) = S^{0,0}_{1,2} (\T)$, which is Stein's classical square function characterisation of the Hardy space $H^1 (\T)$; see  \cite{Stein_mult}. See also \cite{Grafakos_modern}*{Theorem 2.2.9} for the corresponding Euclidean case.

Using Theorem \ref{LP_Psi_r_predual_Thm} one can obtain a direct proof of implication $(3) \implies (1)$ of part (b) of Theorem \ref{gen_mult_q_leq2}, i.e. for $q=2$, and for growth functions of the form $\Psi= \Psi_{r,1}$, $r >0$. Indeed, suppose that  $ \lambda = \{ \lambda_n \}_{n \in \N_0}$ is a sequence of complex numbers satisfying condition $(iii)$ in part (b) of Theorem \ref{gen_mult_q_leq2}.  Then take an $f \in H^{\Psi_{r,1}} (\T)$ and observe that, by using $(iii)$, one has
\begin{align*}
\left( \sum_{k \in \N_0} | \lambda_k \widehat{f} (k) |^2 \right)^{1/2} & = \left( \sum_{n \in \N_0} \sum_{2^{n-1} \leq  k  < 2^n } | \lambda_k \widehat{f} (k) |^2 \right)^{1/2} \\
& \leq \left( \sum_{n \in \N_0}  \max_{2^{n-1} \leq k < 2^n} | \widehat{f} (k) |^2  \sum_{2^{n-1} \leq k < 2^n } | \lambda_k |^2 \right)^{1/2} \\
& \lesssim \left( \sum_{n \in \N_0} \frac{ \max_{2^{n-1} \leq k < 2^n} | \widehat{f} (k)|^2 }{(n+1)^{2r}}  \right)^{1/2} \\
&\lesssim \norm{\left( \sum_{n \in \N_0} \frac{ \left| \Delta_n (f) \right|^2}{\left( n +1 \right)^{2r}} \right)^{1/2}  }_{L^1 (\T)} .
\end{align*}
Hence, by using  Minkowski's inequality and then Theorem \ref{LP_Psi_r_predual_Thm}, one deduces that 
\[
\left( \sum_{n \in \N_0} | \lambda_n \widehat{f} (n) |^2 \right)^{1/2} \lesssim \norm{ f }_{H^{\Psi_{r,1}}(\T)}
\]
i.e. that $(i)$ holds.

\section{Results in the Euclidean setting}\label{eucl_proof}
\subsection{Spaces of logarithmic smoothness}\label{s_log_smooth}
Let $\varphi_0$ be a Schwartz function, radial, positive and supported in $\{ \abs{\xi}\leq 1 \}$, which is equal to $1$ in $\{ \abs{\xi}\leq 1/2 \}$.  Let $\varphi_1(\xi):=\varphi_0(\xi/2)-\varphi_0(\xi)$, and 
\[
\varphi_j(\xi):=\varphi_1(2^{-j}\xi), \qquad j\geq 2.
\]
For each $j\geq 1$, the function $\varphi_j$ is supported in the annulus $\{2^{j-1}\leq\abs{\xi}\leq 2^{j+1}\}$, and for all $\xi\in\R^d$ we have
\begin{equation}\label{PasQ}
1=\varphi_0(\xi)+\sum_{\ell=1}^\infty\varphi_\ell(\xi).
\end{equation}
Such a family $\{\varphi_j\}_{j=0}^\infty$ is referred to as a non-homogeneous resolution of unity. 
For $f \in \mathcal{S}' (\R^d)$ we set 
\[ 
\widetilde{\Delta}_j (f) : =\brkt{\varphi_j \widehat{f}}^\vee, \qquad \mbox{for $j\geq 0$.}
\]  

\begin{define}\label{GSTL}
Let $s,r\in\R$, $0<p\leq\infty$, $0<q\leq\infty$, and let $\lbrace\varphi_j\rbrace_{j=0}^\infty$ be a resolution of unity as above. 
\begin{itemize}
\item \cites{cae-mou,M} If $p<\infty$, we define the Triebel--Lizorkin space of generalised smoothness, $F_{p,q}^{s,r}(\R^d)$, to be the set of all tempered distributions $f$ for which
\[
\norm{f}_{F_{p,q}^{s,r}(\R^d)}:=\norm{\left(\sum_{j=0}^\infty 2^{jsq}\brkt{1+j}^{rq}\abs{\widetilde{\Delta}_jf}^q\right)^{1/q}}_{L^p(\R^d)}
\]
is finite, with the usual modification if $q=\infty$.
\item \cite{AR2}*{Definition 2.7} Let $0<q<\infty$ and let $\mathcal{D}$ be the set of all dyadic cubes in $\R^n$. We define $F_{\infty,q}^{s,r}(\R^n)$ to be the set of all tempered distributions $f$ for which
\begin{align*}
\norm{f}_{F_{\infty,q}^{s,r}(\R^d)} :=& \norm{\widetilde{\Delta}_0f}_{L^{\infty}(\R^d)}\\
+& \sup_{\substack{Q\in\mathcal{D} \\ \ell(Q)\leq 1}}\left(\frac{1}{\abs{Q}}\int_Q\sum_{j=-\log_2 \ell(Q)}^\infty 2^{sj q}\brkt{1+j}^{rq}\abs{\widetilde{\Delta}_jf(x)}^q \, \dd x\right)^{1/q}
\end{align*}
is finite.
\end{itemize}
\end{define}

\begin{rmk}
In the previous definition, the case where $r=0$ recovers the classical definition of Triebel--Lizorkin space. 
\end{rmk}
\begin{rmk} Let $p(\xi)=\sum_{\abs{\alpha}\leq N} c_{\alpha} \xi^\alpha$ be a polynomial expression in $\xi$ with constant coefficients, where for every multi-index $\alpha\in \N^d$,  
\[
\abs{\alpha}=\alpha_1+\ldots\alpha_d, \quad \xi^\alpha=\xi_1^{\alpha_1}\cdots \xi_d^{\alpha_d}.
\]
Let $D:=-i\partial$, so $p(D)$ denotes the constant coefficient differential operator
\[
p(D)f=\sum_{\abs{\alpha}\leq d} c_\alpha D^\alpha f.
\]
Using the properties of the Fourier transform, for any $f\in \mathcal{S}$, this can be written as the Fourier multiplier operator with symbol $p(\xi)$ given by
\[
p(D)f(x)=(2\pi)^{-d}\int_{\R^d} p(\xi) \widehat{f}(\xi) e^{ix\xi} \dd \xi.
\]
Notice that such an expression makes sense, on $\mathcal{S}(\R^d)$, and on $\mathcal{S}'(\R^d)$ for a wider class of symbols $p$ than polynomials, such as those in the Kohn--Nirenberg classes $S^m(\R^d)$, with $m\in \R$. Let us recall that a symbol $\sigma$ belongs to $ S^m (\R^d)$ if, and only if, $\sigma$ is smooth and, for all multi-indices $\alpha\in\N^n$, satisfies 
\begin{equation}\label{S0}
\sup_{\xi\in\R^n} \jap{\xi}^{-m+\abs{\alpha}}\abs{\partial^\alpha_\xi\sigma(\xi)}<\infty.   
\end{equation}
Here, we use the shorthand notation \[\jap{\xi}:=(1+|\xi|^2)^{1/2}, \quad \xi \in \R^d.
\]
In particular, if one defines $p(\xi)=\abs{\xi}^2$, the Laplacian of $f$ can be written as 
\[
-\Delta f(x)=\sum_{j=1}^{d} (-i)^2\partial^{2e_j} f(x) =p(D) f(x).  
\]
This allows us to define, for any $s\in \R$, $(1-\Delta)^{s/2}$ as the Fourier multiplier operator with symbol $\jap{\xi}^{s}$.

\end{rmk}

These spaces of generalised smoothness, see Definition \ref{GSTL} above, can be realised as potential type spaces. To be more specific, if one defines, for $r\in \R$,
\[
\log^r(e-\Delta)f(x) :=(2\pi)^{-d}\int_{\R^d} \log^r(e+\abs{\xi}^2) \widehat{f}(\xi) e^{i x\xi} \, \dd \xi, \quad x \in \R^d,
\]
then 
\begin{equation}\label{eq:equivalent_norms}
\norm{f}_{F^{s,r}_{p,q} (\R^d) }\approx	\norm{(1-\Delta)c\frac{s}{2} \log^r(e-\Delta)f}_{F^{0,0}_{p,q} (\R^d) }.
\end{equation}
This is a consequence of the lifting property of these spaces (see \cite{cae-mou}*{Proposition 3.2} for $p<\infty$ and \cite{AR3}*{Proposition 2.15} for the case $p=\infty$) and the fact that for all multi-indices $\alpha$, 
\begin{equation}\label{right_symbol_class}
\abs{\partial_\xi^\alpha \log^r(e+\abs{\xi}^2)}\lesssim (1+\abs{\xi})^{-\abs{\alpha}}{\log^r(e+\abs{\xi})}.
\end{equation}
In particular, we have that
\[
\norm{f}_{F^{0,r}_{p,2}(\R^d)}\approx \norm{\log^r(e-\Delta)f}_{h^p(\R^d)}.
\]
Indeed, if one defines $u_r(\xi):=\log^r(e+\abs{\xi}^2)$,  and 
\begin{equation}\label{log_symbol}
\mathrm{w}_r(\xi)=\sum_{j=0}^\infty (1+j)^r\varphi_j(\xi),\quad \xi\in\R^d,
\end{equation}
the aforementioned lifting property yields that
\[
\norm{f}_{F^{s,r}_{p,q} (\R^d) }=	\norm{(1-\Delta)^\frac{s}{2} \mathrm{w}_r(D)f}_{F^{0,0}_{p,q} (\R^d)}.
\] 
A simple calculation shows that both $u_r/\mathrm{w}_r$ and $\mathrm{w}_r/u_r$ belong to the Kohn--Niren\-berg class $S^0(\R^d)$. Hence, the associated Fourier multiplier operators are bounded in Triebel--Lizorkin spaces (see \cite{Tri83}*{Theorem 2.3.7}), establishing the equivalence in \eqref{eq:equivalent_norms}.

The description as potential-type spaces allows known results to be lifted into spaces of generalised smoothness. For instance, one can show that 
\[
\brkt{F^{0,-r}_{1,2}(\R^d)}^{\ast} \cong F^{0,r}_{\infty,2}(\R^d),
\]
by using the identification of the dual of $h^1(\R^d)=F^{0,0}_{1,2}(\R^d)$ as $\mathrm{bmo}(\R^d)=F^{0,0}_{\infty,2}(\R^d)$. 
This is a consequence of the fact that 
\begin{align*}
\sup_{\norm{ f }_{F^{0,-r}_{1,2} ( \R^d )}\leq 1} |\langle f,g \rangle| & =\sup_{\norm{ \mathrm{w}_{-r}(D)f }_{h^1 ( \R^d ) }\leq 1} |\langle \mathrm{w}_{-r}(D)f, (\mathrm{w}_{-r})^{-1}(D)g \rangle|\\
&= \sup_{\norm{ h }_{h^1 (\R^d)}\leq 1} | \langle h,(\mathrm{w}_{-r})^{-1}(D)g \rangle|\\
&=\norm{ (\mathrm{w}_{-r})^{-1}(D) g }_{\mathrm{bmo} ( \R^d )}\approx \norm{\mathrm{w}_{r}(D)g }_{\mathrm{bmo}( \R^d ) } , %
\end{align*}
where in the last equivalences, we use that both $\mathrm{w}_{-r}/\mathrm{w}_r$ and $\mathrm{w}_{r}/\mathrm{w}_{-r}$ belong to $S^0(\R^d)$, and the boundedness of these Fourier multipliers on $\mathrm{bmo} (\R^d)$.
\subsection{Embedding of local Hardy--Orlicz spaces into spaces of logarithmic smoothness} \label{LP_Psi_r_predual}

In this section we shall prove the following Euclidean analogue of Theorem \ref{LP_Psi_r_predual_Thm}, which can be regarded as an embedding result of the space $h^{\Psi_r,1}(\R^d)$ into the Triebel--Lizorkin space of generalised smoothness $F^{0,-r}_{1,2}(\R^d)$.

\begin{theo}\label{LP_Psi_r_euclidean} Let $r \geq 0$.
There exists a constant $C_r>0$ such that
\begin{equation}\label{LP_H_ineq}
\norm{ \left( \sum_{j \geq 0} \frac{ | \widetilde{\Delta}_j (f)  |^2}{(j+1)^{2r}} \right)^{1/2} }_{L^1 (\R^d)} \leq C_r \norm{ f }_{h^{\Psi_{r,1}} (\R^d)}, 
\end{equation}
where $\Psi_{r,1} (t) := t [\log (e+t)]^{-r}$, $t \geq 0$.
\end{theo}

Theorem \ref{LP_Psi_r_euclidean} will be obtained by combining the following lemma with the atomic decomposition of $h^{\Psi_{r,1}} (\R^d)$.

\begin{lemma}\label{main_lemma} Let $r \geq 0$. There exists a constant $M_r >0$ such that for any cube $Q \subseteq \R^d$ and for any function $\beta_Q \in L^{\infty} (\R^d)$ satisfying that:
\begin{enumerate}[(i)]
\item $\mathrm{supp} (\beta_Q) \subseteq Q$;
\item $\int_Q \beta_Q = 0$ if $\abs{Q}<1$;
\end{enumerate}
one has 
\[
\norm{ \left( \sum_{j \geq 0} \frac{ | \widetilde{\Delta}_j ( \beta_Q )  |^2}{(j+1)^{2r}} \right)^{1/2} }_{L^1 (\R^d)} \leq M_r | Q| \frac{ \norm{ \beta_Q }_{L^{\infty} (\R^d)}} { \left[ \log \left(e + |Q|^{-1} \right) \right]^r }.
\]
\end{lemma} 

\begin{proof} Let $r \geq 0$ and let $Q$ and $\beta_Q$ be as in the statement of the lemma with $\abs{Q}\leq 1$.  Without loss of generality, by translation-invariance, we may assume that $Q$ is centred at the origin. 
	
We write
\begin{equation}\label{AB}
\left( \sum_{j \geq 0} \frac{ | \widetilde{\Delta}_j (\beta_Q) (x) |^2}{(j+1)^{2r}} \right)^{1/2} \leq A (x) + B(x),
\end{equation}
where
\[
A (x) := \left( \sum_{2^j < \abs{Q}^{-1}  } \frac{ | \widetilde{\Delta}_j (\beta_Q) (x)  |^2}{(j+1)^{2r}} \right)^{1/2}
\]
and
\[
B (x) := \left( \sum_{2^j \geq \abs{Q}^{-1} } \frac{ | \widetilde{\Delta}_j (\beta_Q) (x)  |^2}{(j+1)^{2r}} \right)^{1/2} . 
\]
We shall prove that there exists a constant $K_r>0$ such that
\begin{equation}\label{L-P_ineq_A}
\norm{ A }_{L^1 (\R^d)} \leq K_r | Q | \frac{ \norm{ \beta_Q }_{L^{\infty} (\R^d)} } { \left[ \log \big( e + | Q |^{-1}) \right]^r }
\end{equation}
and
\begin{equation}\label{L-P_ineq_B}
\norm{ B }_{L^1 (\R^d)} \leq K_r | Q |  \frac{ \norm{ \beta_Q }_{L^{\infty} (\R^d)} } { \left[ \log \big( e +  | Q |^{-1}  \big)  \right]^r }.
\end{equation}
Note that the desired estimate then follows from \eqref{AB}, \eqref{L-P_ineq_A}, and \eqref{L-P_ineq_B}. %
	
For the proofs of \eqref{L-P_ineq_A} and \eqref{L-P_ineq_B}, we shall use the standard facts that there is an absolute constant $C>0$ such that
\begin{equation}\label{proj_bound1}
| \widetilde{\Delta}_j (\beta_Q) (x) | \leq C 2^{jd} | Q | \norm{ \beta_Q }_{L^{\infty} (\R^d)} \quad \text{for all } x \in \R^d,
\end{equation}
and for each $N >0$ there exists a constant $C_N>0$ such that
\begin{equation}\label{proj_bound2}
| \widetilde{\Delta}_j (\beta_Q) (x) | \leq C_N
| Q |^{d+1} \| \beta_Q \|_{L^{\infty} (\R^d)} \frac{ 2^{j(2d-N)}} {|x|^N} 
\end{equation}
for all $x \in \R $ with $ |x| > 2|Q| $ and for all $j \geq 0$. 
	
To prove \eqref{proj_bound1} and \eqref{proj_bound2}, consider Schwartz functions 
\[
\Phi_j:=\varphi_j^\vee.
\]
It follows that 
\[
\norm{ \Phi_j }_{L^{\infty} (\R^d)} \lesssim 2^{jd} \quad \text{for all } j \geq 0.
\]
One thus deduces that
\[
\norm{ \widetilde{\Delta}_j (\beta_Q) }_{L^{\infty} (\R^d)} = \norm{ \Phi_j \ast \beta_Q }_{L^{\infty} (\R^d)}  \lesssim 2^{jd} | Q | \norm{ \beta_Q }_{L^{\infty} (\R^d)} \quad  \text{for all } j \geq 0.
\]
Hence, \eqref{proj_bound1} holds. To prove \eqref{proj_bound2}, let $x$ be such that $|x| > 2\abs{Q}$. For $j \geq 1$, properties (i) and (ii) of $\beta_Q$ yield that
\begin{align*}
| \widetilde{\Delta}_j (\beta_Q) (x) | &= \left| \int_Q \beta_Q (y) \left[ \Phi_j (x-y) - \Phi_j (x) \right] \, \dd y \right| \\
& \leq \| \beta_Q \|_{L^{\infty} (\R^d)} \int_Q  \left| \Phi_j (x-y) - \Phi_j (x) \right| \, \dd y \\
& \lesssim_N \| \beta_Q \|_{L^{\infty} (\R^d)} \int_{Q} |y| \frac{ 2^{2jd}} {(1 + 2^j |x| )^N} \, \dd y \\
& \lesssim | Q |^{d+1} \| \beta_Q \|_{L^{\infty} (\R^d)} \frac{ 2^{j(2d-N)}} {  |x|^N } ,
\end{align*}
where the implicit estimate that was used to go from the second to the third line can be justified by appealing to the mean value theorem and the rapid decay of $\varphi_1$. 
	
Going back to the proof of  \eqref{L-P_ineq_A}, write
\begin{equation}\label{A_dec}
\norm{ A }_{L^1 (\R^d)} = \int_{|x| \leq 2| Q | } |A (x)| \, \dd x +  \int_{| x | > 2| Q | } |A (x)| \, \dd x .
\end{equation}
Hence, by using \eqref{proj_bound1}, we have
\begin{align*}
\int_{|x| \leq 2\abs{Q} } |A (x)| \, \dd x \lesssim | Q |^d \| A \|_{L^{\infty} (\R^d)} 
&\lesssim | Q |^d \left( \sum_{2^j <\abs{Q}^{-1} } \frac{ \| \widetilde{\Delta}_j (\beta_Q) \|^2_{L^{\infty} (\R^d)} }{(j+1)^{2r}} \right)^{1/2}  \\
& \lesssim | Q |^{d+1} \| \beta_Q \|_{L^{\infty} (\R^d)} \left( \sum_{2^j < \abs{Q}^{-1} } \frac{2^{2jd} }{(j+1)^{2r}} \right)^{1/2}
\end{align*}
and as 
\[
\sum_{2^j <\abs{Q}^{-1} } \frac{2^{2jd} }{(j+1)^{2r}} \lesssim_r \frac{ |Q|^{-2d} }{ \left[ \log  \left(e + |Q|^{-1} \right) \right]^{2r}},
\]
we get
\begin{equation}\label{L-P_ineq_A1}
\int_{|x| \leq 2\abs{Q}} |A (x)| \, \dd x  \lesssim_r | Q | \frac{\norm{ \beta_Q }_{L^{\infty} (\R^d)} }{  \left[ \log (e + | Q |^{-1}) \right]^r}.
\end{equation}
To handle the second term, note that by using \eqref{proj_bound2}, we have
\begin{align*}
\int_{|x| >2 \abs{Q}} |A (x)| \, \dd x & \lesssim | Q |^{d+1}  \norm{ \beta_Q }_{L^{\infty} (\R^d)} \left( \sum_{2^j <\abs{Q}^{-1} } \frac{2^{j(4d-2N)}}{(j+1)^{2r}} \right)^{1/2} \int_{|x| > 2| Q | } \frac{\dd x}{|x|^{N}}\\ 
\end{align*}
provided $2d>N>d$. This, combined with
\[
\sum_{2^j < \abs{Q}^{-1} } \frac{2^{j(4d-2N)}}{(j+1)^{2r}} \lesssim_r \frac{| Q |^{-(4d-2N)} }{ \left[ \log \left( e + | Q |^{-1} \right) \right]^{2r}}, 
\]
yields
\begin{equation}\label{L-P_ineq_A2}
\int_{|x| > |Q|} |A (x)| \, \dd x  \lesssim_r | Q |  \frac{ \norm{ \beta_Q }_{L^{\infty} (\T)} } { \left[ \log \big( e + |Q|^{-1} \big) \right]^r} . 
\end{equation}
In view of \eqref{L-P_ineq_A1} and \eqref{L-P_ineq_A2}, the proof of \eqref{L-P_ineq_A} is complete.
	
To prove \eqref{L-P_ineq_B}, note that
\[
B (x) \leq \frac{1}{\left[ \log (e+|Q|^{-1}) \right]^r} \left( \sum_{2^j \geq |Q|^{-1} }  | \widetilde{\Delta}_j (\beta_Q) (x)  |^2 \right)^{1/2} . 
\]
Hence, it suffices to show that
\begin{equation}\label{H^1_est}
\int_{\R^d} \left( \sum_{2^j \geq | I  |^{-1} } | \widetilde{\Delta}_j (\beta_Q) (x)  |^2 \right)^{1/2} \, \dd x \lesssim |Q| \norm{ \beta_Q }_{L^{\infty} (\R^d)}. 
\end{equation}
To this end, write
\begin{align*}
\int_{\R^d} \left( \sum_{2^j \geq | Q  |^{-1} }  | \widetilde{\Delta}_j (\beta_Q) (x)  |^2 \right)^{1/2} \, \dd x &= \int_{|x| \leq 2| Q |} \left( \sum_{2^j \geq | Q  |^{-1} }  | \widetilde{\Delta}_j (\beta_Q) (x)  |^2 \right)^{1/2}\, \dd x \\
&+ \int_{|x| > 2| Q |} \left( \sum_{2^j \geq | I  |^{-1} }  | \widetilde{\Delta}_j (\beta_Q) (x)  |^2 \right)^{1/2} \, \dd x.
\end{align*}
For the first term, by using the Cauchy--Schwarz inequality and Parseval's identity, we get
\[
\int_{|x| \leq 2| Q |} \left( \sum_{2^j \geq |Q|^{-1} }  | \widetilde{\Delta}_j (\beta_Q) (x)  |^2 \right)^{1/2} \, \dd x \lesssim |Q|^{1/2} \norm{\beta_Q}_{L^2 (\R^d)}
\]
and hence,
\begin{equation}\label{H^1_est1}
\int_{| x| \leq 2| Q |} \left( \sum_{2^j \geq |Q|^{-1} }  | \widetilde{\Delta}_j (\beta_Q) (x)  |^2 \right)^{1/2} \, \dd x \lesssim |Q| \norm{\beta_Q}_{L^{\infty} (\R^d)}. 
\end{equation}
To handle the second term, we use \eqref{proj_bound2} for $N>2d$,
\begin{align*}
& \int_{| x| > 2| Q |} \left( \sum_{2^j \geq |Q|^{-1} }  | \widetilde{\Delta}_j (\beta_Q) (x)  |^2 \right)^{1/2} \, \dd x \lesssim \\
&  |Q|^{d+1} \norm{\beta_Q}_{L^{\infty} (\R^d)} \brkt{ \sum_{2^j \geq |Q|^{-1}} 2^{2j(2d-N)} }^{1/2} \int_{| x| > 2|Q|} \frac{1}{| x|^N} \, \dd x. 
\end{align*}
Hence,
\begin{equation}\label{H^1_est2}
\int_{| x| > | Q |} \left( \sum_{2^j \geq | I  |^{-1} }  | \widetilde{\Delta}_j (\beta_Q) (x)  |^2 \right)^{1/2} \, \dd  x \lesssim |Q| \norm{\beta_Q}_{L^{\infty} (\R^d)}. 
\end{equation}
Therefore, in view of \eqref{H^1_est1} and \eqref{H^1_est2},
\eqref{L-P_ineq_B}  holds and so, the proof of the lemma is complete for the case $\abs{Q}<1$.
	
Assume now that $\abs{Q}\geq 1$. Notice then that $\beta_Q$ is then a multiple of an $h^1(\R^d)$-atom as
\[
a_Q=\frac{\beta_Q}{\norm{\beta_Q}_{L^\infty(\R^d)}\abs{Q}}
\]
is an $h_1(\R^d)$-atom. The the characterisation of $h^1(\R^d)$ in terms of Littlewood--Paley partitions (see e.g. \cite{Stein_mult} for the corresponding periodic case), the homogeneity of the $h^1(\R^d)$-norm, and the fact that for $\abs{Q}\geq 1$
\[
\log(e+\abs{Q}^{-1})\approx 1,
\]
yield
\begin{align*}
\norm{ \left( \sum_{j \geq 0} \frac{ | \widetilde{\Delta}_j ( \beta_Q )  |^2}{(j+1)^{2r}} \right)^{1/2} }_{L^1 (\R^d)} &\leq \norm{ \left( \sum_{j \geq 0} | \widetilde{\Delta}_j ( \beta_Q ) |^2 \right)^{1/2} }_{L^1 (\R^d)}\\
&\approx\norm{\beta_Q}_{h^1(\R^d)}\\
&\lesssim | Q | \norm{ \beta_Q }_{L^{\infty} (\R^d)} \\
& \approx \frac{ | Q |\norm{ \beta_I }_{L^{\infty} (\R^d)}} { \left[ \log \left(e + |Q|^{-1} \right) \right]^r }.        
\end{align*}
This completes the proof of the lemma. \end{proof}

\subsection*{Proof of Theorem \ref{LP_Psi_r_euclidean}} The case $r=0$ is well-known. In fact, for $r=0$, \eqref{LP_H_ineq} holds as an equivalence; this is the Littlewood--Paley characterisation of $h^1 (\R^d)$ (see e.g. \cite{Stein_mult} for the corresponding periodic case or \cite{Grafakos_modern}*{Theorem 2.2.9} for the homogeneous case i.e. for the Littlewood--Paley characterisation of $H^1 (\R^d)$).

We shall therefore give a proof for the case $r >0$. 
Towards this aim, observe that, as the norms in both sides of \eqref{LP_H_ineq} are homogeneous, it suffices to establish \eqref{LP_H_ineq} for $f \in h^{\Psi_{r,1}} (\R^d)$  with $\| f \|_{h^{\Psi_{r,1}} (\R^d)} = 1$. 

To this end, consider an $f \in h^{\Psi_{r,1}} (\R^d)$  with $\| f \|_{h^{\Psi_{r,1}} (\R^d)} = 1$ and note that by the atomic decomposition of $h^{\Psi_{r,1}} (\R^d)$ there exists an absolute constant $A_r>0$ and a sequence of multiples $\{ \beta_{Q_k} \}_{k \in \N}$ of atoms such that
\begin{equation}\label{conv}
f = \sum_k \beta_{Q_k} \quad \text{in } \mathcal{S}'(\R^d) 
\end{equation}
and
\begin{equation}\label{uppA}
\sum_{k \in \N} | I_k | \frac{ \norm{ \beta_{Q_k} }_{L^{\infty} (\R^d)} }{ \left[ \log \left( e + \norm{ \beta_{Q_k} }_{L^{\infty} (\R^d)} \right) \right]^r } \leq A_r . 
\end{equation}
By \eqref{conv}, one has
\[
\norm{ \left( \sum_{j \geq 0} \frac{ | \widetilde{\Delta}_j (f)  |^2}{(j+1)^{2r}} \right)^{1/2} }_{L^1 (\R^d)} \leq \sum_{k \in \N} \norm{ \left( \sum_{j \geq 0} \frac{ | \widetilde{\Delta}_j ( \beta_{Q_k} ) |^2}{(j+1)^{2r}} \right)^{1/2} }_{L^1 (\R^d)} 
\]
and hence, in view of \eqref{uppA}, it suffices to prove that there exists an $M_r >0$, depending only on the constant $A_r$ appearing in \eqref{uppA}, such that
\begin{equation}\label{mult_atom}
\norm{ \left( \sum_{j \geq 0} \frac{ | \widetilde{\Delta}_j ( \beta_{Q_k} ) |^2}{(j+1)^{2r}} \right)^{1/2} }_{L^1 (\R^d)} \leq M_r | Q_k | \frac{ \norm{ \beta_{Q_k} }_{L^{\infty} (\R^d)}} { \left[ \log \left(e + \norm{ \beta_{Q_k} }_{L^{\infty} (\R^d)} \right) \right]^r  }
\end{equation}
for all $k \in \N$. 

For each $k \in \N$, as $\beta_{Q_k}$ is a multiple of an $h^{\Psi_{r,1}}$-atom, $\beta_{Q_k}$ satisfies properties (i) and (ii) of Lemma \ref{main_lemma}. Moreover, it follows from \eqref{uppA} that
\[
\frac{ \norm{ \beta_{Q_k} }_{L^{\infty} (\R^d)} }{ \left[ \log \left( e + \norm{ \beta_{Q_k} }_{L^{\infty} (\R^d)} \right) \right]^r } \leq A_r | Q_k |^{-1},
\]
which implies that 
\begin{equation}\label{inv_bound}
\left[ \log \left( e + \norm{ \beta_{Q_k} }_{L^{\infty} (\R^d)} \right) \right]^r  \lesssim_r  \left[ \log \left( e + |Q_k|^{-1} \right) \right]^r . 
\end{equation}
Therefore, \eqref{mult_atom} follows from \eqref{inv_bound} the Lemma  \ref{main_lemma}.

Here we give two re-statements of Theorem \ref{LP_Psi_r_euclidean}. The first one is read as a embedding of spaces, while the second one, by using  \eqref{eq:equivalent_norms}, can be interpreted as a Sobolev-type embedding.  
\begin{corollary}\label{cor:embedding} Let $r\geq 0$. The space $h^{\Psi_r,1}(\R^d)$ is continously embedded in the space $F^{0,-r}_{1,2}(\R^d)$.
\end{corollary}

\begin{corollary} Let $r\geq 0$. There exists a constant $C_r >0$ such that for all $f\in h^{\Psi_{r,1}}(\R^d)$ one has
\[
	\norm{\log^{-r}(e-\Delta) f}_{h^1(\R^d)}\leq C_r  \norm{f}_{h^{\Psi_{r,1}}(\R^d)}.
\] 
\end{corollary}

\subsection{Behaviour of the Fourier transform of distributions in Hardy--Orlicz spaces}\label{behaviour}

Let $\Psi$ be a growth function. In this section we study properties of the Fourier transform of distributions belonging to  $H^{\Psi} (\R^d)$, or $h^{\Psi} (\R^d)$. We first obtain pointwise estimates, and then use them to prove Theorem \ref{eucl}, which is an extension of the  Hardy--Littewood inequality to $H^{\log}$-spaces. 

\begin{proposition}\label{FT_HO} Let $\Psi$ be a growth function of order $p \in (0,1]$.

If $f\in H^{\Psi} (\R^d)$, then its Fourier transform coincides with a continuous function that we denote by $\widehat{f}$. Moreover, there exists a constant $A_{d, \Psi}> 0$, depending only on $\Psi$ and $d$, but not on $f$, such that
\[
\abs{\widehat{f}(\xi)} \leq A_{d, \Psi}  \frac{ \Psi^{-1} (a_d \abs{ \xi }^d) }{a_d \abs{ \xi }^d} \norm{f}_{H^{\Psi} (\R^d)}
\]
for all $\xi \in \R^d \setminus \{ 0 \}$. Here $\Psi^{-1}$ denotes the inverse of $\Psi$ and $a_d : = \abs{B(0,1)}^{-1}$. 
\end{proposition}

\begin{proof} Let $\phi $ be a fixed Schwartz function on $\R^d$ such that $\widehat{\phi} (\xi) = 1$ for all $\abs{\xi } \leq 1$. 
For $f \in \mathcal{S}' (\R^d)$, define
\[
M_{\phi}^{\ast} [f] (x) : =  \sup_{\substack{ (y,t) \in \R^d \times [0, \infty):\\
 \abs{y-x}< t }} \abs{f \ast \phi_s (y)} , \quad x \in \R^d . 
\]
Let $f \in H^{\Psi} (\R^d)$ be non-zero. Suppose first that
\[
\| f \|_{H^{\Psi} (\R^d)} = \inf \left\{ \lambda >0 : \int_{\R^d} \Psi \brkt{\lambda^{-1} M_{ \mathcal{F}_{m_{\Psi}}}^{\ast} [f] (x)} \, \dd x \leq 1 \right\} = 1. 
\]
Observe that this implies 
\begin{equation}\label{extra_assumption}
\int_{\R^d} \Psi \brkt{  M_{\phi}^{\ast} [f] (x)} \, \dd x \leq 1. 
\end{equation}
Note first that, for $t>0$, if $\abs{x-y}<t$ then
\[	\abs{f \ast \phi_t (x)}\leq \sup_{\substack{ (z,s) \in \R^d \times [0, \infty):\\
 \abs{z-y}<s }} \abs{f \ast \phi_s (z)} = M_{\phi}^{\ast} [f] (y), 
\]
which implies that for $t>0$ and $\abs{x-y}<t$,
\[
\Psi \brkt{\abs{f*\phi_t(x)}}\leq \Psi \brkt{M^*_{\phi}f(y)}. 
\]
This yields
\begin{equation}\label{eq:pointwise}
\Psi \brkt{\abs{f*\phi_t(x)}}\leq a_dt^{-d}\int_{B(x,t)} \Psi \brkt{M^*_{\phi}[f](y)} \, \dd y,
\end{equation}
where $a_d : =\abs{B(0,1)}^{-1}$. It follows from \eqref{eq:pointwise}   that   
\begin{equation}\label{eq:pointwise_2}
\sup_{x \in \R^d}{\abs{f \ast \phi_t (x)}}\leq \Psi^{-1}\brkt{a_d t^{-d} },
\end{equation}
where we also used \eqref{extra_assumption}. 
	
Let us define $\eta : (0, \infty) \rightarrow (0, \infty)$ given by $\eta (t) : = t \cdot [\Psi (t)]^{-1}$, $t > 0$. Observe that 
\begin{align*} 
\abs{f*\phi_t(x)} \leq  \eta \brkt{ \abs{f*\phi_t(x)} } \Psi \brkt{ M_{\phi}^{\ast} [f] (x)} \leq \eta \brkt{ \sup_{x \in \R^d}{\abs{f \ast \phi_t (x)}}  } \Psi \brkt{ M_{\phi}^{\ast} [f] (x)}
\end{align*}
and so, by employing \eqref{eq:pointwise_2}, we obtain
\begin{equation}\label{eq:pointwise_ineq}
\abs{f*\phi_t(x)} \leq  \eta\circ \Psi^{-1} \brkt{a_dt^{-d}} \Psi \brkt{ M_{\phi}^{\ast} [f] (x)}.
\end{equation}
Then using \eqref{extra_assumption}, we get 
\begin{equation}\label{eq:trivial}
\int_{\R^d}  \abs{f \ast \phi_t(x)} \, \dd x \leq   \eta \circ \Psi^{-1}\brkt{a_d t^{-d} } .
\end{equation}
Note that for all $0<t<1$, we have $h_t:=f \ast \phi_t \in L^1(\R^d)$, and $\widehat{\phi_t}(\xi)=1$ for $\abs{\xi}\leq 1/t$, which yields that, in the sense of distributions, 
\[
\widehat{f}\widehat{\phi_t} = \widehat{h_t}.
\]
So we have that, for all $0<t<1$, the distribution $\widehat{f}$ concides with the continuous function $\widehat{h_t}$ on the compact set $\{ \abs{\xi}\leq 1/t \}$, and notice that for $0<s<t<1$, $\widehat{h_s}=\widehat{h_t}$ for all $\xi \in \R$ with $\abs{\xi}\leq 1/t$. This allows us to construct a continuous function $g$,   such that $\widehat{f}$ coincides with $g$ in a distributional sense and
 \begin{equation}\label{expr}
g(\xi) = \widehat{ f \ast \phi_t } (\xi), \quad \abs{\xi}\leq 1/t.
 \end{equation}

Fixing $\xi \in \R^d \setminus \{ 0 \} $, and combining  \eqref{expr} (with the choice $t = |\xi|^{-1}$) with \eqref{eq:trivial}, we have
\[ 
\abs{g (\xi)} =\abs{ \widehat{ f \ast \phi_{|\xi|^{-1}} } (\xi) } \leq  {  \norm{ f \ast \phi_{|\xi|^{-1}} }_{L^1 (\R^d)}   } \leq \eta \circ \Psi^{-1} (a_d \abs{ \xi }^d) = \frac{ \Psi^{-1} (a_d \abs{ \xi }^d) }{a_d \abs{ \xi }^d}.  
\]
Since $\Psi$ is of lower type $p$ it follows that for all $\xi\in \R^d$
\[
\abs{g (\xi)}\lesssim (1+\abs{\xi})^{d(\frac{1}{p}-1)},
\]
which yields $g\in \mathcal{S}'(\R^d)$. 
Therefore, the proof is complete if $\norm{f}_{H^{\Psi} (\R^d)} = 1$.

If $f$ is non-zero and $\norm{f}_{H^{\Psi} (\R^d)} \neq 1$, define
\[
\widetilde{f} : = \norm{f}_{H^{\Psi} (\R^d)}^{-1} f  
\]
and then apply the previous step to $\widetilde{f}$. By homogeneity, we obtain the desired result for the general case. 
\end{proof}
 
The proof of the theorem above can be easily modified to obtain the following counterpart for local Hardy--Orlicz spaces. We omit the details. 

\begin{proposition}\label{FT_HO_local}
Let $\Psi$ be a growth function of order $p \in (d/(d+1),1]$. If $f\in h^{\Psi} (\R^d)$, then its Fourier transform coincides with a continuous function that we denote by $\widehat{f}$. Moreover, there exists a constant $A_{d, \Psi}> 0$, depending only on $\Psi$ and $d$, but not on $f$, such that
\[
\abs{\widehat{f}(\xi)} \leq A_{d, \Psi}  \frac{ \Psi^{-1} (a_d(1+\abs{ \xi })^d) }{(1+\abs{ \xi })^d} \norm{f}_{h^{\Psi} (\R^d)}.
\]	
\end{proposition}

If we take $\Psi_{r,p} (t) : = t^p \left[ \log (e+t) \right]^{-r}$ as in \eqref{Psi_r,p_def}, 
with $0<p\leq 1$ and $r\geq 0$, then %
it follows from \eqref{inverse_r} that
\begin{equation}\label{eq:approx_inverse}
\Psi^{-1}_{r,p}(t)\approx t^{1/p}\log(e+t)^{r/p}.
\end{equation}

We thus deduce from Propositions \ref{FT_HO} and \ref{FT_HO_local} the following result.
 
\begin{corollary}\label{FT_Llog} Given $p \in (d/(d+1), 1]$ and $r\geq 0$, there exist positive constants $A_{d,p,r}, B_{d,p,r}$, depending only on the dimension, $p$, and $r$, such that:
\begin{enumerate}
\item for all $f \in H^{\Psi_{r,p}} (\R^d)$,
\[
\sup_{\xi \in \R^d} \frac{ \abs{\widehat{f} (\xi)} } {\abs{\xi}^{\frac{1}{p}-1}\log^{r/p} (e+\abs{\xi})} \leq A_{d,p,r}  \norm{f}_{H^{\Psi_{r,p}} (\R^d)};
\]
\item for all $f \in h^{\Psi_{r,p}} (\R^d)$,
\[
\sup_{\xi \in \R^d} \frac{ \abs{\widehat{f} (\xi)} } {(1+\abs{\xi})^{\frac{1}{p}-1}\log^{r/p} (e+\abs{\xi})} \leq B_{d,p,r}  \norm{f}_{h^{\Psi_{r,p}} (\R^d)}.
\]
\end{enumerate} 
\end{corollary}

\begin{rmk}
If $p \in (d/(d+1),1]$ and $\Phi_p (t): =\Psi_{0,p}(t)= t^p $, $t \geq 0$, the first part of Corollary \ref{FT_Llog} recovers the classical fact that
\[
\abs{\widehat{f} (\xi)} \lesssim_{d,p} \abs{\xi}^{d (p^{-1}-1)} \norm{f}_{H^p (\R^d)}, \quad \xi \in \R^d; 
\]
see \S 5.4 (a) in Chapter III  of \cite{Big_Stein}, and its local counterpart. 
\end{rmk}

\subsection{Proof of Theorem \ref{eucl}}\label{eucl_det}
Arguing as in the proof of \cite{BPRS}*{Theorem 28}, we shall prove that there exists a constant $A_{d, \Psi,p} > 0$, depending only on $d$, $\Psi$, and $p$, such that
\begin{equation}\label{atom_bound}
\int_{\R^d} \frac{ \Psi \brkt{\abs{\xi}^d\abs{\widehat{a_Q}(\xi)}}}{\abs{\xi}^{2d}} \, \dd \xi \leq A_{d, \Psi, p} \abs{Q} \Psi \brkt{ \norm{a_Q}_{L^{\infty} (\R^d)}},
\end{equation}
for all $L^{\infty}$-functions $a_Q$  that are supported in some cube $Q$ with $\int_Q a_Q (x) \dd x = 0$.

To this end, fix an $L^{\infty}$-function $a_Q$ of this type and write
\begin{equation}\label{A-B}
\int_{\R^d} \frac{ \Psi \brkt{\abs{\xi}^d\abs{\widehat{a_Q}(\xi)}}}{\abs{\xi}^{2d}} \, \dd \xi = A + B,
\end{equation}
where
\[ A:= \int_{\abs{\xi}^{d}\abs{Q}\leq 1} \frac{\Psi \brkt{\abs{\xi}^d\abs{\widehat{a_Q}(\xi)}}}{\abs{\xi}^{2d}} \, \dd \xi  \quad \text{and} \quad B:= \int_{\abs{\xi}^{d}\abs{Q} > 1} \frac{\Psi \brkt{\abs{\xi}^d\abs{\widehat{a_Q}(\xi)}}}{\abs{\xi}^{2d}} \, \dd \xi . 
 \]

To handle $A$, note that by using the cancellation of $a_Q$ and the fact that 
\[ 
\abs{ e^{-i2\pi x \xi} - e^{-i2\pi y \xi} } \leq 2 \pi  \abs{ \xi } \abs{ x-y } \lesssim_d \abs{\xi} \abs{Q}^{1/d} \quad \text{ for all } x,y \in Q, 
\]
one deduces that
\begin{equation}\label{a_Q-bound}
\abs{\widehat{a_Q}(\xi)} \lesssim_d \abs{\xi} \abs{Q}^{1+1/d}\norm{a_Q}_{L^{\infty} (\R^d)} \quad \text{for all } \xi \in \R^d. 
\end{equation}
It thus follows from \eqref{a_Q-bound} that
\begin{align*}
\int_{\abs{\xi}^{d}\abs{Q}\leq 1} \frac{\Psi \brkt{\abs{\xi}^d\abs{\widehat{a_Q}(\xi)}}}{\abs{\xi}^{2d}} \, \dd \xi &\lesssim_d  \int_0^{\abs{Q}^{-1/d}} \frac{\Psi \brkt{ \abs{ Q}^{1+1/d}\norm{a_Q}_{L^{\infty} (\R^d)} s^{d+1} } }{s^{d+1}} \, \dd s \\
&\lesssim_p {\abs{Q}^{p(d+1)/d}}  \Psi \brkt{ \norm{a_Q}_{L^{\infty} (\R^d)}} \int_0^{\abs{Q}^{-1/d}}  \rho^{p (d+1) -\dd- 1} \, \dd \rho \\
& \approx_{d,p} {\abs{Q}} \Psi \brkt{ \norm{a_Q}_{L^{\infty} (\R^d)}}
\end{align*}
since $\Psi$ is of order $p>d/(d+1)$. Hence,
\begin{equation}\label{A}
\int_{\abs{\xi}^{d}\abs{Q}\leq 1} \frac{\Psi \brkt{\abs{\xi}^d\abs{\widehat{a_Q}(\xi)}}}{\abs{\xi}^{2d}} \, \dd \xi \lesssim_{p,d} {\abs{Q}} \Psi \brkt{ \norm{a_Q}_{L^{\infty} (\R^d)}},
\end{equation}
where the implied constant depends only on $p$ and $d$ and not on $a_Q$. 

To deal with $B$, note that  
H\"older's inequality (with exponents $4$ and $4/3$) and Plancherel's theorem yield
\begin{align*}
B &=\int_{\abs{\xi}^{d}\abs{Q} > 1} \frac{\Psi\brkt{\abs{\xi}^d\abs{\widehat{a_Q}(\xi)}}}{\abs{\xi}^{2d}{\abs{\widehat{a_Q}(\xi)}^{1/2}}} \abs{\widehat{a_Q}(\xi)}^{1/2} \, \dd \xi \\ 
&\leq \norm{a_Q}_{L^2 (\R^d)}^{1/2} \brkt{\int_{\abs{\xi}^{d}\abs{Q} > 1} \brkt{ \frac{\Psi\brkt{\abs{\xi}^d\abs{\widehat{a_Q}(\xi)}}}{  \abs{\xi}^{d/2} \abs{ \widehat{a_Q}(\xi)}^{1/2}}}^{4/3} \frac{\dd \xi}{\abs{\xi}^{2d}}}^{3/4} \\
&\leq \abs{Q}^{1/4} \norm{a_Q}_{L^{\infty} (\R^d)}^{1/2} \brkt{\int_{\abs{\xi}^{d}\abs{Q} > 1} \brkt{ \frac{\Psi\brkt{\abs{\xi}^d\abs{\widehat{a_Q}(\xi)}}}{\abs{\xi}^{d/2} \abs{ \widehat{a_Q}(\xi)}^{1/2}}}^{4/3} \frac{\dd \xi}{\abs{\xi}^{2d}} }^{3/4} .
\end{align*}
Since $1/2 \leq d/ (d+1) < p$, the function $t \mapsto t^{-1/2} \Psi(t)$ is quasi-increasing. Hence, as for all $\xi\in \R^d$ one has
\[
\abs{\widehat{a_Q} (\xi)}\leq \abs{Q} \norm{a_Q}_{L^{\infty} (\R^d)} , 
\]
we deduce that 
\begin{equation}\label{psi_bound}
\frac{\Psi\brkt{\abs{\xi}^d\abs{\widehat{a_Q}(\xi)}}}{ \brkt{ \abs{\xi}^d \abs{\widehat{a_Q}(\xi)} }^{1/2}} \lesssim_p \frac{\Psi \brkt{ \abs{\xi}^d \abs{Q} \norm{a_Q}_{L^{\infty} (\R^d)} } }{\brkt{ \abs{\xi}^d \abs{Q}  \norm{a_Q}_{L^{\infty} (\R^d)} }^{1/2}} .
\end{equation}
Then, in view of \eqref{psi_bound}, one has 
\[
B \lesssim_p \abs{Q}^{1/4} \norm{a_Q}_{L^{\infty} (\R^d)}^{1/2} \brkt{\int_{\abs{\xi}^{d}\abs{Q} > 1} \brkt{ \frac{\Psi\brkt{\abs{\xi}^d \abs{Q} \norm{a_Q }_{L^{\infty} (\R^d)} }}{\abs{\xi}^{d/2} \abs{Q}^{1/2} \norm{ a_Q }_{L^{\infty}(\R^d)}^{1/2}}}^{4/3} \frac{\dd \xi}{\abs{\xi}^{2d}} }^{3/4} .
\]
Since $ t \mapsto t^{-1 } \Psi (t) $ is non-increasing, it follows that  
\[
B \lesssim_p  \Psi \brkt{ \norm{a_Q}_{L^{\infty} (\R^d)}} \brkt{    \int_{\abs{\xi}^{d}\abs{Q} > 1}   \frac{ 1 } { \abs{\xi}^{4d/3} } \, \dd \xi }^{3/4}. 
\]
Hence,
\begin{equation}\label{B}
B \lesssim_{d,p} \int_{\abs{\xi}^{d}\abs{Q} >  1} \frac{\Psi \brkt{\abs{\xi}^d\abs{\widehat{a_Q}(\xi)}}}{\abs{\xi}^{2d}} \, \dd \xi \lesssim_{p,d} {\abs{Q}} \Psi \brkt{ \norm{a_Q}_{L^{\infty} (\R^d)}},
\end{equation}
where the implied constant depends only on $p$ and $d$ and not on $a_Q$. 

To complete the proof of the theorem, note that for any given atomic decomposition of $f$ i.e.
\[ 
f=\sum_j b_j, 
\]
where $b_j$ are constant multiples of atoms in $H^{\Psi} (\R^d)$, supported in cubes $Q_j$, the sub-linearity of $\Psi$ and \eqref{atom_bound} imply that 
\[
\int_{\R^d} \frac{ \Psi \brkt{\abs{\xi}^d\abs{\widehat{f}(\xi)}}}{\abs{\xi}^{2d}} \, \dd \xi\lesssim_{d,p} \sum_j  \abs{Q_j}\Psi \brkt{ \norm{b_j}_{L^{\infty} (\R^d)}} .
\]
By taking the infimum over all possible atomic decompositions of $f$, one obtains the result of the statement.
 
The following inequality, is a consequence of Theorem \ref{eucl} combined with Proposition \ref{FT_HO}, and can be regarded as a Euclidean version of Theorem \ref{HL_complex}. 

\begin{corollary}\label{cor}
Let $\Psi$ be a growth function of order $p$ with  $p \in ( d/(d+1), 1 ]$. Then, for every $1\leq q<\infty$, there exist constants $a_d >0$ and $B_{d, \Psi,p,q } > 0$ such that   
\[
\brkt{  \int_{\R^d} \brkt{ \frac{\abs{\widehat{f}(\xi)} \abs{\xi}^d }{ \Psi^{-1} \brkt{a_d \abs{\xi}^d }} }^q \frac{\dd \xi}{ \abs{\xi}^d} }^{1/q} \leq  B_{d, \Psi,p,q } \norm{f}_{H^{\Psi} (\R^d)}.
\]
\end{corollary}

\begin{proof} Set $a_d : = \abs{B(0,1)}^{-1}$ and observe that Proposition \ref{FT_HO} corresponds to the `$q=\infty$' case. Therefore, it suffices to prove the desired inequality for $q=1$ and then, the case $q \in (1, 
\infty)$ follows by interpolation.

We shall first establish the  $q=1$ case when $\norm{f}_{H^{\Psi} (\R^d)} = 1$. Write
\[ \int_{\R^d} \frac{\abs{\widehat{f}(\xi)}}{\Psi^{-1}( a_d \abs{\xi}^d)} \, \dd \xi= 
A_{d, \Psi} \int_{\R^d}  
\Psi \brkt{ \frac{ \abs{\xi}^d \abs{\widehat{f}(\xi)} } {A_{d, \Psi}}}
 \frac{\frac{ \abs{\xi}^d\abs{\widehat{f}(\xi)}}{A_{d, \Psi}}}{ \Psi \brkt{ \frac{\abs{\xi}^d \abs{\widehat{f}(\xi)}} {A_{d, \Psi}} }} \frac{1}{\Psi^{-1}(a_d \abs{\xi}^d) \abs{\xi}^d} \, \dd \xi ,
\]
where $A_{d, \Psi}$ is the constant in Proposition \ref{FT_HO}.  By Proposition \ref{FT_HO} together with the fact that $t \mapsto t^{-1 } \Psi (t)$ is non-increasing, we have that
\[
\frac{ \abs{\xi}^d \abs{\widehat{f}(\xi)} }{ A_{d, \Psi} } \leq \Psi^{-1}(a_d \abs{\xi}^d).
\]
We get
\[
\int_{\R^d} \frac{\abs{\widehat{f}(\xi)}}{\Psi^{-1}( a_d \abs{\xi}^d)} \, \dd \xi \leq    
\frac{ A_{d, \Psi} } {a_d} \int_{\R^d}  
\Psi \brkt{ \frac{ \abs{\xi}^d \abs{\widehat{f}(\xi)} } {A_{d, \Psi}}}
  \frac{1} { \abs{\xi}^{2d} } \, \dd \xi
\]
and since $\Psi (2t) \leq 2 \Psi (t)$ for all $t \geq 0$, we deduce that
\[ 
\int_{\R^d} \frac{\abs{\widehat{f}(\xi)}}{\Psi^{-1}( a_d \abs{\xi}^d)} \, \dd \xi \leq    
\frac{ 2 } {a_d} \int_{\R^d}  
\Psi \brkt{   \abs{\xi}^d \abs{\widehat{f}(\xi)}  }
\frac{1} { \abs{\xi}^{2d} } \, \dd \xi .
\]
Hence, by Theorem \ref{eucl} we obtain 
\[
\int_{\R^d} \frac{\abs{\widehat{f}(\xi)}}{\Psi^{-1}( a_d \abs{\xi}^d)}  \,\dd \xi \leq    
B_{d, \Psi, p} ,  
\]
where $ B_{d, \Psi, p} : = 2 a_d^{-1} C_{d, \Psi, p}$ with $C_{d, \Psi, p}$ being as in the statement of Theorem \ref{eucl}. 
Therefore, the case `$q=1$' holds for all $f\in H^{\Psi} (\R^d)$ with $\norm{f}_{H^{\Psi} (\R^d) } = 1$. The general case follows by homogeneity.
\end{proof}

\begin{corollary}\label{cor_local}
Let $\Psi$ be a growth function of order $p$ with  $p \in ( d/(d+1), 1 ]$. Then, for every $1\leq q<\infty$, there exist constants $a_d >0$ and $B_{d, \Psi,p,q } > 0$ such that   
\begin{equation}\label{local_version}
\brkt{  \int_{\R^d} \brkt{ \frac{\abs{\widehat{f}(\xi)} \jap{\xi}^d }{ \Psi^{-1} \brkt{a_d \jap{\xi}^d }} }^q \frac{\dd \xi}{ \jap{\xi}^d} }^{1/q} \leq  B_{d, \Psi,p,q } \norm{f}_{h^{\Psi} (\R^d)}.    
\end{equation}
\end{corollary}
\begin{proof}
By using \cite{NaSa}*{Lemma 7.4} we have that a distribution belongs to the local space $h^\Psi$ if, and only if, $\varphi_0(D)f\in L^\Psi$ and $(1-\varphi_0)(D)f\in H^\Psi$ and 
\[
\norm{f}_{h^\Psi} \approx \norm{(1-\varphi_0)(D)f}_{H^\Psi}+\norm{\varphi_0(D)f}_{L^\Psi}
\]
where $\varphi_0$ is a Schwartz function, radial, positive and supported in $\abs{\xi}\leq 1$, which is equal to $1$ for $\abs{\xi}\leq 1/2$. Let $f_0=\varphi_0(D)f$, and $f_1=f-f_1$. So, it suffices to show \eqref{local_version} for both $f_0$ and $f_1$.

Using that for $\abs{\xi}\geq 1$, $\abs{\xi}\approx \jap{\xi}$, a direct application of Corollary \ref{cor} to deduce that \eqref{local_version} holds for $f_1$.

Since for $\abs{\xi}\leq 1$, $\jap{\xi}\approx 1$, using Proposition \ref{FT_HO_local} we have 
\[
\int_{\R^d} \brkt{ \frac{\abs{\widehat{f_0}(\xi)} \jap{\xi}^d }{ \Psi^{-1} \brkt{\jap{\xi}^d }} }^q \frac{\dd \xi}{ \jap{\xi}^d}  \lesssim \int_{\abs{\xi}\leq 1} \abs{\widehat{f}(\xi)}^q\dd \xi\lesssim \norm{f_0}_{h^\Psi(\R^d)}^q.
\]
\end{proof}

\begin{rmk}
Note that if $\Psi_1 (t) :  = t \cdot \left[ \log ( e + t ) \right]^{-1}$, $t \geq 0$, then Corollary \ref{cor} yields
\begin{equation}\label{Llog_FT}
\int_{\R^d} \frac{\abs{\widehat{f}(\xi)}}{\abs{\xi}^d \log \brkt{ e + \abs{\xi} } } \, \dd \xi \lesssim_d \norm{f}_{H^{\Psi_1} (\R^d)}
\end{equation}
for all $f\in H^{\log} (\R^d)$, which is a Euclidean version of \eqref{H-L_Hlog}.
	
\end{rmk}
\subsection{Applications of the Euclidean results}\label{applicationsection}
\subsubsection{Sobolev embedding-type results}

By using Corollary \ref{cor}, and its local version \eqref{local_version}, with $q=1$, and appealing to the properties of the Fourier transform, we have the following Sobolev embedding-type results.

\begin{corollary} \label{Cor:18}
Let $\Psi$ be a growth function of order $p$ with  $p \in ( d/(d+1), 1 ]$.
\begin{enumerate}
\item Then there exist positive constants $a_d$ and $B_{d, \Psi, p}$ such that for any Schwartz function $f$ on $\R^d$ 
\[
\norm{f}_{L^\infty (\R^d)} \leq B_{d, \Psi, p} \norm{\Psi^{-1}\brkt{ a_d (-\Delta)^{d/2}} f}_{H^{\Psi} (\R^d)}.
\]
\item Then there exist positive constants $a_d$ and $C_{d, \Psi, p}$ such that for any Schwartz function $f$ on $\R^d$ 
\[
\norm{f}_{L^\infty (\R^d)} \leq C_{d, \Psi, p} \norm{\Psi^{-1}\brkt{ a_d {\brkt{1-\Delta}}^{d/2}} f}_{h^{\Psi} (\R^d)}.
\]
\end{enumerate}
\end{corollary} 

To give an intuitive interpretation of the previous statement, let us consider the particular case that we take $\Psi_{r,p}(t)=t^p\log(e+t)^{-r}$ as in \eqref{Psi_r,p_def}. Using \eqref{eq:approx_inverse} and the corollary,  we have the following result.

\begin{corollary}\label{cor_sobolev_embedding} Let $ p \in (d/(d+1), 1]$ and $r\geq 0$. Then there exists a positive constant $B_{d, r, p}$ such that for all Schwartz functions $f$ on $\R^d$
\[
\norm{f}_{L^\infty (\R^d)} \leq B_{d, r, p} \norm{(-\Delta)^{\frac{dp}{2}}\log^{\frac r p}(e-\Delta) f}_{H^{\Psi_{r,p}} (\R^d)}
\]
and 
\[
\norm{f}_{L^\infty (\R^d)} \leq B_{d, r, p} \norm{(1-\Delta)^{\frac{dp}{2}}\log^{\frac r p}(e-\Delta) f}_{h^{\Psi_{r,p}} (\R^d)}.
\]
\end{corollary}

Roughly speaking, these inequalities say that if a distribution has its derivatives of order $dp$ plus $r/p$-logarithmic order in $h^{\Psi_{r,p}}(\R^d)$, then it belongs to $\mathcal{C}_0(\R^d)$.

Corollary \ref{cor}, and the identification in \cite{AR}*{Proposition 2.26} yield that for $2\leq q < \infty$ one has the following result.
\begin{corollary}\label{cor:Sobolev-embedding_local}
Let $p \in (d/(d+1), 1]$ and $r\geq 0$. Then for all Schwartz functions $f$ on $\R^d$ and for $2\leq q<\infty$ one has
\begin{equation}\label{J-H^{log}_emb}
\begin{split}
\norm{f}_{F^{0,-\frac r p}_{q,2}(\R^d)}&\approx \norm{\log^{-r/p}(e-\Delta)f}_{L^q (\R^d)}\lesssim \norm{ (-\Delta)^{\frac{dp}{2q'}} f}_{H^{\Psi_{r,p}} (\R^d)}
\end{split}
\end{equation}
and 
\begin{equation}\label{J-H^{log}_emb_inh}
\begin{split}
\norm{f}_{F^{-\frac{dp}{q'},-\frac{r}{p}}_{q,2}(\R^d)}&\lesssim \norm{f}_{h^{\Psi_{r,p}} (\R^d)}.
\end{split}
\end{equation}
\end{corollary}

\begin{rmk}
Recall that Corollary \ref{cor:embedding}  yields
\[
\norm{f}_{F^{0,-r}_{1,2}(\R^d)}\approx \norm{\log^{-r}(e-\Delta)f}_{h^1 (\R^d)} \lesssim_{d,q} \norm{ f}_{h^{\log^r} (\R^d)}.
\]
Note that, for all $q>1$, the standard Sobolev embedding between Triebel--Lizorkin spaces \cite{Tri83}*{\S 2.7.1} and this last inequality yield 
\[
\norm{f}_{F^{-\frac{d}{q'},-r}_{q,2}(\R^d)}\lesssim\norm{f}_{F^{0,-r}_{1,2}(\R^d)}\lesssim \norm{f}_{h^{\log^r} (\R^d)},
\]
proving that, for $p=1$, \eqref{J-H^{log}_emb_inh} holds for all $1\leq q<\infty$. 
\end{rmk}

\subsubsection{Embedding of analytic function spaces on the right halfplane}
In this section we show how $L^2$-based analytic Sobolev spaces of generalised logarithmic smoothness, and as a consequence, Hardy--Orlicz spaces of logarithmic type, can be embedded into certain spaces of analytic functions on the right halfplane (see Proposition \ref{prop:logartimic_smoothness_embedd} and its corollary). 

To be more specific, let $\nu$ be a positive regular Borel measure on $[0,\infty)$,  satisfying the doubling condition
\begin{equation}\label{eq:doubling}
\sup_{t>0} \frac{\nu[0,2t)}{\nu[0,t)} < \infty.
\end{equation}
The Zen space $A^2_\nu$ is defined to consist of all analytic  functions $F$ on \[\C_+:=\set{z\in \C: \mathrm{Re}(z)>0}\]
such that 
\[
\norm{F}^2=\sup_{t>0} \int_{\overline{\C_+}} |F(x+iy+t)|^2 \, \dd \nu(x) \, \dd y< \infty.
\]
Examples of Zen spaces include Hardy spaces
$H^p(\C_+)$ (when $\nu=\frac{\delta_0}{2\pi}$) and weighted Bergman
spaces $\mathcal{B}^p_\alpha(\C_ +)$ (take $\alpha>-1$ and $\dd \nu(x)=\frac{x^\alpha}{\pi}\dd x$). We refer the reader to \cites{JPP, Harp1, Harp2} and the references therein.

A feature of Zen spaces is that the Laplace transform
defines an isometric map  from weighted
$L^2$ spaces on $(0,\infty)$ into certain spaces of analytic functions. More specifically we have (see e.g. \cite{JPP}*{Proposition 2.3}):
\begin{thm}[Paley–Wiener theorem for Zen spaces]\label{thm:Paley_Wiener}
Suppose that $w$ is given as a weighted Laplace transform
\[
w(t)=2\pi \int_0^\infty e^{-2rt} \, d\nu(r) \quad (t >0).
\]
Then the Laplace transform provides an isometric map 
\[
\mathcal{L}:  L^2((0,\infty),   w(t) \, \dd t ) \to A^2_\nu,
\]
where 
\[
\mathcal{L}(g)(z)=\int_0^\infty e^{-zt} g(t) \, \dd t .
\]
\end{thm}

\begin{define}
Set 
\[
\mathcal{H}^{s,r}(\R^d):=\set{f\in \mathcal{S}': (1-\Delta)^{\frac{s}{2}}\log^{r}(e-\Delta) f\in L^2(\R^d)}=F^{s,r}_{2,2}(\R^d),
\]
and endow $\mathcal{H}^{s,r}(\R^d)$ with the norm 
\[
\norm{f}_{\mathcal{H}^{s,r}}=\norm{(1-\Delta)^{\frac{s}{2}}\log^{r}(e-\Delta) f}_{L^2}.
\]
We shall denote by  $\mathcal{H}_A^{s,r}(\R)$ those distributions in $\mathcal{H}^{s,r}(\R)$ whose Fourier transform in supported in $[0,\infty)$.
\end{define}

The main result of this section is the following one.
\begin{proposition}
\label{prop:logartimic_smoothness_embedd} Let $a>-1$ and $\gamma>0$. The Laplace--Fourier transform given by
\[
Tf:=\mathcal{L}(\widehat{f})(z)
\] 
provides an isometric map 
\[
T:\mathcal{H}_A^{-\frac{a+1}{2},-\frac{\gamma}{2}}(\R) \to A^{2}_{\nu_{a,\gamma}}(\C_+),
\]
where $\dd\nu_{a,\gamma}(z)=\pi k_{a,\gamma}(x/2)\, \dd x$ and
\begin{equation}\label{kfunction}
    k_{a,\gamma}(t):=e^{-te}\int_0^\infty \frac{t^{a+b}b^{\gamma-1}}{\Gamma(\gamma)\Gamma(a+b+1)} \, \dd b.
\end{equation}
\end{proposition}

The previous proposition, jointly with \eqref{J-H^{log}_emb_inh} with the choice $q=2$, implies the following result.
\begin{corollary}\label{cor:embedding_Analytic} Let $p \in (1/2,1]$  and $r>0$. Then the Laplace--Fourier transform $T$ is a bounded map 
\[
T: h^{\Psi_{r,p}}_A(\R)\to A^{2}_{\nu_{a,\gamma}}(\C_+),
\] 
where 
\[
a=p-1,\quad \gamma=\frac{2r}{p}, \quad \dd\nu_{a,\gamma}(z)=\pi k_{a,\gamma}(x/2) \, \dd x
\]
and $h^{\Psi_{r,p}}_A(\R)$ is defined as those distributions in $h^{\Psi_{r,p}}(\R)$ whose Fourier transform is supported in $[0,\infty)$.
\end{corollary}

\subsubsection*{Proof of Proposition \ref{prop:logartimic_smoothness_embedd}}
We start the proof by giving two technical results that identify the symbol of the Fourier multiplier defining the space $\mathcal{H}^{s,r}(\R^d)$ as the Laplace transform of a suitable doubling measure.  
\begin{lemma}\label{lem:Laplace_trans}
Let $a>-1, \gamma>0$. If $k_{a,\gamma}$ denotes the function given in \eqref{kfunction} then
\[
\mathcal{L}(k_{a,\gamma})(s)=\frac{1}{(e+s)^{a+1}\ln(e+s)^{\gamma} }.
\]
\end{lemma}
\begin{proof}
Notice that for $a>-1$
\[
(e+s)^{-(a+1)}\Gamma(a+1)=\int_0^\infty t^a e^{-te} e^{-ts} \, \dd t,
\]
so if we take 
\begin{equation}\label{eq:g}
g_a(t)=\frac{t^ae^{-te}}{\Gamma(a+1)},
\end{equation}
then 
\[
\mathcal{L}(g_a)(s)=(e+s)^{-(a+1)}.
\]
Also, we have that 
\begin{align*}
\frac{\Gamma(\gamma)}{\ln(e+s)^{\gamma}}=\int_0^\infty \frac{b^{\gamma-1}}{(e+s)^{b}} \, \dd b & =\int_0^\infty \left(\int_0^\infty \frac{t^{b-1}b^{\gamma-1}}{\Gamma(b)} \dd b\right) e^{-te} e^{-ts} \, \dd t\\
&=\Gamma(\gamma)\mathcal{L}(h_\gamma)(s),
\end{align*}
where 
\begin{equation}\label{eq:h}
h_\gamma(t):=\frac{e^{-te}}{\Gamma(\gamma)}\int_0^\infty \frac{t^{b-1} b^{\gamma-1}}{\Gamma(b)} \, \dd b.
\end{equation}
A change of the order of integration, a change of variables, and the relation between the Beta function and the Gamma function yield 
\begin{align*}
h_\gamma*g_a(t) 
& =\frac{e^{-te}}{\Gamma(a+1)\Gamma(\gamma)}\int_0^\infty \frac{b^{\gamma-1}}{\Gamma(b)}\int_0^t (t-s)^as^{b-1} \,  \dd s \,  \dd b\\
&=e^{-te}\int_0^\infty \frac{t^{a+b}b^{\gamma-1}}{\Gamma(\gamma)\Gamma(a+b+1)} \,  \dd b.
\end{align*}
The result follows since, by the properties of the Laplace transform, we have that
\[
\mathcal{L}(k_{a,\gamma})=\mathcal{L}(h_{\gamma})\mathcal{L}(g_a).
\]
\end{proof}

\begin{lemma}\label{lem:doubling} Let $a>-1$ and $\gamma>0$. Define 
\[
\Phi_{a,\gamma}(t):=\int_0^t k_{a,\gamma}(s) \, \dd s.
\]
Then 
\begin{equation}\label{doubling2}
    \sup_{t>0} \frac{\Phi_{a,\gamma}(2t)}{\Phi_{a,\gamma}(t)}<\infty.  
\end{equation}
\end{lemma}
\begin{proof} 
A change on the order of integration and a change of variables yield
\begin{align*}
\Phi_{a,\gamma}(t)& =\int_0^\infty \frac{b^{\gamma-1}}{\Gamma(\gamma)\Gamma(a+b+1)} \int_0^t s^{a+b} e^{-se} \,  \dd s \,  \dd b\\
&=\int_0^\infty \frac{e^{-(a+b+1)}b^{\gamma-1}}{\Gamma(\gamma)\Gamma(a+b+1)} \int_0^{et} s^{a+b} e^{-s} \,  \dd s \,  \dd b. 
\end{align*}

In particular it follows that $\Phi_{a,\gamma}$ is continuous on $[0,\infty)$, increasing and
\[
\sup_{t>0}\Phi_{a,\gamma}(t)=e^{-a-1}. 
\] 
Hence, for $t\geq 1/2$ we have that
\begin{equation}\label{eq:doubling_high}
1\leq\frac{\Phi_{a,\gamma}(2t)}{\Phi_{a,\gamma}(t)}\leq \frac{e^{-a-1}}{\Phi_{a,\gamma}(1/2)}.
\end{equation}
Then, it suffices to check the doubling property for $0<t<1/2$.
	
For $t<1$, a change of variables gives
\begin{align*}
\eta(t):=\int_0^\infty \frac{t^{a+b}b^{\gamma-1}}{\Gamma(\gamma)\Gamma(a+b+1)} \, \dd b & =t^{a}\abs{\ln t}^{-\gamma}\int_0^\infty \frac{e^{-u}u^{\gamma-1}}{\Gamma(\gamma)\Gamma(a+\frac{u}{\abs{\ln t}}+1)} \, \dd b \\
 & = t^{a}\abs{\ln t}^{-\gamma} \beta(t).
\end{align*}
Then, for $0<t<1/2$  
\[
\frac{\eta(2t)}{\eta(t)}=2^{a}\abs{\frac{\ln 2t}{\ln t}}^\gamma \frac{\beta(2t)}{\beta(t)}.
\]
It then follows that 
\[
\lim_{t\to 0^+}	\frac{\Phi_{a,\gamma}(2t)}{\Phi_{a,\gamma}(t)}=2^a.
\]
In particular, this yields that we can extend $\frac{\Phi_{a,\gamma}(2t)}{\Phi_{a,\gamma}(t)}$ as a continuous function on $[0,1/2]$. By compactness, it has a maximum on that interval, which jointly with \eqref{eq:doubling_high} yields \eqref{doubling2}.
\end{proof}

Notice first that Lemma \ref{lem:doubling} guarantees the doubling property \eqref{eq:doubling}. By the definition of the norm on $\mathcal{H}^{s,r}_A$, Lemma \ref{lem:Laplace_trans} and the Paley--Wiener theorem \ref{thm:Paley_Wiener} we have
\[
\begin{split}
\norm{f}_{\mathcal{H}_A^{-\frac{a+1}{2},-\gamma/2}(\R)}^2 &=\int_{0}^\infty \frac{\abs{\widehat{f}(s)}^2}{(s+e)^{a+1}\log^\gamma (e+s)} \,  \dd s\\
&=\int_{0}^\infty {\abs{\widehat{f}(s)}^2}\mathcal{L}(k_{a,\gamma})(s) \,  \dd s=\norm{\mathcal{L}(\widehat{f})}_{A^2_{\nu_{a,\gamma}}(\C_+)}^2,
\end{split}
\]
where 
\[
\dd\nu_{a,\gamma}(z)=\pi k_{a,\gamma}(x/2) \,  \dd x,
\]
finishing the proof of Proposition \ref{prop:logartimic_smoothness_embedd}.

\begin{rmk} In the case that $r=0$, with a similar argument as above one shows that 
\[
T: h^p_A(\R)\to A_{\nu_a}^2(\C_+)
\]
with
\[
\dd\nu_{a}(z)=\pi g_{a}(x/2) \,  \dd x\, \dd y=\frac{\pi x^a}{2^a\Gamma(a+1)} e^{-xe/2} \, \dd x \,  \dd y.
\] 
\end{rmk}
\begin{rmk} Note that for $a>-1$ 
\[
\mathcal{L}(x^a)(t)=\Gamma(a+1) t^{-a-1}.
\]
So defining for $\gamma>0$ and $a>-1$
\[
n_{a,\gamma}(t):=\frac{1}{\Gamma(a+1)}\int_0^t (t-s)^a h_{\gamma}(s) \, \dd s,
\]
with $h_\gamma$ as in \eqref{eq:h}, then 
\[
\mathcal{L}(n_{a,\gamma})(s)=\frac{1}{s^{a+1}\log^\gamma(e+s)}.
\]
Tracing the argument above, one shows that for $p \in (1/2, 1]$ and $r>0$,  the Laplace--Fourier transform $T$ is a bounded map 
\[
T: H^{\Psi_{r,p}}_A(\R)\to A^{2}_{\nu_{a,\gamma}}(\C_+),
\] 
where 
\[
a=p-1,\quad \gamma=\frac{2r}{p}, \quad \dd\nu_{a,\gamma}(z)=\pi n_{a,\gamma}(x/2) \, \dd x
\]
and $H^{\Psi_{r,p}}_A(\R)$ is defined  as those distributions in $H^{\Psi_{r,p}}(\R)$ whose Fourier transform is supported in $[0,\infty)$.

One can rephrase this boundedness result in terms of a Sobolev-like embedding involving a Bergman space. For instance, one obtains that for $p \in (1/2, 1]$ and $r\geq 0$
\[
\begin{split}
\norm{Tf}_{B^2_{2(\frac{1}{p}-1)}(\C_+)}^2\lesssim  \norm{\log \brkt{e{{-\Delta}}}^{r/p}f}_{H^{\Psi_{r,p}}(\R)}^2,
\end{split}
\]
where $B^2_a(\C_+)=A_{\nu_a}^2(\C_+)$ with 
\[
\dd \nu_a(z)=\frac{\pi x^a}{2^a\Gamma(a+1)} \,  \dd x \, \dd y.
\]
\end{rmk}

\section*{Acknowledgements}
O.B. would like to thank I. Parissis for some useful discussions on Rubio de Francia-type square functions.

A.S. would like to thank J. Cima and M. Mitkovski for useful comments on the results in this paper.

\end{document}